\documentclass[a4paper,10pt]{amsart}
\usepackage[utf8]{inputenc}
\usepackage{amssymb,amsmath}
\usepackage[all]{xy}

\newtheorem{theorem}{Theorem}[section]
\newtheorem{definition}[theorem]{Definition}
\newtheorem{example}[theorem]{Example}
\newtheorem{proposition}[theorem]{Proposition}
\newtheorem{lemma}[theorem]{Lemma}
\newtheorem{remark}[theorem]{Remark}

\newtheorem{corollary}[theorem]{Corollary}

\renewcommand{\AA}{\mathcal{A}}
\newcommand{\Aff}{\mathbb{A}}
\newcommand{\BB}{\mathcal{B}}
\newcommand{\CC}{\mathbb{C}}

\newcommand{\EE}{\mathcal{E}}
\newcommand{\FF}{\mathcal{F}}

\newcommand{\kk}{\Bbbk}
\newcommand{\KK}{\mathbb{K}}
\newcommand{\LL}{\mathbb{L}}
\newcommand{\NN}{\mathbb{N}}
\newcommand{\PP}{\mathbb{P}}
\newcommand{\QQ}{\mathbb{Q}}
\newcommand{\TT}{\mathbb{T}}
\newcommand{\Ta}{T}
\newcommand{\ZZ}{\mathbb{Z}}
\newcommand{\Mst}{\mathfrak{M}}
\newcommand{\Msp}{\mathcal{M}}

\newcommand{\unit}{1}
\newcommand{\ICS}{\mathcal{IC}}
\newcommand{\DTS}{\mathcal{DT}}
\newcommand{\Part}{\mathcal{P}}
\newcommand{\OO}{\mathcal{O}}
\newcommand{\Lin}{\mathcal{L}}
\newcommand{\Xst}{\mathfrak{X}}

\DeclareMathOperator{\Hom}{Hom}
\DeclareMathOperator{\End}{End}
\DeclareMathOperator{\Ext}{Ext}
\DeclareMathOperator{\dgV}{dg-Vect}
\DeclareMathOperator{\Vect}{Vect}
\DeclareMathOperator{\rep}{Rep}
\DeclareMathOperator{\Ka}{K}
\DeclareMathOperator{\Aut}{Aut}
\DeclareMathOperator{\codim}{codim}

\DeclareMathOperator{\MHM}{MHM}
\DeclareMathOperator{\rat}{rat}
\DeclareMathOperator{\IC}{IC}
\DeclareMathOperator{\DT}{DT}
\DeclareMathOperator{\Sym}{Sym}
\DeclareMathOperator{\Alt}{Alt}
\DeclareMathOperator{\Var}{Var}
\DeclareMathOperator{\Spec}{Spec}
\DeclareMathOperator{\Gl}{GL}

\DeclareMathOperator{\Ho}{H}
\DeclareMathOperator{\rk}{rk}
\DeclareMathOperator{\im}{im}
\DeclareMathOperator{\pr}{pr}
\DeclareMathOperator{\id}{id}
\DeclareMathOperator{\Isom}{\mathcal{I}som}
\DeclareMathOperator{\Lie}{Lie}
\DeclareMathOperator{\Stab}{Stab}
\DeclareMathOperator{\Proj}{Proj}

\DeclareMathOperator{\ext}{ext}
\DeclareMathOperator{\cl}{cl}
\DeclareMathOperator{\Specbf}{\bf Spec}
\DeclareMathOperator{\Gr}{Gr}
\DeclareMathOperator{\rad}{rad}
\DeclareMathOperator{\coker}{coker}
\DeclareMathOperator{\Iso}{Iso}
\DeclareMathOperator{\Coh}{Coh}
\DeclareMathOperator{\Char}{char}
\DeclareMathOperator{\gr}{gr}
\DeclareMathOperator{\HHom}{\mathcal{H}om}
\DeclareMathOperator{\Quot}{Quot}

\title[DT invariants for categories of homological dimension one]{Donaldson--Thomas invariants vs.\ intersection cohomology for categories of homological dimension one}
\author{Sven Meinhardt}

\begin{document}

\begin{abstract}
The present paper is an extension of a previous paper written in collaboration with Markus Reineke dealing with quiver representations. The  aim of the paper is to generalize the theory and to provide a comprehensive theory of Donaldson--Thomas invariants for abelian categories of homological dimension one (without potential) satisfying some technical conditions. The theory will apply for instance to representations of quivers, coherent sheaves on smooth projective curves, and some coherent sheaves on smooth projective surfaces. We show that the (motivic) Donaldson--Thomas invariants satisfy the Integrality conjecture and identify the Hodge theoretic version with the (compactly supported) intersection cohomology of the corresponding moduli spaces of objects. In fact, we deal with a refined  version of Donaldson--Thomas invariants which can be interpreted as classes in the Grothendieck group of some ``sheaf'' on the moduli space. In particular, we reproduce the intersection complex of 
moduli spaces using Donaldson--Thomas theory. 
\end{abstract}

\maketitle

\tableofcontents

\section{Introduction}

The theory of Donaldson--Thomas invariants  started around 2000 with the seminal work of R.\ Thomas \cite{Thomas1}; associating numerical invariants, that is, numbers, to moduli spaces in the absence of strictly semistable objects. Six years later D. Joyce \cite{JoyceI},\cite{JoyceCF},\cite{JoyceII},\cite{JoyceIII},\cite{JoyceMF},\cite{JoyceIV} and Y.\ Song \cite{JoyceDT} extended the theory, producing numbers even in the presence of semistable objects which is the generic situation. Around the same time, M.\ Kontsevich and Y.\ Soibelman \cite{KS1},\cite{KS2},\cite{KS3} independently proposed a theory producing motives instead of simple numbers, also in the presence of semistable objects. The technical difficulties occurring in their approach disappear in the special situation of representations of quivers (with zero potential). This case has been  intensively studied by Markus Reineke in a series of papers \cite{Reineke2},\cite{Reineke3},\cite{Reineke4}. \\
Despite some computations of motivic or even numerical Donaldson--Thomas invariants for quivers with or without potential (see \cite{BBS},\cite{DaMe1},\cite{DaMe2},\cite{MMNS}), the true nature of Donaldson--Thomas invariants still remains mysterious.\\[1ex]
This paper is a second step to disclose the secret by showing that the Donaldson--Thomas invariants for abelian categories of homological one satisfying some technical condition compute the compactly supported intersection cohomology of the closure of the simple locus inside the associated coarse moduli space of (semisimple) objects. The first step has already been done by the author in collaboration with Markus Reineke in \cite{MeinhardtReineke} which contains all the results presented here in the case of quiver without potential.  \\ 
We will actually prove an even stronger version by defining a Donaldson--Thomas ``sheaf'' on the coarse moduli space $\Msp$. Strictly speaking, this sheaf is not a (perverse) sheaf but a class in a suitably extended Grothendieck group of mixed Hodge modules. The cohomology with compact support of that sheaf is the usual Hodge theoretic Donaldson--Thomas invariant - a class in the Grothendieck group of mixed Hodge structures. Our main result is the following (we refer to the following sections for precise notation):
\begin{theorem}
 Let $\AA_\kk$ be an abelian $\kk$-linear category satisfying some technical condition. Then the Donaldson--Thomas sheaf and the intersection complex of the closure of the simple locus $\Msp^{s}$ inside the coarse moduli space $\Msp$ agree in the Grothendieck group of mixed Hodge modules. In particular, by taking cohomology with compact support, we obtain for every ``dimension vector'' $d$ 
 \[ \DT_d=\begin{cases} \IC_c(\Msp_d,\QQ)= \IC(\Msp_d,\QQ)^\vee &\mbox{ if } \Msp^{s}_d\neq \emptyset, \\
          0 & \mbox{ otherwise}
         \end{cases} \]
in the Grothendieck ring of (polarizable) mixed Hodge structures.
\end{theorem}
The conditions mentioned in the Theorem are the following:
\begin{enumerate}
\item[(1)] Existence of a moduli theory $\AA$,
\item[(2)] Existence of a good moduli stack $\Mst$,
\item[(3)] Existence of a fiber functor $\omega:\AA\to \Vect^I$,
\item[(4)] Existence of good GIT-quotients $\Msp$,
\item[(5)] Representability and properness of the universal Grassmannian,
\item[(6)] Existence of a good deformation theory,
\item[(7)] The number $(E,F):=\dim_\KK \Hom_{\AA_\KK}(E,F)-\dim_\KK\Ext^1_{\AA_\KK}(E,F)$ is locally constant on $\Mst\times\Mst$,
\item[(8)] The pairing $(-,-)$ is symmetric.
\end{enumerate}
The first six properties ensure that the objects, we are talking about, exist and have the ``usual'' properties\footnote{One can show that Representability in (5) is a consequence of (6). Properness of the Grassmannian is not used in this paper but will become important in future projects.}. Coherent sheaves on smooth projective varieties or representations of quivers with relation fall in the class of examples. Condition (7) is our form of saying that $\AA_\KK$ is of homological dimension one for every field extension $\KK\supset \kk$. Note that $\AA_\KK$ might not have enough projectives or injectives and the definition of higher $\Ext$-groups needs some care which we bypass by  assumption (7). Property (8) is the most important one in Donaldson--Thomas theory and cannot be overestimated. For the usual categories of homological dimension one, (8) boils down to a genericity assumption on stability conditions. \\
It is good to keep in mind that $\AA_\kk$ is usually  not the category of all objects of interest but only the subcategory of semistable objects of a fixed ``slope''. The simple objects are the stable ones and the semisimple objects correspond the polystable ones. It is this restriction which guaranties properties (1)--(8) even if they are not satisfied for the ``big'' category of all objects. An even more special situation occurs for surfaces of Kodaira dimension $-\infty$. Here, $\AA_\kk$ can be exhausted by full abelian subcategories  which satisfy our assumptions even though $\AA_\kk$ might not.\\
A technical tool which turns out the be of central interest is the moduli space $\Msp_{f,d}$ of semistable framed objects. Framed objects are objects $E$ in $\AA_\kk$ together with a bunch of vectors in some $I$-graded finite dimensional vector space $\omega_\kk(E)$ of graded dimension $d\in \NN^{\oplus I}$ associated to $E$. Such an object is called semistable, if the vectors are not contained in the subvector space $\omega_\kk(E')$ for some subobject $E'\subset E$. By forgetting the vectors, we get the so-called Hilbert--Chow morphism $\pi_{f,d}:\Msp_{f,d}\longrightarrow \Msp_d$ which has the following remarkable property.  
\begin{theorem} Under the assumptions (1)--(8), the morphism $\pi_{f,d}:\Msp_{f,d} \longrightarrow \Msp_d$ forgetting the framing is projective and virtually small, that is, there is a finite stratification $\Msp_d=\sqcup_\xi S_\xi$ with empty or dense stratum $S_0=\Msp^{s}_d$ such that $\pi_{f,d}^{-1}(S_\xi) \longrightarrow S_\xi$ is \'etale locally trivial and 
 \[ \dim \pi_{f,d}^{-1}(x_\xi) - \dim \PP^{f\cdot d-1} \le \frac{1}{2} \codim S_\xi\] 
for every $x_\xi\in S_\xi$ with equality only for $S_\xi=S_0\not=\emptyset$ with fiber $\pi_{f,d}^{-1}(x_0)\cong \PP^{f\cdot d-1}$.
\end{theorem} 
Notice that Donaldson--Thomas invariants can be computed quite efficiently. 
Thus, our  theorem provides a quick algorithm to determine intersection Hodge numbers. There is already an algorithm  to compute intersection numbers going  back to extensive work of F.\ Kirwan around 1985 (see \cite{Kirwan1},\cite{Kirwan4},\cite{Kirwan2},\cite{Kirwan3}). However, writing $\Msp_d$ as a GIT-quotient $X_d/\!\!/G_d$, this algorithm is very complicated as  one has to understand to action of $G_d$ on $X_d$ in a very precise way. In Donaldson--Thomas theory the precise action can be dropped, and one only needs to know the intersection cohomology of $X_d$. Moreover, using wall-crossing formulas, we are now able to understand the change of intersection Hodge numbers under variations of stability conditions. \\
Let us give a couple of applications of the results proven here. 
\begin{corollary}[Positivity]
 For representations of quivers the (motivic) Donaldson--Thomas invariant is a polynomial in the Lefschetz motive with positive coefficients. 
\end{corollary}
Indeed, the coefficients are the dimensions of the intersection cohomology groups and the latter carry a pure Hodge structure.\\
Another corollary is obtained by using the fact that the moduli space of semistable quiver representations admits a proper map to the affine, connected moduli space of semisimple representations of the same dimension vector. If the quiver is acyclic, there is only one, thus the moduli space must be compact and we can apply the Hard Lefschetz theorem to intersection cohomology. 
\begin{corollary}[Unimodularity]
 If $Q$ is acyclic, the Donaldson--Thomas invariant for a generic stability condition is a unimodular polynomial in the Lefschetz motive. The unimodularity remains true for the Hodge polynomial of the  Donaldson--Thomas invariants for Gieseker semistable sheaves on smooth compact curves or smooth projective surfaces of Kodaira dimension $-\infty$.
\end{corollary}
The next result is a direct consequence of our main theorem and Proposition \ref{localDT} and Corollary \ref{localDT2}.
\begin{corollary}[Locality]
 Fix a  point  in the moduli space of objects in $\AA_\kk$, that is, a semisimple object $E=\bigoplus_{k\in K} E_k^{m_k}$ with simple pairwise non-isomorphic summands $E_k$. If the component of the moduli space containing $E$ also contains stable representations, then the fiber at $E$ of the intersection complex of the moduli space $\Msp$ is given by the intersection cohomology of a moduli space associated to the $\Ext^1$-quiver of the collection $(E_k)_{k\in K}$.  
\end{corollary}
The paper is organized as follows. Section 2 provides some background on quivers an their representations. The main purpose is to fix notation used later and to provide an easy example of the theory. \\
\\
Section 3 generalizes section 2 to more general abelian categories of homological dimension one. We discuss all the assumptions (1)--(8) in detail and proof some results which basically say that moduli spaces and stacks behave as they should do. The most important result of section 3 is Theorem \ref{virtsmall},  stating that the so-called Hilbert--Chow morphism is virtually  small. Parts of the proof of this important technical result are postponed to section 7.\\
\\
Section 4 is devoted to intersection complexes and the Schur functor formalism. As we need the notion of a weight filtration, restricting to perverse sheaves is not sufficient. Hence, we have to consider mixed Hodge modules, but there is no reason to be worried about that. We only need that the Grothendieck group is freely generated (as a group) by some sort of intersection complexes, and also the Decomposition Theorem of Beilinson, Bernstein, Deligne, Gabber and Saito will be used sometimes. \\
Taking direct sums of representations induces a symmetric monoidal tensor product on the category of mixed Hodge modules by convolution. Using some general machinery (see \cite{Deligne1}), one can introduce Schur (endo)functors. Among them the symmetric and alternating powers are the most famous ones, and we finally obtain a $\lambda$-ring structure on the Grothendieck group of mixed Hodge structures. \\

The $\lambda$-ring structure is used in Section 5 to define the Donaldson--Thomas ``sheaves''.  Using the virtual smallness of the Hilbert--Chow morphism and the Decomposition Theorem, we finally deliver the proof of our main theorem by comparing degrees of the weight filtration.\\ 

While proving our main result in section 5, we will observe that a certain finiteness condition is crucial. It turns out that this condition is a sheaf version of the famous integrality conjecture in Donaldson--Thomas theory. We provide a proof  of this ``sheafified integrality conjecture'' reducing the problem to a result of Efimov (see \cite{Efimov}, Theorem 1.1).  Here is the main result of section 6.

\begin{theorem}[Integrality Conjecture, sheaf version]  
For an abelian category $\AA_\kk$ satisfying the conditions (1)--(8) the motivic Donaldson--Thomas sheaf $\DTS^{mot}$ is in the image of the natural map 
\[ \hat{\Ka}(\Var/\Msp)[\LL^{-1/2}] \longrightarrow \hat{\Ka}(\Var/\Msp)[\LL^{-1/2}, (\LL^n-1)^{-1}\, :\, n\in \NN].\] 
\end{theorem}

By ``integrating'' over the moduli space of objects of a given dimension vector $d$, we obtain a proof of the famous integrality conjecture.
\begin{corollary}[Integrality Conjecture] 
For $\AA_\kk$ as above the motivic Donaldson--Thomas invariant $\DT^{mot}_d$ is in the image of the natural map 
\[ \hat{\Ka}(\Var/\kk)[\LL^{-1/2}] \longrightarrow \hat{\Ka}(\Var/\kk)[\LL^{-1/2}, (\LL^n-1)^{-1}\, :\, n\in \NN].\]
\end{corollary}
This result has been obtained by Efimov for representations of symmetric quivers  and trivial stability condition (see  \cite{Efimov}, Theorem 1.1). A very complicated proof of the integrality conjecture even for quivers with potential was given by Kontsevich and Soibelman (see \cite{KS2}, Theorem 10).\\
\\
Let me finally give some outline of other papers occurring in the near future. In \cite{DavisonMeinhardt3} the authors use the results proven here to give a comprehensive treatment of Donaldson--Thomas theory for categories $\AA_\kk$ as above satisfying assumption (1)--(8) which are equipped with a  ``potential'' or even more general structures. In another paper \cite{DavisonMeinhardt4} a categorification of all these results is provided in the context of mixed Hodge modules.
\\

\textbf{Acknowledgments.} The main result of the paper was originally observed and conjectured by J.\ Manschot while doing some computations. The author is very grateful to him for sharing his observations and his conjecture which was the starting point of this paper. The author would also like to thank Jörg Schürmann for answering patiently all  questions about mixed Hodge modules.
Special thanks goes to Markus Reineke for providing a wonderful atmosphere to complete this research. Many of the results proven here took there origin in a previous version \cite{MeinhardtReineke} for quivers only written in collaboration with Markus. His beautiful ideas still provide the background of this expanded version.

\section{Moduli spaces of quiver representations}

\subsection{Quiver representations}

We fix a field $\KK$ which might either be our ground field $\kk$ or, as in section 6, a not necessarily algebraic extension of the latter. Let $Q=(Q_0,Q_1,s,t)$ be a quiver consisting of a not necessarily finite set $Q_0$ of vertices and a  set $Q_1$ of arrows with $s,t:Q_1\to Q_0$ being the source and target maps. We assume that for every pair $(i,j)\in Q_0^2$ the set $Q(i,j)$ of arrows from $i$ to $j$ is finite. To any quiver we associate its $\KK$-linear path category $\KK Q$ with set of objects $Q_0$ and $\Hom_{\KK Q}(i,j)$ being the $\KK$-vector space generated by all paths from $i$ to $j$. Composition is induced by $\KK$-linear extension of concatenation of paths. \\
There is a second (dg-)category associated to $Q$, namely its Ginzburg category $\Gamma_\KK Q$. The underlying $\KK$-linear category  is the path category $\KK Q^{ex}$ associated to the extended quiver $Q^{ex}=(Q_0, Q_1\sqcup Q_1^{op} \sqcup Q_0,s^{ex},t^{ex})$ obtained from $Q$ by adding to every arrow $\alpha:i\to j$ of $Q$ another arrow $\alpha^\ast:j \to i$ with opposite orientation, and a loop $l_i:i\to i$ for every vertex $i\in Q_0$. We make $\Gamma_\KK Q$ into a dg-category by introducing a grading such that $\deg(\alpha)=0, \deg(\alpha^\ast)=-1$, and $\deg(l_i)=-2$. The differential is uniquely determined by putting 
\[ d\alpha=d\alpha^\ast=0 \quad\mbox{ and }\quad dl_i=\sum_{\alpha:i \to j} \alpha^\ast \alpha - \sum_{\alpha:j \to i} \alpha \alpha^\ast. \]      
Apparently, $H^0(\Gamma_\KK Q)\cong \KK Q$ can be interpreted as a dg-category with zero grading and trivial differential. \\
By looking at dg-functors $V:\KK Q \longrightarrow \dgV_\KK$ and $W: \Gamma_\KK Q  \longrightarrow \dgV_\KK$ into the category of dg-vector spaces with finite dimensional total cohomology, we get two dg-categories with model structures and associated triangulated homotopy ($A_\infty$-)categories $D^b(\KK Q-\rep)$ and $D^b(\Gamma_\KK Q - \rep)$. Each has a bounded t-structure with heart $\KK Q-\rep$ being the abelian category of quiver representations, that is, of functors $V:\KK Q \longrightarrow \Vect_\KK$ into the category of finite dimensional $\KK$-vector spaces. In particular, 
\[ \Ka_0( D^b(\KK Q - \rep)) \cong \Ka_0 (D^b(\Gamma_\KK Q -  \rep)) \cong \Ka_0 (\KK Q-\rep). \]
There is a group homomorphism $\dim: \Ka_0(\KK Q -\rep)\longrightarrow \ZZ^{\oplus Q_0}$ associating to every representation $V$ the tuple $(\dim_\KK V_i)_{i\in Q_0}\in \NN^{\oplus Q_0}$ of dimensions of the vector spaces $V_i:=V(i)$. There are two pairings on $\ZZ^{\oplus Q_0}$ defined by \begin{eqnarray*}
(d,e)& := & \sum_{i\in Q_0} d_ie_i \; - \sum_{Q_1\ni \alpha:i\to j} d_ie_j \\
\langle d,e \rangle &:=& (d,e)\;-\;(e,d) 
\end{eqnarray*}
such that the pull-back of these pairings via $\dim$ is just the Euler pairing induced by $D^b(\KK Q-\rep)$ resp.\ $D^b(\Gamma_\KK Q-\rep)$. The skew-symmetry of the latter reflects the fact that $D^b(\Gamma_\KK Q-\rep)$ is a 3-Calabi--Yau category, that is, the triple shift functor $[3]$ is a Serre functor. 

\subsection{Moduli spaces}
The stack of $Q$-representations, that is, of objects in $\KK Q-\rep$, can be described quite easily. For this, fix a dimension vector $d=(d_i)\in \NN^{\oplus Q_0}$ and let $G_d:=\prod_{i\in Q_0} \Aut(\KK^{d_i})$ act on $R_d:=\prod_{\alpha:i\to j} \Hom(\KK^{d_i},\KK^{d_j})$ by simultaneous conjugation. The stack of $\KK Q$-representations of dimension $d$ is just the quotient stack $\Mst_d=R_d/G_d$.  \\
We want to study semistable representations of $Q$. As the radical of the Euler pairing contains the kernel of $\dim:\Ka_0(\KK Q-\rep)\longrightarrow \ZZ^{\oplus Q_0}$, every tuple $\zeta=(\zeta_i)_{i\in Q_0}\in \{r\exp(i\pi\phi)\in \mathbb{C}\mid r>0, 0<\phi\le 1\}^{Q_0}\subset \mathbb{C}^{Q_0}$ provides a numerical Bridgeland stability condition on $D^b(\KK Q-\rep)$ and on $D^b(\Gamma_\KK Q -\rep)$ with central charge $Z(V)=\zeta\cdot \dim V:=\sum_{i\in Q_0}\zeta_i\dim_\KK V_i$ of slope $\mu(V):=- \Re e Z(V)/ \Im m Z(V)$ and standard t-structure. Hence we get an open substack $\Mst^{\zeta-ss}_d=R^{\zeta-ss}_d/G_d$ of semistable $\KK Q$-representations.  For $\mu\in (-\infty,+\infty]$ we call $\zeta$ $\mu$-generic if $\langle d,e\rangle =0$ for all $d,e\in \NN^{\oplus Q_0}$ of slope $\mu$, and generic if that holds for all $\mu$. For finite $Q_0$ the non-generic ``stability conditions'' $\zeta$ lie on a countable but locally finite union of  real codimension one walls in the complex manifold $\{r\exp(i\pi\
phi)\in \mathbb{C}\mid r>0, 0<\phi\le 1\}^{Q_0}$. Obviously every stability for a symmetric quiver is generic. Another important class is given by complete bipartite quivers and the maximally symmetric stabilities used in \cite{RW} to construct a correspondence between the cohomology of quiver moduli and the GW invariants of \cite{GPS}.\\    
As we wish to form moduli schemes, we should restrict ourselves to King stability conditions $\zeta=(-\theta_i + \sqrt{-1})_{i\in Q_0}$ for some $\theta=(\theta_i)\in \ZZ^{Q_0}$, giving rise to a linearization of the $G_d$-action on $R_d$ with semistable points $R^{\zeta-ss}_d$. Let us denote the GIT-quotient by $\Msp^{\zeta-ss}_d= R^{\zeta-ss}_d/\!\!/ G_d$. The points in $\Msp^{\zeta-ss}_d$ correspond to polystable representations $V=\oplus_{\kappa} V_\kappa$, and the obvious morphism $p:\Mst^{\zeta-ss}_d \longrightarrow \Msp^{\zeta-ss}_d$ maps a semistable representation to the direct sum of its stable factors. We also have the open substack $\Mst^{\zeta-st}_d\subset \Mst^{\zeta-ss}_d$ of stable representations mapping to the open subscheme $\Msp^{\zeta-st}_d\subset \Msp^{\zeta-ss}_d$ of stable representations. 
Note that $\Mst_d, \Mst^{\zeta-ss}_d, \Mst^{\zeta-st}_d,$ and $\Msp^{\zeta-st}_d$ are smooth while $\Msp^{\zeta-ss}_d$ is not. Moreover, $\Msp^{\zeta-st}_d$ is either dense in $\Msp^{\zeta-ss}_d$ or empty. We call $\theta$ ($\mu$-)generic if $\zeta=(-\theta_i + \sqrt{-1})_{i\in Q_0}$ is ($\mu$-)generic in the previous sense. \\

For later applications we also need framed $Q$-representations (see \cite{Reineke1}). For this we fix a framing vector $f\in \NN^{Q_0}$ and consider representations of a new quiver $Q_f=(Q_0\sqcup\{\infty\}, Q_1\sqcup \{\beta_{l_i}:\infty \to i \mid i\in Q_0, 1\le l_i\le f_i \})$ with dimension vector $d'$ obtained by extending $d$ via $d_\infty=1$. We also extend $\theta$ appropriately (see \cite{Reineke1}) and get a King stability condition $\theta'$ for $Q_f$. Let $\Msp^{\zeta-ss}_{f,d}$ be the moduli space of $\theta'$-semistable $Q_f$-representations of dimension vector $d'$. It turns out (cf.\ Proposition \ref{framed_reps_2}) that $\Msp^{\zeta-ss}_{f,d}=\Msp^{\zeta-st}_{f,d}$, and thus $\Msp^{\zeta-ss}_{f,d}$ is smooth. There is an obvious morphism $\pi:\Msp^{\zeta-ss}_{f,d} \longrightarrow \Msp^{\zeta-ss}_d$, obtained by restricting a $\theta'$-(semi)stable representation of $Q_f$ to the subquiver $Q$ which turns out to be $\theta$-semistable. \\

A quiver with relations is a quiver $Q$ and a collection $r(i,j)\subset \Hom_{\KK Q}(i,j)$ of finite dimensional subspaces for $(i,j)\in Q_0^2$. Note that $\Hom_{\KK Q}(i,j)$ does not need to be finite dimensional unless $Q_0$ is finite. We say that a representation $V:\KK Q\to \Vect_\KK$ satisfies the relations if $V(f)=0$ for all $f\in r(i,j)$ and all pairs $(i,j)\in Q_0^2$. The substack of representations of $Q$ satisfying $r$ is closed in $\Mst$ and can be described as a quotient $R_d^r/G_d$ for a closed subscheme $R^r_d\subset R_d$. 

\section{Moduli of objects in  abelian categories}

\subsection{Moduli functors of objects in abelian categories} 

As we want to talk about moduli stacks and spaces, we need some  moduli functor $\Mst$  which attaches to every $\kk$-scheme $S$ the groupoid $\Mst(S)$ of ``families of objects'' in some abelian category $\AA_\kk$ together with isomorphisms between them. For later applications it is useful not only to consider isomorphisms  but also general morphisms between families of objects. Therefore, we make the following first assumption.\\
\\
\textbf{(1) Existence of a moduli theory:} There is a contravariant (pseudo)functor $\AA:S\mapsto \AA_S$ from the category of $\kk$-schemes $S$ to the category of essentially small exact categories  using the shorthand $\AA_R:=\AA_{\Spec R}$ for any $\kk$-algebra $R$, satisfying the usual axioms of a stack. Moreover, for every $\kk$-scheme $S$ and every pair $E_1,E_2\in\AA_S$ the groups $\Hom_{\AA_S}(E_1,E_2)$ and $\Ext^1_{\AA_S}(E_1,E_2)$ should carry the structure of a finitely generated $\OO_S(S)$-module such that composition is $\OO_S$-bilinear. We assume that pull-backs and push-outs of short exact sequences exist in $\AA_S$, and the action of $f\in \OO_S(S)$  on $\Ext^1_{\AA_S}(E_1,E_2)$ coincides with the pull-back along  $f\in \End_{\AA_S}(E_1)$ or the push-out of a short exact sequence along $f\in \End_{\AA_S}(E_2)$. If $\KK\supset \kk$ is a field extension, $\AA_\KK$ should be abelian. \\

\begin{example}\rm
Let $(Q,r)$ be a quiver with  relations and let $\AA_S$ be the category of quiver representations on vector bundles on $S$, i.e.\ the category of functors \\ $V:\kk Q \longrightarrow \Vect_S$ vanishing on $r$.  
\end{example}
\begin{example} \rm
Let $X$ be a smooth projective variety and let $\AA_S$ be the category of coherent sheaves on $S\times X$ flat over $S$. 
\end{example}

We denote with $\Mst$ the subfunctor of $\AA$ mapping $S$ to the groupoid $\Isom(\AA_S)$ of isomorphisms in $\AA_S$. It is also a stack. \\
\\
\textbf{(2) Existence of good moduli stacks:} We require that $\Mst$ is isomorphic to $\sqcup_{d\in \NN^{\oplus I}} X_d/G_d$ for some algebraic space  $X_d$ with $G_d=\prod_{i\in I}\Gl(d_i)$. For $E\in \Mst_d$ we use the notation $\dim E=d$ and call $\dim E$ the dimension (vector) of $E$.\\

Let us explain the last condition. Since $\sqcup_{d\in \NN^{\oplus I}}\Spec\kk/G_d$ is the stack of $I$-graded vector spaces of total finite dimension and $X_d/G_d \longrightarrow  \Spec\kk/G_d$ is representable, we obtain a faithful functor $\omega_S:\Isom(\AA_S)\longrightarrow\Isom(\Vect_S^I)$ from the isomorphism groupoid of $\AA_S$ to the isomorphism groupoid of the category of $I$-graded vector bundles on $S$ of finite total rank. Moreover, $S\mapsto\omega_S$ is compatible with pull-backs. Conversely, every such collection of  functors $\omega_S$ compatible with pull-backs provides a  morphism $\omega:\Mst \to \sqcup_{d\in \NN^{\oplus I}}\Spec\kk/G_d$ of stacks, and we just require that $\omega$ is representable. 
By definition of the fiber product, the $S$-points of $X_d$ are given by equivalence classes of pairs $(E,\psi)$ with $E\in \AA_S$ and $\psi:\omega_S(E)\xrightarrow{\sim}\OO^d_S$ an isomorphism of $I$-graded vector bundles on $S$ with $\OO^d_S:=\bigoplus_{i\in I}\OO_S^{\oplus d_i}$. Two pairs $(E,\phi)$ and $(E',\psi')$ are equivalent if there is a (unique) isomorphism $\phi:E\xrightarrow{\sim}E'$ with $\psi'\omega_S(\phi)=\psi$. 
\begin{example} \rm
 The assumption (2) is fulfilled for representation of quivers with or without relations. For sheaves on smooth projective varieties however, this condition is not fulfilled as the stack is in general only a nested union of quotient stacks. This is, where stability conditions come into play. Fixing a suitable stability condition, one ends up with an open substack of semistable objects which might satisfy this assumption. This is for instance the case for Gieseker semistable sheaves and from now on we assume that $\AA_S$ is the category of flat families of Gieseker semistable sheaves of a particular normalized Hilbert polynomial $p\in \QQ[x]$.     
\end{example}

\textbf{(3) Existence of a fiber functor:} As  the values of $\Mst$ and $\sqcup_{d\in \NN^{\oplus I}}\Spec\kk/G_d$ are the  isomorphism groupoids of exact categories, it is natural to require that $\omega$ has a faithful extension to a morphism $\omega:\AA \to \Vect^I$ of functors with values in exact categories. In particular, $\omega_S:\AA_S\to \Vect^I_S$ should be an exact functor. Moreover, as $\OO_S(S)$ acts on $\Hom$- and $\Ext^1$-groups in $\AA_S$ as well as in $\Vect^I_S$, we also require that $\omega_S$ is $\OO_S(S)$ linear.  We also require that every morphism $f:E\to E''$ in $\AA_S$ with $\omega_S(f):\omega_S(E)\to \omega_S(E'')$ fitting into an exact sequence $0 \to V' \to\omega_S(E) \xrightarrow{\omega_S(f)} \omega_S(E'')\to 0$ in $\Vect^I_S$ can be completed to an exact sequence $0\to E' \to E \xrightarrow{f} E''\to 0$ in $\AA_S$ which implies  $V'\cong \omega_S(E')$ as $\omega_S$ is exact.\\ 
\\
As an immediate consequence we observe that $\AA_\KK$ must be of finite length. In particular, every object in $\AA_\KK$ has a Jordan--H\"older filtration with simple subquotients.
\begin{lemma} \label{isomorphism2}
The functor $\omega_S$ is conservative, i.e.\ if $f:E\to E''$ is a morphism in $\AA_S$ such that $\omega_S(f)$ is an isomorphism, then $f$ is already an isomorphism.
\end{lemma}
\begin{proof}
As $0\to 0\to \omega_S(E)\xrightarrow{\omega_S(f)}\omega_S(E'')\to 0$ is exact, we find an exact sequence $0\to E'\to E \xrightarrow{f} E''\to 0$ with $\omega_S(E')=0$. Thus,  $\omega_S(\id_{E'})=\omega_S(0)$ which implies $\id_{E'}=0$ by the  faithfulness of $\omega_S$. Hence, $E'=0$ and $f$ is an isomorphism.
\end{proof}

\begin{example}\rm
As $\omega_\kk:\AA_\kk\to \Vect_\kk^I$ is faithful and exact, general Tannakian theory tells us that $\AA=\AA_\kk$ is equivalent to the category of $B$-comodules for some $I$-graded coalgebra $B$ over $\kk$. Conversely, given such a coalgebra $B$, we can define $\AA_S$ to be the category of $I$-graded vector bundles $V$ of finite total rank with a morphism $V\to B\otimes_\kk V$ of $\OO_S$ modules such that $V_s \to B\otimes_\kk V_s$ is a semistable $B\otimes_\kk \KK$-comodule for every closed point $s\in S$ with residue field $\KK$. Semistability can be defined as for quivers using a family $\zeta=(\zeta_i)_{i\in I}$ of complex numbers in the upper half plane. There is certainly a faithful exact functor $\omega$, but it is not clear whether or not $\omega|_{\Mst^{\zeta-ss}}$ is representable unless $B$ satisfies some finiteness conditions.     
\end{example}
\begin{example} \rm
If $\AA_S$ is the category of  quiver representations satisfying some relations on  $Q_0$-graded vector bundles of finite  rank whose restriction to every closed point is $\zeta$-semistable of slope $\mu$, and if $\omega_S$ forgets the quiver representation, then all our assumptions are fulfilled.\\
We could modify the quiver case by forgetting the $I=Q_0$-grading. Then $d$ is just a natural number and $X_d=\sqcup_{\bar{d}} R_{\bar{d}}^{ss}$, where the sum is taken over all tuples $\bar{d}=(d_i)_{i\in Q_0}\in \NN^{\oplus Q_0}$ of slope $\mu$ such that $d=\sum_{i\in Q_0} d_i$. Thus, $X_d$ does not need to be connected or of pure dimension.   
\end{example}
\begin{example} \rm In the case of coherent sheaves on a smooth projective variety $X$ with ample divisor $D$, we could try to define a functor $\omega_S$ by means of $\omega_S(E)=\pr_{S\,\ast}E(mD)$ for some $m\gg 0$. Unfortunately, it is not known to the author whether or not a lower bound for $m$ depends only on the normalized Hilbert polynomial characterizing the category $\AA_S$.  Nevertheless, we can always pretend that such an $\omega_S$ exists by the following reason.  Let $\Mst_\kappa$ be any connected component of the Artin stack $\Mst^{ss}_p$ parametrizing semistable coherent sheaves on $X$ of normalized Hilbert polynomial $p$. Denote with $\AA^\kappa_S$ the  full subcategory in $\AA_S$ consisting flat families  of those semistable sheaves whose stable factors occur as subquotients of objects in $\Mst_\kappa$. This subcategory  is exact and even abelian if $S=\Spec\KK$ is a point.
Moreover, the set of  non-normalized Hilbert polynomials $P(E)$ for stable $E\in \AA^\kappa_\KK$ and arbitrary  field extensions $\KK\supset \kk$ is finite and we can find an $m\gg 0$ such that $R^i\pr_{S\, \ast} E(mD)=0$ for all $i>0$ and all flat families in $\AA^\kappa_S$. (see \cite{HuybLehn}, Theorem 3.3.7) Thus, $\omega_S|_{\AA_S^\kappa}$ is exact. Furthermore, we may assume $\omega_S(E)=0$ implies $E=0$. We encourage the reader to check that all results of the present paper concerning $\Mst_\kappa$ will only make use of the subfunctor $\AA^\kappa\subset \AA$. \\ 
Let us check the last condition in assumption (3) for $\AA^\kappa_S$. Given a morphism $f:E\to E''$ in $\AA^\kappa_S\subset \AA_S$ such that $\omega_S(f)$ tits into a short exact sequence. We form $E\to E''\to \coker(f)\to 0$ with a possibly non-flat family $\coker(f)$ of semistable sheaves in $\AA^\kappa_\KK$. Restricting the sequence to the fiber over $s\in S(\KK)$ makes it into a (right) exact sequence in $\AA^\kappa_\KK$. Now we apply the exact functor $\omega_\KK$ and conclude $\omega_\KK(\coker(f)|_s)=0$ as $\omega_S(f)|_s$ is an isomorphism. Hence, $\coker(f)|_s=0$ for all fibers, and $\coker(f)=0$ follows by Nakayama's lemma. Now we take $E'=\ker(f)$ which must be in $\AA^\kappa_S$ because $f$ is an epimorphism and $E,E''$ were flat families in $\AA^\kappa_S$. \\
With a similar argument one proves faithfulness of $\omega_S:\AA^\kappa_S\to \Vect_S$.
\end{example}
\begin{example}\rm The case of curves as better behaved as we can choose $m\gg 0$ only depending on the slope. Thus, there is no need for introducing the subcategories $\AA^\kappa_S$ as an auxiliary tool. As before we denote with $\AA_S$ the category of coherent sheaves $E$ on $S\times X$ which are flat over $S$ whose restriction to the fiber over any closed point in $S$ is semistable of slope $\mu$. We define $\omega_S(E)=\pr_{S\,\ast}E(mD)$, where the ample divisor  $D$ on $X$ is pulled back to $S\times X$, and $m > (2g-1-\mu)/\deg (D)$ with $g$ being the genus of the curve. The slope of $\OO_{S\times X}(K_X-mD)$ is smaller than $\mu$, and $\Hom_{\OO_X}(E_s,\OO(K_X-mD))=0$ follows by general arguments about morphisms between semistable sheaves. As $R^1\pr_{S\, \ast} E(mD)\cong \pr_{S\, \ast}\HHom_{\OO_{S\times X}}(E,\OO(K_X-mD))^\vee$ by relative Serre duality, the functor $\omega_S$ is exact on $\AA_S$. This example actually shows that we cannot take $\AA_\kk$ to be the category of all semistable sheaves 
even on an elliptic 
curve $X$ as such an $\omega$ might not exist for all slopes. Let us finally check faithfulness. If $E\not=0$ is a semistable vector bundle of slope $\mu$ defined over $\KK$
\begin{eqnarray*} \dim(E)&=&\dim \Ho^0(X,E(mD))\;=\;\deg(E(mD))+\rk(E(mD))(1-g)\\ &=& \deg(E)+m\rk(E)\deg(D)+\rk(E)(1-g)\\ &=&\rk(E)(\mu+m\deg(D) +1-g) \\ &>& \rk(E)g\;\ge \;0.
\end{eqnarray*}
If $E\not=0$ is a torsion sheaf, $\dim(E)>0$ is obvious. Thus $\omega_\KK(E)=0$ implies $E=0$, and $\omega_\KK$ must be faithful. Indeed, if $f:E\to E'$ is a morphism of sheaves, then $\omega_\KK(f)=0$ implies $0=\im(\omega_\KK(f))=\omega_\KK(\im(f))$ as $\omega_\KK$ is exact, and $\im(f)=0$ follows. If $f$ is defined  over $S$, we pull back the sequence $E\to E'\to \coker(f)\to 0$ of coherent sheaves on $X\times S$ to the fibers of the projection $S\times X \to S$ and conclude $\coker(f)=0$ by Nakayama's lemma and the previous discussion for $\Spec\KK$. Now we proceed in the same way with the exact sequence $0\to \ker(f)\to E\to E'\to 0$ of flat coherent sheaves and conclude $\ker(f)=0$. 
\end{example}

\subsection{GIT-theory of moduli of objects}

As we want to make a statement about moduli spaces  instead of moduli stacks, we need another assumption ensuring the existence of good quotients. Before we do this, let us recall the following well-known facts. Given a one-parameter subgroup, or 1-PS for short, $\lambda:\Gl(1)\to G_d$ defined over $\KK$, there is a decomposition of $\KK^d:=\oplus_{i\in I} \KK^{d_i}$ into irreducible $\Gl(1)$-subrepresentations parametrized by an integer $n\in \ZZ$, i.e.\ $\KK^d=\bigoplus_{n\in \ZZ} W_n$ for some $I$-graded vector spaces $W_n$ on which $\Gl(1)$ acts by multiplication with $z^n$. Introduce the descending filtration 
\[ \ldots \subset W^{n+1} \subset W^n \subset W^{n-1}\subset \ldots \subset \KK^d \]
with limits $0=\cap_{n\in \ZZ} W^n$ and $\KK^d=\cup_{n\in \ZZ}W^n$ given by $W^n=\bigoplus_{m\ge n}W_m$. Such a filtration is uniquely determined by $\lambda$, but a different 1-PS corresponding to different complementary subspaces $W_n$ inside $W^n$ will produce the same filtration. Conversely, every such filtration is created by a 1-PS which can be constructed by choosing compliments $W_n$ of $W^{n+1}$ inside $W^n$ and requiring that $\lambda$ acts with character $z\to z^n$ on $W^n$. The 1-PS's associated to different choices are conjugate inside the parabolic subgroup of $G_d$ preserving $W^\bullet$. Note that the filtration is proper, i.e.\ has not only one step, if and only if $\lambda$ does not map to $\Gl(1)$ embedded diagonally into $G_d$. Moreover, given a pair $(W,n)$ consisting of a subvector space $W\subset \KK^d$ and an integer $n$, we produce the descending filtration $W^\bullet$ of $\KK^d$ with $W^{n+1}=0, W^n=W$ and $W^{n-1}=\KK^d$ to which we can associate a 1-PS $\lambda^{W,n}$ defined over 
$\KK$ up to 
conjugation in the parabolic group associated to $W^\bullet$. \\
Given a scheme $Y_d$ with a $G_d$-action and a $G_d$-linearization $\Lin_d$, we consider the limit  $x^\lambda_0:=\lim\limits_{z\to 0}\lambda(z)x$ which might or might not exist. If it does, $\lambda$ induces an irreducible linear representation of $\Gl(1)$ in the fiber of $\Lin_d$ over $x_0^\lambda$ of weight $-\mu^{\Lin_d}(x,\lambda)$. We use the  convention $\mu^{\Lin_d}(x,\lambda)=+\infty$ if the limit does not exist. With these background at hand, we are now ready to state the following assumption.\\
\\
\textbf{(4) Existence of good GIT-quotients:} Assume that  every $X_d$ has an embedding into the open subscheme  $\overline{X}_d^{ss}$ of semistable points inside some $\kk$-scheme $\overline{X}_d$ with $G_d$-action and 
$G_d$-linearization $\Lin_d$. Moreover, there should be a projective morphism $f_d:\overline{X}_d\to \Spec A_d$ to some affine scheme of finite type such that $\Lin_d$ is ample with respect to $f_d$. Let $\tilde{p}_d:\overline{X}^{ss}\to \overline{X}^{ss}/\!\!/G_d$ be the uniform categorical quotient (see \cite{MumfordGIT}, Theorem 1.10). We also require $X_d=\tilde{p}_d^{-1}(\Msp_d)$ for some subscheme $\Msp_d\subset \overline{X}^{ss}_d/\!\!/G_d$ which should be open if $\Char\kk>0$. Furthermore, for every 1-PS in $G_d$ and any $x\in \overline{X}_d$ we require 
\begin{equation} \label{eq7} \mu^{\Lin_d}(x,\lambda)=\sum_{n\in \ZZ} \mu^{\Lin_d}(x,\lambda^{W^n,n}). \end{equation}
Moreover, for $x\in X_d$ represented by $(E,\psi)$ defined over $\KK$ the conditions 
\begin{enumerate}
 \item[(a)] the limit point $x_0^\lambda$ exists and is in $X_d$,
 \item[(b)] there are subobjects $E^n\subset E$ with $\omega_\KK(E^n)=\psi(W^n)$ for all $n\in \ZZ$,
 \item[(c)] the equation $\mu^{\Lin_d}(x,\lambda)=0$ holds,
\end{enumerate}
should be equivalent. In this case $x_0^\lambda$ is given by the associated graded (trivialized) object $(\gr E^\bullet=\oplus_{n\in\ZZ} E^n/E^{n+1},\gr \psi^\bullet)$.\\

\begin{remark}\rm
 By general GIT-theory, affine GIT-quotients $\Spec A/\!\!/G=\Spec A^G$ by reductive groups are uniform categorical and even universal categorical for $\Char\kk=0$. In particular, every open subscheme $Y\subset \Spec A/\!\!/G$ is the (uniform) categorical GIT-quotient of its preimage in $\Spec A$, and for $\Char\kk=0$ we can even drop the assumption to be open. By construction, $\overline{X}^{ss}/\!\!/G_d$ is covered by affine GIT-quotients $\Spec A_i/G_d$ with intersections also being affine GIT-quotients. We can consider $Y_i= \Msp_d\cap \Spec A_i$ to conclude that $\Msp_d$ is the uniform categorical quotient of $X_d$. In particular, $X_d$ must contain every orbit closure taken in $\overline{X}^{ss}_d$. Therefore, (a),(b) and (c) are also equivalent to the condition (d): For every $x\in X_d$ and every 1-PS $\lambda$ the limit point $x_0^\lambda$ exists and is in $\overline{X}^{ss}_d$.
\end{remark}

Let us spend a few words to explain our assumption. We wish to apply GIT-theory to the $G_d$-action on $X_d$. To do that we need some linearization on $X_d$. As seen in the examples, $X_d$ should parametrize semistable objects (with some trivialization) maybe satisfying some addition (open) properties. For practical calculations, we also want  Mumford's numerical criterion at our disposal. For that we need to enlarge $X_d$ by embedding it into some ``closure'' $\overline{X}_d$ with an extended $G_d$-action and $G_d$-linearization. Note that $X_d$ does not need to be dense in $\overline{X}_d$, but we do not loose anything if we replace $\overline{X}_d$ with the closure of $X_d$ inside $\overline{X}_d$.  Having a projective morphism to some affine scheme $\Spec A_d$ ensures Mumford's criterion for every relative ample line bundle on $\overline{X}_d$ due to \cite{GHH}, Theorem 3.3. As already mentioned, the condition $ X_d=\tilde{p}_d^{-1} (\Msp_d)$ implies that $\Msp_d=\tilde{p}_d(X_d)$ is a uniform categorical 
quotient for the $G_d$-action on $X_d$. In particular, it only depends on $X_d$ and not on $\Lin_d$ or the partial compactification $\overline{X}_d$. The final condition on 
limit points has the following interpretation. The filtration $\psi(W^\bullet)\subset \omega_\KK(E)$ is induced by a filtration $E^\bullet$ of $E$ and $\lambda(z)x$ corresponds to a one-parameter deformation of $E$ to the direct sum $\bigoplus_{n\in\ZZ} E^n/E^{n+1}$ of the corresponding subquotients of $E$. However, the subquotient are of different slopes and might not be in $\AA_\KK$. The limit is in $\AA_\KK$ if and only if all the subobjects $E^n$ are in $\AA_\KK$ as the latter should be extension closed. But this is just saying that we can lift $\psi(W^n)$ to some objects in $\AA_\KK$. Notice that $\mu^{\Lin_d}(x,\lambda)\ge 0$ for all 1-PS's as $X_d\subset \overline{X}_d^{ss}$. Moreover, $\Msp_d$ is quasiprojective by assumption on $\overline{X}_d$.  
Let us also mention that by \cite{MumfordGIT}, Proposition 2.7, $\mu^{\Lin_d}(\lambda^{W^n,n},x)$ does not depend on the choice of $\lambda^{W^n,n}$, i.e.\ on the choice of complements. Moreover, it does not even depend on $n$, i.e.\ $\mu^{\Lin_d}(\lambda^{W^n,n},x)=\mu^{\Lin_d}(\lambda^{W^n,m},x)$ for all $m\in \ZZ$, if we assume without loss of generality that $X_d$ is dense in $\overline{X}_d$. We will write $\mu^{\Lin_d}(\lambda^{W^n},x)$ for simplicity. Indeed, changing $m$ multiplies $\lambda^{W^n,m}$ with a power of the diagonal (central) embedding $\lambda_0$. However, $\lambda_0$ acts trivially on $X_d$ and, thus, also on $\overline{X}_d$ by continuity. Hence $\mu^{\Lin_d}(\lambda_0^k,x)=k\mu^{\Lin_d}(\lambda_0,x)\ge 0$ for all $x\in \overline{X}^{ss}_d$ and all $k\in \ZZ$ by Mumford's criterion. We conclude $\mu^{\Lin_d}(\lambda_0,x)=0$ for all $x\in \overline{X}^{ss}_d$. The (fiberwise) action of $\lambda_0$ on $\Lin_d|_{\overline{X}_d}$ is continuous, and we even conclude $\mu^{\Lin_d}(\lambda_0,
x)=0$ for all $x\in \overline{X}_d$. Since $\lambda$ and $\lambda_0^k\lambda$ 
applied to $x$ converge to the same limit in $\overline{X}_d$, $\mu^{\Lin_d}(\lambda_0^k\lambda,x)=k\mu^{\Lin_d}(\lambda_0,x)+\mu^{\Lin_d}(x,\lambda)=\mu^{\Lin_d}(x,\lambda)$ follows for every $x\in \overline{X}_d$ proving the independence of $\mu^{\Lin}(\lambda^{W,n},x)$ on $n$ for every subspace $W\subset \KK^d$.      
\begin{example} \rm The case of a quiver $Q$ with relations $R$ has been analyzed by King in \cite{King}, and it was the blueprint for the affine situation $\overline{X}_d=\Spec A_d=R_d^r$ using the notation of section 2, and $X_d=R_d^{r,\zeta-ss}$ is the open subvariety of semistable points with respect to the trivial line bundle with a possibly non-trivial $G_d$-action provided by a character $\theta\in \ZZ^{Q_0}$ with $\theta\cdot d=0$. The proof of our assumptions can be found in \cite{King}, section 3. The intuitive picture described above is absolutely correct in this situation. Let us modify this example a little bit by choosing a finite set of cycles in $Q$. We denote with $X_d\subset R_d^{r,\zeta-ss}$ the open subset of semistable representations for which the chosen cycles act invertible. This condition is closed under extension and forming subquotients. Thus, $\tilde{p}_d^{-1}(\tilde{p}_d(X_d))=X_d$, and also the equivalence of the  conditions (a),(b) and (c) is fulfilled. 
\end{example}
\begin{example}\rm 
The case of Gieseker semistable coherent sheaves  on smooth projective varieties is extensively studied in \cite{HuybLehn}, section 4. Here $\overline{X}_d=\overline{R}$ is the closure of the open subvariety $X_d=R=\overline{R}^{ss}$ inside some Quot-scheme parametrizing semistable quotients of a fixed normalized Hilbert polynomial for which $\Ho^0$ applied to the quotient map is an isomorphism. Notice that in contrast to the quiver case $\overline{X}_d/G_d$ is not the moduli stack of some bigger abelian category. The reader should be aware of the fact that  Huybrechts and Lehn use the opposite ascending filtration $V_{\le n}=\bigoplus_{m \le n}W_n$, but the arguments remain the same. Equation (\ref{eq7}) is implicitly contained in the proof of Lemma 4.4.5.\ as $\theta(V_{\le n})/d=\mu^{\Lin_d}(x,\lambda^{V_{\le n}})$ using their notation. Notice that our 2-step filtration associated to a vector space $W$ is different from theirs. The picture described above is also true in this case. The equivalence 
required in our assumption can be proven as follows. The limit $x_0^\lambda$ exists in $\overline{R}^{ss}=R$ if and only if all subsheaves $E_{\le n}$ associated to $\psi(V_{\le n})$ are semistable with $p(E^n)(l)=p(E)(l)$ for the normalized Hilbert polynomials and sufficiently large $l\in \NN$. This is the case if and only if the value 
$\theta(V_{\le n})=d\cdot P(E_{\le n},l)-\dim V_{\le n}\cdot P(E,l)$ is zero as every subsheaf $E'\subset E$ satisfying $p(E')(l)=p(E)(l)$ for $l\gg0$ is automatically semistable. This proves  the remaining equivalence as $\mu^{\Lin_d}(x,\lambda)=\sum_{n\in \ZZ}\theta(V_{\le n})/d$. See \cite{HuybLehn} for the details and the notation. Again, we can modify this example by choosing an open property which is closed under extensions and subquotients.  
\end{example}

\begin{lemma}
There is bijection between the points of $\Msp$ defined over $\KK$ and the semisimple objects in $\AA_\KK$. 
\end{lemma}
\begin{proof} Recall that the points in $\Msp$ correspond to the closed orbits in $X_d=\tilde{p}_d^{ss}(\Msp)$.
Let $E\in \AA_\KK$ be semisimple of dimension $d$ and let $x_0\in X_d$ be a point in the closure of $G_d x$ for some lift $x=(E,\psi)\in X_d$ of $E$. There is a 1-PS in $G_d$ defined over $\KK$ with limit point $\lim\limits_{z\to 0}\lambda(z)x=x_0$ which because of (a)$\Leftrightarrow$(b) implies $x_0=(\gr E^\bullet,\gr \psi^\bullet)$ for some filtration $E^\bullet$ on $E$. As $E$ is semisimple, $\gr E^\bullet \cong E$ and $x_0 \in G_dx$ follows. Conversely, assume $G_dx=\overline{G_dx}$, where the closure is taken in $X_d$. Using the equivalence of (a) and (b) again, there is a 1-PS defined over $\KK$ associated to the Jordan--H\"older filtration $E^\bullet$ of $E\in\AA_\KK$. The limit point $x_0^\lambda$ represents $\gr E^\bullet$ and is in the closed orbit $G_dx$. Hence, $E\cong \gr E^\bullet$ is semisimple.  
\end{proof}
By general GIT-theory, the map $\tilde{p}_d:X_d\to \Msp_d$ restricts to a geometric quotient over the image $\Msp_d^s$ of the open subscheme  $X_d^{st}=\overline{X}_d^{st}\cap X_d$ of stable points in $X_d$. If $X_d^{st}$ is even smooth, $\tilde{p}_d:X_d^{st} \to \Msp_d^s$ is a principal $PG_d$-bundle. 

\begin{lemma}
There is a bijection between the points of $\Msp^s$ defined over $\KK$ and the simple objects in $\AA_\KK$.
\end{lemma}
\begin{proof}
Assume $E$ is simple and $x\not\in X^{st}_d$ for some lift $x=(E,\psi)$ of $E$. The points of  $X^{st}_d$ can be described as the set of all $y\in X_d$ such that $\mu^{\Lin_d}(y,\lambda)>0$ for all 1-PS's not in $\Gl(1)\subset G_d$. Thus, there is a 1-PS $\lambda$ not mapping to $\Gl(1)\subset G_d$ with $\mu^{\Lin_d}(x,\lambda)=0$ as we always have $\mu^{\Lin_d}(x,\lambda)\ge 0$ by assumption $X_d\subset \overline{X}_d^{ss}$. The associated filtration is proper, and using our equivalent conditions we find some proper subspace $W^n\subset \KK^d$ which is $\omega_\KK(E')$ for some proper subobject $E'$ in $E$, a contradiction.  Conversely, given $x\in X^{st}_d$ and assume that $E$ is not simple. Then there is a proper subobject $E'\subset E$ inducing a proper 2-step filtration $\omega_\KK(E')=W\subset \KK^d$ filtration unique up to shifts. It corresponds to a 1-PS $\lambda^{W}$ with $\mu(x,\lambda^W)=0$ contradicting $x\in X^{st}$.    
\end{proof}

\subsection{The universal Grassmannian}

By the universal property of $\Msp=\sqcup_{d\in \NN^{\oplus I}} \Msp_d$ the map $\oplus:\Mst\times\Mst \to \Mst$ sending two objects to their direct sum descends to $\oplus:\Msp\times\Msp \to \Msp$. We want to show that the latter map has good properties. For that and also for later applications, we need to study extensions of objects in $\AA_\kk$. Let $\mathfrak{E}xact_S$ denote the isomorphism groupoid of all short exact sequences in $\AA_S$. Using the fact that $\AA$ satisfies the gluing axioms of a stack, we can easily see that $\mathfrak{E}xact$ is a stack, too. For fixed dimension vectors $d,d'\in \NN^{\oplus I}$, we denote with $\mathfrak{E}xact_{d,d'}$ the substack of exact sequences $0\to E_1\to E_2\to E_3 \to 0$ with $\dim E_1=d$ and $\dim E_3=d'$. There are obvious morphisms $\pi_i:\mathfrak{E}xact \to \Mst$ mapping a sequence as above to its i-th entry. Note that $\pi_2$ is faithful and the following assumption makes perfectly sense.\\
\\
\textbf{(5) The universal Grassmannian is proper:} The map $\pi_2:\mathfrak{E}xact \to \Mst$ is representable and proper. \\

By definition, the fiber of $\pi_2$ at $E\in \AA_\KK$ is just the set of all subobjects of $E$. The previous assumption ensures that this set is a proper scheme\footnote{Strictly speaking it is an algebraic space, but we will see shortly that  it is actually a scheme.}.
\begin{example}\rm
In the case of sheaves, assumption (5) is fulfilled by  the existence and properness of Grothendieck's  (universal) Quot-scheme. The same holds for quivers (with relations).  
\end{example}
Using our assumption, we get a  description of $\mathfrak{E}xact_{d,d'}$ as a quotient stack $\mathfrak{E}xact_{d,d'}=Y_{d,d'}/G_{d+d'}$ with $Y_{d,d'}=\mathfrak{E}xact_{d,d'} \times_{\Mst_{d+d'}}X_{d+d'}$ representing the functor mapping a $\kk$-scheme $S$ to triples $(E,\psi,E')$ with $E\in \AA_S,\psi:\omega_S\xrightarrow{\sim} \OO_S^{d+d'}$ and $E'\subset E$ a subobject in $\AA_S$ of dimension $d$. Moreover, $Y_{d,d'}\to X_{d+d'}$ is proper. \\
By exactness of $\omega_S$, every exact sequence in $\AA_S$ provides an exact sequence of $I$-graded vector bundles on $S$ of finite total rank. The stack of the latter is just $\sqcup_{d,d'\in \NN^{\oplus I}} \Spec\kk/P_{d,d'}$ with $P_{d,d'}\subset G_{d+d'}$ being the parabolic subgroup of block triangular matrices in $G_{d+d'}$ fixing the subvector space $\kk^d \hookrightarrow \kk^d\oplus \kk^{d'}=\kk^{d+d'}$. Thus, we get the following commutative diagram
\[ \xymatrix { \mathfrak{E}xact_{d,d'} \ar[r]^(0.3)\iota \ar[dr]_{\pi_2} & \Mst_{d+d'}\times_{\Spec\kk/G_{d+d'}} \Spec\kk/P_{d,d'} \ar[rr]^(0.6)\pr \ar[d]^\pr & & \Spec\kk/P_{d,d'} \ar[d] \\ & \Mst_{d+d'} \ar[rr] & & \Spec\kk/G_{d+d'}.  }\]
The fiber product represents the functor of pairs $(E,U)$ with $E\in \AA_S$ of dimension $d+d'$ and $U$ a subbundle of $\omega_S(E)$ of rank $d$. It can be written as the quotient stack $X_{d+d'}\times \Gr_d^{d+d'} / G_{d+d'}$, where $\Gr_d^{d+d'}$ is the Grassmann variety parametrizing $I$-graded subvector spaces in $\kk^{d+d'}$ of dimension $d$.  We claim that $\iota$, mapping $0\to E_1\to E_2 \to E_3\to 0$ to $(E_2,\omega_S(E_1))$ is a closed embedding, and in particular  representable. Thus, the composition $\mathfrak{E}xact_{d,d'} \longrightarrow \Spec\kk/P_{d,d'}$ is representable, too. Therefore, $\mathfrak{E}xact_{d,d'}$ has another description as a quotient stack, namely $X_{d,d'}/P_{d,d'}$ with $X_{d,d'}=\mathfrak{E}xact_{d,d'}\times_{\Spec\kk/P_{d,d'}} \Spec\kk$ representing the functor mapping a $\kk$-scheme $S$ to the isomorphism classes\footnote{We could also take the isomorphism groupoid which is equivalent to its set of isomorphism classes.} of tuples \[ (0\to E_1\to E_2 \to E_3\to 0, \psi_1,\
psi_2,\psi_
3) \]
consisting of a short exact sequence in $\AA_S$ and an isomorphism 
\[ \xymatrix { 0 \ar[r] & \omega_S(E_1) \ar[d]_{\psi_1}^\wr \ar[r] & \omega_S(E_2) \ar[d]_{\psi_2}^\wr \ar[r] &\omega_S(E_3) \ar[d]_{\psi_3}^\wr \ar[r] & 0 \\ 0 \ar[r] & \OO_S^d \ar[r] & \OO_S^{d+d'} \ar[r] & \OO_S^{d'} \ar[r] & 0, }\]
where we used the shorthand $\OO_S^d$ for the trivial $I$-graded vector bundle $\OO_S\otimes_\kk \kk^d=\bigoplus_{i\in I} \OO_S^{\oplus d_i}$ of rank $d$.
To see that $\iota$ is a closed embedding, it is sufficient to check that $\hat{\iota}:Y_{d,d'} \to X_{d+d'}\times \Gr_d^{d+d'}$ is a closed embedding. First of all it is proper as $\hat{\pi}_2:Y_{d,d'} \xrightarrow{\hat{\iota}} X_{d+d'}\times \Gr_d^{d+d'} \xrightarrow{\pr} X_{d+d'}$ and $\pr$ are proper. Furthermore, $\iota$ is injective on $\KK$-points. Indeed, Given $E',E''\subset E$ with $\omega_\KK(E')=\omega_\KK(E'')$, we have $\omega_\KK(E'/E'\cap E'')=0$ by exactness of $\omega_\KK$. Hence, $E'=E'\cap E''$ as $\omega_\KK$ is faithful. By symmetry, $E'=E''$. Due to Zariski's main theorem, every proper morphism between locally noetherian schemes which is injective on points is a closed embedding. As a consequence, $Y_{d,d'}\subset X_{d+d'}\times \Gr_{d}^{d+d'}$ is a scheme and not just an algebraic space. Because of $Y_{d,d'}\cong X_{d,d'}\times_{P_{d,d'}}G_{d,d'}$, the same must be true for $X_{d,d'}=X_{d,d'}\times_{P_{d,d'}}P_{d,d'}$ being a closed subscheme of 
$Y_{d,d'}$.\\ 
\begin{lemma}
The map $\hat{\pi}_2:X_{d,d'}\to X_{d+d'}$ sending a trivialized short exact sequence to the trivialized middle term is a closed embedding.
\end{lemma}
\begin{proof}
Being a  composition $X_{d,d'}\hookrightarrow Y_{d,d}\longrightarrow X_{d+d'}$ of a closed embedding and a proper map, $\hat{\pi}_2$ must be proper. It is also injective on $\KK$-points, due to the fact that every subobject $E_1$ of $E_2$ is uniquely determined (up to isomorphism) by the requirement $\psi(\omega_\KK(E_1))= \KK^d \subset \KK^{d+d'}$ as we have seen already. Every proper map which is injective on $\KK$-points for all field extensions $\KK\supset \kk$ is a closed embedding.
\end{proof}

The following lemma is a generalization of a statement used to show that $\iota$ and $\hat{\pi}_2:X_{d,d'}\to X_{d+d'}$ are injective on points. Note that a subobject of $E\in \AA_S$ is an equivalence class of morphisms $E'\to E$ in $\AA_S$ which extend to a short exact sequence $0\to E'\to E\to F\to 0$ in $\AA_S$. As one can check for vector bundles, this  is a much stronger condition than saying that $E'\to E$ is a monomorphism. 

\begin{lemma} \label{lemma3}
If $E',E''$ are subobjects of $E\in \AA_S$ with $\omega_S(E')=\omega_S(E'')$ inside $\omega_S(E)$, then $E'=E''$.
\end{lemma}
\begin{proof}
Let us introduce the notation $\Xst:=\Mst_{d+d'}\times_{\Spec\kk/G_{d+d'}}\Spec\kk/P_{d,d'}$ and assume for simplicity that $S$ is connected. By assumption, $E',E''$ define a morphism from $S$ to the fiber product $\mathfrak{E}xact_{d,d'}\times_\Xst \mathfrak{E}xact_{d,d'}$ for suitable $d,d'\in \NN^I$. Since $\iota$ is a closed embedding, it is in particular a monomorphism, and the projection to each of the factors provides an isomorphism between $\mathfrak{E}xact_{d,d'}\times_\Xst \mathfrak{E}xact_{d,d'}$ and $\mathfrak{E}xact_{d,d'}$. 
\end{proof}

\begin{lemma}
If $X_d=\overline{X}^{ss}_d$ of all $d\in \NN^{\oplus I}$, then the morphism $\oplus:\Msp\times \Msp \longrightarrow \Msp$ is finite. 
\end{lemma}
\begin{proof}
As the isomorphism types and multiplicities of the stable summands of a polystable object are unique, the morphism is certainly quasi-finite. It remains to show that $\oplus$ is proper. 
Recall that 
\[ \overline{X}^{ss}_d/\!\!/G_d=\Proj\big( \bigoplus_{k\in \NN} \Ho^0(\overline{X}^{ss}_d,\Lin^{\otimes k}_d|_{\overline{X}^{ss}_d})^{G_d} \Big) \longrightarrow \Spec \kk[\overline{X}^{ss}_d]^{G_d} \]
obtained by identifying $\kk[\overline{X}^{ss}_d]^{G_d}$ with the degree zero part of the graded ring $\bigoplus_{k\in \NN} \Ho^0(\overline{X}^{ss}_d,\Lin^{\otimes k}_d|_{\overline{X}^{ss}_d})^{G_d} $ is always proper. Under the assumptions of the lemma, we end up with a commutative diagram
\[ \xymatrix @C=2cm { \Msp_d\times \Msp_{d'} \ar[r]^\oplus \ar[d] & \Msp_{d+d'} \ar[d] \\  \Spec[X_d]^{G_d}\times \Spec[X_{d'}]^{G_{d'}} \ar[r]^\oplus & \Spec[X_{d+d'}]^{G_{d+d'}}  }\]
with proper vertical maps. Hence, it suffices to show that 
\[\oplus:\Spec[X_d]^{G_d}\times \Spec[X_{d'}]^{G_{d'}} \longrightarrow \Spec[X_{d+d'}]^{G_{d+d'}}\] 
is proper.
Consider the following commutative diagram
\[ \xymatrix { & \mathfrak{E} xact_{d,d'}\cong Y_{d,d'} /G_{d+d'} \ar@/_/[dl]_{\pi_1\times \pi_3} \ar[dr]^{\pi_2}  \ar[dd] & \\ X_d/G_d \times X_{d'}/G_{d'}, \ar@/_/[ur]_{\sigma_0} \ar[dd] & & X_{d+d'}/G_{d+d'}  \ar[dd] \\
& \Spec \kk[Y_{d,d'}]^{G_{d+d'}}  \ar@/_/[dl]_{\tilde{\pi}_1\times \tilde{\pi}_3} \ar[dr]^{\tilde{\pi}_2} &  \\ 
\Spec \kk[X_d]^{G_d} \times \Spec \kk[X_{d'}]^{G_{d'}} \ar@/_/[ur]_{\tilde{\sigma}_0} \ar[rr]_\oplus & & \Spec \kk[X_{d+d'}]^{G_{d+d'}} }  \]
with $Y_{d,d'}=X_{d,d'}\times_{G_{d,d'}}G_{d+d'}=\mathfrak{E} xact_{d,d'}\times_{\Mst_{d+d'}}X_{d+d'}$. Here, $\sigma_0$ maps a pair $(E,E')$ of objects to its direct sum $E\oplus E'$ providing a right inverse of $\pi_1\times \pi_3$. Thus, $\tilde{\sigma}_0$ is also a section providing a closed embedding. It remains to show that $\tilde{\pi}_2$ is proper. Note that $\hat{\pi}_2:Y_{d,d'}\longrightarrow X_{d+d'}$, being the pull-back of $\pi_2$, must be proper with Stein factorization $Y_{d,d'} \to \Specbf \hat{\pi}_{2\,\ast}\OO_{Y_{d,d'}} \to X_{d+d'}$. Taking global sections on $X_{d+d'}$ we see that $\kk[X_{d+d'}]\longrightarrow \kk[Y_{d,d'}]$ is finite, hence integral. Applying the Reynolds operator of $\kk[Y_{d,d'}]$ to an integral equation for $a\in \kk[Y_{d,d'}]^{G_{d+d'}}$, we obtain that $\kk[X_{d+d'}]^{G_{d+d'}}\longrightarrow \kk[Y_{d,d'}]^{G_{d+d'}}$ is integral, too. Thus $\tilde{\pi}_2$ is finite, hence proper.
\end{proof}

\begin{corollary} \label{property} Assume that $X_d=\overline{X}^{ss}_d$ for all $d\in \NN^{\oplus I}$. Let $P$ be a property of objects in $\AA_\KK$ which is closed under subquotients and extension, that means for every short exact sequence
\[ 0\longrightarrow E_1 \longrightarrow E_2 \longrightarrow E_3\longrightarrow 0 \]
in $\AA_\KK$, $E_2$ has property $P$ if and only if $E_1$ and $E_3$ have property $P$. In particular, the full subcategory $\AA^P_\KK\subset \AA_\KK$ of objects having property $P$ is a Serre subcategory. We also assume the points in $\Msp$ having property $P$ are the points of a subscheme $\Msp^P$. We define $\AA^P_S$ to be the full subcategory of objects $E$ in $\AA_S$ such that $E|_s$ is in $\AA^P_\KK$ for every point $s\in S$ with residue field $\KK$.  Then $\Msp^P$ is certainly the moduli space for $\AA^P$ and $\oplus:\Msp^P\times \Msp^P\to \Msp^P$ is proper. 
\end{corollary}
\begin{proof}
As $P$ is closed under extensions and subquotients, the diagram
\[ \xymatrix @C=2cm { \Msp^P\times \Msp^P \ar[r]^{\oplus} \ar@{^{(}->}[d] & \Msp^P \ar@{^{(}->}[d] \\ \Msp\times \Msp \ar[r]^{\oplus} & \Msp } \]
is cartesian and the lower horizontal map is proper by the previous lemma.
\end{proof}
A careful look at the proof shows that giving a property $P$ as above is just the same as to give a subscheme $\Msp^P\subset \Msp$ such that the previous diagram is cartesian. 
Let us give some examples of this situation.
\begin{example} \rm \label{property2}
In the case of quivers (with relations), we can choose a finite set of cycles in $Q$ and say that a representation has property $P$ if every chosen cycle acts invertibly. This defines an open subscheme $\Msp^P$ in $\Msp$. Another example is given by given by requiring that the cycles act nilpotently which defines a closed subscheme $\Msp^P$ in $\Msp$. 
\end{example}

\subsection{Deformation theory of objects}

We come to another assumption we need to impose on $\AA$. Notice that $\Ho^{-1}(\Ta^\bullet_E\Mst_d)=\Lie \Stab_{G_d}(E,\psi)=\Hom_{\AA_\KK}(E,E)$ for any object $E\in \AA_\KK$ and a choice of trivialization $\psi:\omega_\KK(E)\xrightarrow{\sim}\KK^d$. The other non-trivial cohomology group $\Ho^0(\Ta^\bullet_E\Mst_d)$ of the tangent complex at $E\in \Mst_d$ is given by the tangent space of the associated deformation functor which is also the normal space of the $G_d$-orbit through $(E,\psi)$. We want to assume that this $\KK$-vector space is given by $\Ext^1_{\AA_\KK}(E,E)$ in such a way that the action of $\Stab_{G_d}(E,\psi)=\Aut_{\AA_\KK}(E)$ is compatible with this isomorphism. On way to require this, is to say that there is a natural quasi-isomorphism between $\Hom^\bullet_{B_\KK}(I^\bullet,I^\bullet)[1]$ computing all $\Ext$-groups with $\Ta^\bullet_E\Mst$, at least if $\Aut_{\AA_\KK}(E)$ is connected. Here, $E\to I^\bullet$ is an injective resolution\footnote{Note that $\AA_\KK$ might not have 
enough projective or injective objects. In particular, there is no obvious  way to define higher $\Ext$-groups and a homological dimension of $\AA_\KK$.} of $E$ considered as a comodule over the coalgebra $B_\KK$ associated to $\omega_\KK:\AA_\KK\to \Vect_\KK^I$. This would imply, that $\AA_\KK$ considered as a subcategory of $B_\KK$-comodules is of homological dimension one. In practice, $\AA_\KK$ embeds into another abelian category having enough projective or injective objects, and we could define higher $\Ext$-groups with respect to this abelian category. In any case, checking $\Hom^\bullet(I^\bullet,I^\bullet)[1]\cong \Ta^\bullet_E\Mst$ in $D^b(\Vect_\KK)$ can be difficult in practice.\\

We give another criterion which is often easier to check. 
\begin{definition}
Let $R$ be a commutative local Artin $\kk$-algebra and $S$ be a $\kk$-scheme.  An $R$-object in $\AA_S$ is a pair $(E,\nu)$ with $E\in \AA_S$ and $\nu:R \longrightarrow \End_{\AA_S}(E)$ a $\kk$-algebra homomorphism. A morphism between $R$-objects is given by a  morphism $f:E\to E'$ in $\AA_S$ commuting with the $R$-action. We denote the resulting category with $\tilde{\AA}_S^R$. The full subcategory $\AA_S^R\subset \tilde{\AA}_S^R$ contains those $R$-objects  for which $\omega_S(E)$ equipped with the canonical $R$-action via $\omega_\KK\circ \nu$ is a locally free $R\otimes_\kk \OO_S$-module. In particular, $\omega_S(f)$ is a morphism between locally free $R\otimes \OO_S$-modules.   
\end{definition}
The space of homomorphism of $R$-objects and the space of extensions between such objects has a natural structure of an $R\otimes \OO_S(S)$-module. Note that $\AA_S^R$ is closed inside $\tilde{\AA}_S^R$ under extensions. We denote with $\omega_S^R$ the functor sending $(E,\nu)\in \AA_S^R$ to $\omega_S(E)\in \Vect^I_{S\otimes R}$, where $S\otimes R$ is a short notation for $S\times \Spec R$. \\ 
Given a homomorphism $R\to R'$ of local Artin $\kk$-algebras, we write $R'$ as a quotient $R^m\xrightarrow{M} R^n\twoheadrightarrow R'$ of free $R$-modules. We define $R'\otimes_R E$ to be the cokernel of the induced morphism $E^m \xrightarrow{M} E^n$ if it exists. Then $R'\otimes_R E$ is an $R'$-object and we have 
\begin{equation} \label{eq6} \Hom_{\tilde{\AA}_S^{R'}}(R'\otimes_R E, F)\cong \Hom_{\tilde{\AA}_S^R}(E, F_R) \end{equation}
for every $R'$-object $F$ in $\AA_S$, where $F_R$ is $F$ with the induced $R$-action given by restriction of scalars. The $R$-object $F_R$ might not be in $\AA_S^R$ even if $F\in \AA_S^{R'}$. However, assuming existence, $R'\otimes_R E\in \AA^{R'}_S$ if $E\in \AA_S^R$. \\
\\
\textbf{(6) Existence of a good deformation theory:} For every $\kk$-scheme $S$ and every local Artin $\kk$-algebra $R$, there is an $R\otimes\OO_S$-linear equivalence $p_S^R:\AA_{S\otimes R} \xrightarrow{\sim} \AA_S^R$ of categories such that $\omega_S^R\circ p_S^R=\omega_{S\otimes R}$. Moreover, $p_{(-)}^R$ should be compatible with pull-backs along morphisms $S\to S'$ inducing an isomorphism $\AA_{(-)\otimes R}\xrightarrow{\sim} \AA^R_{(-)}$ of functors. Moreover, for every morphism $R\to R'$ and every $R$-object $(E,\nu)$ in $\AA_S^R$ the $R'$-object $R'\otimes_R E$ should exist providing us with another morphism $R'\otimes_R: \AA_S^R \to \AA_S^{R'}$ of categories for every $S$. We require that $p_S^{(-)}$ is also compatible with morphisms $R\to R'$  between local Artin $\kk$-algebras inducing $\phi:S\otimes R' \to S\otimes R$, i.e. the following diagram commutes (up to a 2-isomorphism)
\[ \xymatrix @C=1.5cm { {\AA}_{S\otimes R} \ar[r]^{\phi^\ast} \ar[d]^\wr_{p^R_S} & {\AA}_{S\otimes R'} \ar[d]^\wr_{p^{R'}_S} \\ {\AA}_S^{R} \ar[r]^{R'\otimes(-)} & {\AA_S^{R'}}. }\] 
Thinking of $\AA^{(-)}_{(-)}$ and $\AA_{(-)\otimes (-)}$ as bifunctors, we just required that $p$ is an isomorphism between them.
\begin{example} \rm
In all our examples this assumption is fulfilled. The equivalence is provided by the push-down of objects on $S\times\Spec R$ along the affine morphism $S\times \Spec R \longrightarrow S$. The functor $R'\otimes_R(-)$ is just base change along $S\times \Spec R' \longrightarrow S\times \Spec R$ which preserves flatness and semistability.
\end{example}
The case  $R=\kk[\varepsilon]:=\kk[x]/(x^2)$ plays a special role, as it describes tangent vectors. Moreover, $\Spec \kk[\varepsilon]$ is a co-commutative cogroup, and we write $S[\varepsilon]$ for $S\otimes \Spec R$.
\begin{proposition} \label{normal_ext}
\begin{enumerate}
\item The Rim-Schlessinger axiom holds, i.e.\ for every pair $f:R'\twoheadrightarrow R, g: R''\to  R$ of local Artin $\kk$-algebras with $f$ being surjective, we get 
\[\AA_{S\otimes R'\times_R R''}\cong \AA_{S\otimes R'}\times_{\AA_{S\otimes R}} \AA_{S\otimes R''}. \]
\item There is an equivalence between the category $\AA_{S[\varepsilon]}$ and the category of short exact sequences $0\to E\to E'\to E\to 0$ in $\AA_S$. The restriction map $\AA_{S[\varepsilon]} \to \AA_S$ induced by $\kk[\varepsilon]\to \kk$ corresponds to the functor sending such an exact sequence to $E$. 
\item By general arguments, the fibers of the map $\Iso(\AA_{S[\varepsilon]})\to \Iso(\AA_S)$ between sets of isomorphism classes is an $\OO_S(S)$-module and the induced isomorphism of the fiber onto $\Ext^1_S(E,E)$ is $\OO_S(S)$-linear. 
\item If $S=\Spec\KK$ and $E\in \AA_\KK$, the previous statement provides a $\KK$-linear isomorphism between $\Ext^1_{\AA_\KK}(E,E)$, $\Ho^0(\Ta_E \Mst)$ and the normal space of the $G_d$-orbit $G_d(E,\psi)\subset X_d$ for some trivialization $\psi:\omega(E)\cong \KK^d$. All isomorphisms are  compatible with the action of $\Aut_{\AA_S}(E)=\Stab_{G_d}(E,\psi)$.
\end{enumerate}
\end{proposition}
\begin{proof}
For algebraic stacks, the first statement is proven in \cite{Stackproject}, chap. 80: Artin's axioms, Lemma 5.2. Let us spell out some details also used for the second part. We use the shorthand $\bar{R}=R'\times_R R''$.
Note that objects of $ \AA_{S\otimes R'}\times_{\AA_{S\otimes R}} \AA_{S\otimes R''}$ consist of triples $(E',E'',\alpha)$ with $E'\in \AA_{S\otimes R'}$, $E''\in \AA_{S\otimes R''}$ and an isomorphism $\alpha:E|_{S\otimes R} \xrightarrow{\sim} E'|_{S\otimes R}$. First of all, the statement in (1) is true for $\Vect^I_S$. As $\omega$ is faithful, the induced functor 
\[\AA_{S\otimes \bar{R}}\longrightarrow \AA_{S\otimes R'}\times_{\AA_{S\otimes R}} \AA_{S\otimes R''} \]
must be faithful, too. To show that it is full and essentially surjective, we use our equivalent description obtained by $p_S^{(-)}$. Given an $R'$-object $E'$ in $\AA_S$, an $R''$-object $E''$ in $\AA_S$  and an isomorphism $\alpha$ as above. Then, $E'\oplus E''$ can be seen as an $\bar{R}$-object by restriction of scalars. The same applies to their image $R\otimes_{R'} E' \xrightarrow{\sim} R\otimes_{R''} E''$ in $\AA_S^R$ which we denote with $E$. Consider the following cartesian diagram of $\bar{R}$-objects. 
\[ \xymatrix { 0 \ar[r] & K' \oplus K'' \ar[r] \ar@{=}[d] & \tilde{E} \ar@{^{(}->}[d] \ar[r] & E \ar@{^{(}->}[d]^{\Delta_{E}} \ar[r] & 0 \\  
0 \ar[r] & K' \oplus K'' \ar[r] & E'\oplus E''  \ar[r] & E\oplus E  \ar[r] & 0 }
\]
We claim that $\tilde{E}\in \AA_S^{\bar{R}}$ maps to the triple $(E',E'',\alpha)$. As $E'$ is induced from $R'$ by restriction of scalar, the  morphism $\tilde{E}\to E'$ has a factorization $\tilde{E} \to R'\otimes_{\bar{R}}\tilde{E} \xrightarrow{\beta'} E'$ by equation (\ref{eq6}). Now, $\omega^{R'}_X(\beta)$ is an isomorphism, and by Lemma \ref{isomorphism2}, $\beta'$ must be an isomorphism, too. Similarly for $E''$. Conversely, starting with an $\bar{R}$-object $\bar{E}$, we apply this construction to $E':=R'\otimes_{\bar{R}}E$ and $E'':=R'\otimes_{\bar{R}}E$ and get a morphism $\bar{E}\xrightarrow{\eta}\tilde{E}$ of $\bar{R}$-objects by the universal property of fiber products. As the Rim--Schlessinger axiom holds in $\Vect_S^I$, we conclude that $\omega_S^{\bar{R}}(\eta)$ is an isomorphism which implies that $\eta$ is an isomorphism by Lemma \ref{isomorphism2}.
\\
For the second part we have to construct mutually inverse (up to 2-isomorphisms) functors between $\AA_{S[\varepsilon]}\cong \AA_S^{\kk[\varepsilon]}$ and the category of short exact sequences $0\to E\to E'\to E\to 0$ in $\AA_S$.  For every $\kk[\varepsilon]$-object $(E',\nu)$ in $\AA_S$ with $\kk\otimes_{\kk[\varepsilon]} E'\cong E$ we get a morphism $p:E'\to E$ of objects in $\AA_S$ by equation (\ref{eq6}) applied to $\id_E$ and restriction of scalars to $\kk$ (forget $\varepsilon$-action) afterward. Furthermore, the sequence 
\[ 0 \longrightarrow \omega_S(E)=\varepsilon \omega_S(E') \longrightarrow \omega_S(E') \xrightarrow{\omega_S(p)} \omega_S(E)=\omega_S(E')/\varepsilon\omega_S(E') \longrightarrow 0 \]
is exact. Hence, we find an exact sequence $0\to K\xrightarrow{i} E'\xrightarrow{p} E\to 0$ in $\AA_S$ with $\omega_S(K)=\omega_S(E)$. As $E=\coker \nu(\varepsilon)$ and $K=\ker(p)$, $E'\xrightarrow{\nu(\varepsilon)} E'$ factorizes through $K$, and $K$ is the image of $\nu(\varepsilon)$, $K=\varepsilon E'$ for short. As $\nu(\varepsilon)$ vanishes on $\varepsilon E'$, we obtain a morphism $E\to \varepsilon E'$ which becomes an isomorphism after applying $\omega_S$. Thus, $E\cong \varepsilon E'$ by Lemma \ref{isomorphism2}, giving rise to the short exact sequence
\[ 0 \to E \xrightarrow{\;i\;} E' \xrightarrow{p} E \to 0.\]
with $\nu(\varepsilon)=ip$. Conversely, any such sequence defines a $\kk[\varepsilon]$-action on $E'$ via $\nu(\varepsilon)=ip$ which  turns $\omega_S(E')$ into a locally free $\OO_S[\varepsilon]$-module. This correspondence is functorial, and it is not difficult to see that these functors are mutually inverse to each other other up to 2-isomorphisms.\\ 
For the third part we have to check that the functors constructed in the previous step preserve the $\OO_S(S)$-module structure on the induced fibration on isomorphism classes. The sum $0\to E\to \hat{E}\to E \to 0$ of two short exact sequences \\ $0\to E\to E'\to E\to 0$ and $0\to E\to E''\to E \to 0$ is defined via the commutative diagram 
\[ \xymatrix { 0 \ar[r] & E \oplus E \ar[r]  & E'\oplus E'' \ar[r] & E \oplus E \ar[r] & 0 \\
0 \ar[r] &  E \oplus E \ar[d]^{\nabla_E} \ar[r] \ar@{=}[u] & \tilde{E} \ar[u] \ar[d] \ar[r] & E \ar[u]_{\Delta_E} \ar@{=}[d] \ar[r] & 0 \\
0 \ar[r] & E \ar[r] & \hat{E} \ar[r] & E \ar[r] & 0 } \]
and the requirement that the upper squares are cartesian while the lower squares are cocartesian. Note that the upper squares compute the image of $(E',\nu')\in \AA_S^{\kk[\varepsilon']}$ and $(E'',\nu'')\in \AA_S^{\kk[\varepsilon'']}$ under the isomorphism $\AA_S^{\kk[\varepsilon',\varepsilon'']}\cong \AA_S^{\kk[\varepsilon']}\times_{\AA_S} \AA_S^{\kk[\varepsilon'']}$ of part (1). The lower diagram computes $\kk[\varepsilon',\varepsilon'']/(\varepsilon'-\varepsilon'')\otimes_{\kk[\varepsilon',\varepsilon'']}\tilde{E}$, i.e.\ the image of $(\tilde{E},\tilde{\nu})$ under the map $S[\varepsilon',\varepsilon'']\longrightarrow S[\varepsilon]$ sending $\varepsilon'$ and $\varepsilon''$ to $\varepsilon$ which provides $S[\varepsilon]$ with the structure of a cocommutative cogroup over $S$ inducing a commutative group structure on the fibers of the fibration
$\Iso(\AA_{S[\varepsilon]}) \longrightarrow \Iso( \AA_S)$. Multiplication with $a\in \OO_S(S)$ is induced by the automorphism $\varepsilon \to a\varepsilon$ of $\OO_S[\varepsilon]$ which in terms of short exact sequences $0\to E\to E'\xrightarrow{p} E\to 0$ corresponds the multiplication of $p$ with $a$. \\
The last part of the Proposition is a well-known fact combined with part (3). The statement about the action of $\Aut_S(E)$ is a consequence of the functoriality of our constructions. We leave the details to the reader.
\end{proof}
Recall that an $S$-point of $X_d$ is given by a pair $(E,\psi)$ with $E\in \AA_S$ and $\psi:\omega_S(E)\xrightarrow{\sim}\OO^d_S$ a trivialization of $\omega_S(E)$. By general arguments, an $S$-point of the tangent cone $\Ta X_d:=\Spec\Sym \Omega_{X_d}$ is given by a $S[\varepsilon]$ point $(E',\psi')$ of $X_d$. Here, $E'\in \AA_{S[\varepsilon]}$ corresponds to a short exact sequence $0\to E\xrightarrow{i} E'\xrightarrow{p}E\to 0$ as we have seen above with $\varepsilon$ acting via $ip$. Furthermore, $\psi':\omega(E') \xrightarrow{\sim} \OO_S[\varepsilon]^d$ is an isomorphism of $\OO_S[\varepsilon]$-modules which provides an isomorphism 
\[ \xymatrix @C=2cm { 0 \ar[r] &  \omega_S(E) \ar[r] \ar[d]^\psi_\wr & \omega_S(E') \ar[d]^{\psi'} \ar[r] & \omega_S(E) \ar[d]^\psi_\wr \ar[r] & 0 \\ 0\ar[r] & \OO_S^d \ar[r]^{\iota_1} & \OO_S^{d}\oplus \OO_S^d \ar[r]^{\pr_2} & \OO_S^d \ar[r] & 0 }\] 
of exact sequences of $I$-graded vector bundles on $S$. Defining morphisms of such triples $(0\to E\to E'\to E\to 0,\psi,\psi')$ in a natural way, we have just proven the first part of the following Proposition, leaving the remaining part to the reader. 
\begin{proposition} There is an isomorphism between the set of $S$-valued points on $\Ta X_d$ and the set of isomorphism classes of triples $(0\to E\to E'\to E\to 0,\psi,\psi')$. The projection to $X_d$ corresponds to the function sending such a triple to $(E,\psi)$ and the isomorphism respects the $\OO_S(S)$-linear structure on the fibers. 
\end{proposition}
\begin{corollary} \label{tangent_bundle}
The restriction of $\hat{\pi}_1\times \hat{\pi}_3:X_{d,d}\longrightarrow X_d\times X_d$ to the diagonal $X_d$ is isomorphic to $\Ta X_d$. 
\end{corollary}

We want to apply these results to determine the map $\pi_1\times \pi_3:\mathfrak{E}xact\longrightarrow \Mst\times\Mst$. 
or rather its atlas version $\hat{\pi}_1\times \hat{\pi}_3:X_{d,d'}\longrightarrow X_d\times X_{d'}$
Recall that an $\AA^1_S$-module over a $\kk$-scheme $S$ is a morphism $V\to S$ of finite type from a scheme $V$ to $S$  together with a couple of morphisms
\[ +:V\times_S V\longrightarrow V, \quad 0:S\to V, \quad, -:V\to V, \quad m:\Aff^1_S\times_S V \longrightarrow V\]
over $S$ satisfying the axioms of a module with respect to the ring object $\Aff^1_S$ in the category of schemes over $S$. A full subcategory of $\Aff^1_S$-modules is given by affine morphisms $p:V\to S$ which can be described as $\Specbf_S \Sym E \to S$ for $E\subset p_\ast\OO_V$ being the coherent subsheaf of functions which are linear on the fibers. This provides an equivalence between $\Coh(S)^{op}$, the opposite of the category of coherent sheaves on $S$, and the full subcategory of affine $\Aff^1_S$-modules\footnote{Such a structure is sometimes also called an abelian cone.}. The inclusion has a left adjoint $V\mapsto \Specbf_S p_\ast\OO_V\cong \Specbf_S\Sym E$ with $E$, as before, being the subsheaf of $p_\ast\OO_V$ consisting of fiberwise linear functions. However, the unit $V\longrightarrow \Specbf_S p_\ast \OO_V$ of the adjunction does not need to be an isomorphism. Being a right adjoint functor, the inclusion of $\Coh(X)^{op}$ into the category of all $\Aff^1_S$-modules preserves limits, hence 
kernels. Notice that the section functor of $\Specbf_S \Sym E \longrightarrow S$ on the big Zariski site on $S$ completely determines $E$. However, its restriction to the small Zariski site on $S$ does not. The restricted section functor is just the dual sheaf $E^\vee$ of $E$ which for instance vanishes for all torsion sheaves $E$. The subcategory $\Coh(S)^{op}$ of affine $\Aff^1_S$-modules has another full subcategory $\Vect_S$ given by vector bundles corresponding to locally free sheaves in $\Coh(X)^{op}$. An example of an affine $\Aff^1_S$-module is given by $\Ta S:=\Specbf_S \Sym \Omega_S$ which is a vector bundle if  $S$ is smooth.   
\begin{proposition} \label{vector_bundle}
The map $\hat{\pi}_1\times \hat{\pi}_3:X_{d,d'} \longrightarrow X_d\times_\kk X_{d'}$ is an affine $\Aff^1_{X_d\times X_{d'}}$-module  over $X_d\times X_{d'}$. It can be realized as a direct summand of the restriction of the tangent cone $\Ta X_{d+d'}$ to $X_d\times X_{d'}$ embedded into $X_{d+d'}$ via $(E_1,\psi_1)\oplus (E_3,\psi_3)=(E_1\oplus E_3,\psi_1\oplus \psi_3)$. This restriction splits and $X_{d,d'}\oplus X_{d',d}$ can be identified with the normal cone to the inclusion.  The fiber $F$ of $\hat{\pi}_1\times \hat{\pi}_3$ over a $\KK$-point $\big((E_1,\psi_1),(E_3,\psi_3)\big)$  fits into the   exact sequence 
\[0\longrightarrow  \Hom_{\AA_\KK}(E_3,E_1)\longrightarrow \Hom_{\KK}(\KK^{d'},\KK^{d})\longrightarrow F \longrightarrow \Ext^1_{\AA_\KK}(E_3,E_1)\longrightarrow 0 \] 
of $\KK$-vector spaces. 
In particular, the fiber dimension over the point $(E_1,E_3)$ is given by \[ dd'+\dim_\KK\Ext^1_{\AA_\KK}(E_3,E_1)-\dim_\KK \Hom_{\AA_\KK}(E_3,E_1).\]
\end{proposition}
\begin{proof}
To shorten notation, we write $Z:=X_d\times_\kk X_{d'}$. We start by showing that $\hat{\pi}_1\times \hat{\pi}_3:X_{d,d'}\longrightarrow Z$ is an $\Aff^1_Z$-module, i.e.\ there are morphism
\[ X_{d,d'}\times_Z X_{d,d'} \xrightarrow{\,+\,} X_{d,d'},\quad Z\xrightarrow{\,0\,} X_{d,d'}, \quad -:X_{d,d'}\to X_{d,d'},\quad  \Aff^1_Z\times_Z X_{d,d'} \xrightarrow{\;\cdot\;} X_{d,d'} \]
over $Z$ satisfying the axioms of a module over the ring object $\Aff^1_Z$ in the category of schemes over $Z$. 
In order to construct these maps, we use the corresponding moduli functors represented by the spaces in question. The morphism $0:Z\to X_{d,d'}$ maps to pairs $(E_1,\psi_1)$ and $(E_3,\psi_3)$ on a $\kk$-scheme $S$ to their direct sum. Similarly, $+:X_{d,d'}\times_Z X_{d,d'}\longrightarrow X_{d,d'}$ maps to pairs $(0\to E_1\to E_2\to E_3\to 0,\psi_1,\psi_2,\psi_3)$ and $(0\to E_1\to E'_2\to E_3 \to 0,\psi_1,\psi'_2,\psi_3)$ of short exact sequences on $S$ to their sum $(0\to E_1\to \hat{E}_2\to E_3\to 0,\psi_1,\hat{\psi}_2,\psi_3)$ constructed by pull-backs and push-outs along (co)diagonals giving rise the following commutative diagram 
\[ \xymatrix { 0 \ar[r] & E_1\oplus E_1 \ar@{=}[d] \ar[r] & E_2\oplus E_2' \ar[r] & E_3\oplus E_3 \ar[r] & 0 \\
0 \ar[r] & E_1\oplus E_1 \ar[r] \ar[d]_\nabla & \tilde{E}_2 \ar[u] \ar[d] \ar[r] & E_3 \ar[u]_\Delta \ar[r] \ar@{=}[d] & 0 \\
0 \ar[r] & E_1 \ar[r] & \hat{E}_2 \ar[r] & E_3 \ar[r] & 0 }
\]
  with cartesian upper right square and cocartesian lower left square. After applying $\omega_S$, we obtain an isomorphism of similar diagrams in $\Vect^I_S$. Note that the sum of the trivial sequence 
\begin{equation}\label{trivial_sequence} 0 \longrightarrow  \OO_S^{ d} \longrightarrow   \OO^{d}\oplus \OO_S^{d'} \longrightarrow  \OO_S^{d'} \longrightarrow  0   
\end{equation}
with itself is again the trivial sequence. The ``scalar'' multiplication$\Aff^1_\kk \times_\kk X_{d,d'}\longrightarrow X_{d,d'}$ is constructed as follows: given a function $f$ on $S$ and a sequence $(0\to E_1\to E_2\to E_3\to 0,\psi_1,\psi_2,\psi_3)$ with isomorphisms as above, we define a new $S$-valued point  $(0\to E_1\to E'_2\to E_3 \to 0,\psi_1,\psi'_2,\psi_3)$ of $X_{d,d'}$ using the pull-back along $f$ considered as an endomorphism of $E_3$
\[ \xymatrix {0\ar[r] & E_1\ar@{=}[d] \ar[r] & E_2 \ar[r] & E_3 \ar[r] & 0 \\ 
0 \ar[r] & E_1 \ar[r] & E'_2 \ar[u] \ar[r] & E_3 \ar[u]_f \ar[r] & 0. }
\]
After applying $\omega_S$, we obtain an isomorphism of similar diagrams in $\Vect^I_S$, and the pull-back of the trivial sequence along $f$ is again trivial. The inverse of a tuple $(0\to E_1\xrightarrow{i}E_2\xrightarrow{p}E_3\to 0,\psi_1,\psi_2,\psi_3)$ is defined as $(0\to E_1\xrightarrow{i}E_2\xrightarrow{-p}E_3\to 0,\psi_1,\psi_2,\psi_3)$. We leave it to the reader to check the axioms of an $\Aff^1$-module. \\
By the previous Proposition, an $S$-point of the restriction $\Ta X_{d+d'}|_{X_d\times X_{d'}}$ corresponds to a short exact sequence
\[ 0 \longrightarrow E_1\oplus E_3 \longrightarrow E' \longrightarrow E_1\oplus E_3 \longrightarrow 0 \]
and an isomorphism
\[ \xymatrix @C=0.5cm @R=1.5cm { 0 \ar[r] &  \omega_S(E_1)\oplus\omega_S(E_3) \ar[r] \ar[d]^(0.4){\psi_1\oplus\psi_3}(0.4)_\wr & \omega_S(E') \ar[d]^(0.4){\psi'}(0.4)_\wr \ar[r] & \omega_S(E_1)\oplus \omega_S(E_3) \ar[d]^(0.4){\psi_1\oplus\psi_3}(0.4)_\wr \ar[r] & 0 \\ 
0\ar[r] & {\begin{array}{c} \OO_S^d \\ \oplus \\ \OO_S^{d'}\end{array}} \ar[r]^(0.4){i_1} & { \begin{array}{ccc} \OO_S^{d} &\oplus &\OO_S^d \\ & \oplus & \\ \OO_S^{d'} &\oplus &\OO_S^{d'}\end{array}} \ar[r]^(0.6){\pr_2} & {\begin{array}{c} \OO_S^d\\ \oplus \\ \OO_S^{d'}\end{array}} \ar[r] & 0. }\] 
We construct a sequence $0\to E_1\xrightarrow{i} E_2\xrightarrow{p} E_3\to 0$ by means of the following commutative diagram with cartesian upper squares and cocartesian lower squares.
\[ \xymatrix { 0 \ar[r] & E_1\oplus E_3 \ar[r] & E' \ar[r] & E_1\oplus E_3 \ar[r] &0 \\
 0 \ar[r] &E_1 \oplus E_3 \ar[r] \ar@{=}[u] \ar[d]_{\pr_{E_3}} & E'' \ar[u] \ar[r]\ar[d] & E_3 \ar[u]_{i_{E_3}} \ar[r] \ar@{=}[d] & 0 \\
 0 \ar[r] & E_1 \ar[r] & E_2 \ar[r] & E_3 \ar[r] & 0 }
\]
Applying $\omega_S$ to this diagram yields an isomorphism
\[ \xymatrix { 0 \ar[r] & E_1 \ar[d]^{\psi_1}_\wr \ar[r] & E_2 \ar[d]^{\psi_2} \ar[r] & E_3 \ar[d]^{\psi_3} \ar[r] & 0 \\
 0 \ar[r] & \OO_S^d \ar[r]^(0.4){\iota_1} & \OO_S^{d}\oplus \OO_S^{d'} \ar[r]^(0.6){\pr_2} & \OO_S^{d'} \ar[r] & 0 }\] 
A right inverse for this construction is given by the sequence 
\[ \xymatrix @C=3cm { 0 \to  E_1 \oplus E_3 \ar[r]^{\left(\begin{array}{cc} \scriptstyle  \id_{E_3} & \scriptstyle  0 \\ \scriptstyle  i & \scriptstyle  0 \\ \scriptstyle  0 & \scriptstyle \id_{E_3}\end{array}\right)} & E_1 \oplus E_2\oplus E_3  \ar[r]^{\left(\begin{array}{ccc} \scriptstyle  \id_{E_1} & \scriptstyle  0 &\scriptstyle  0 \\ \scriptstyle  0 & \scriptstyle  p &\scriptstyle  \id_{E_3} \end{array}\right) } & E_1\oplus E_3 \to 0. }\] 
which is the sum of $0\to E_1 \xrightarrow{i} E_2 \xrightarrow{p} E_3 \to 0$ with $0\to E_3 \xrightarrow{\id} E_3\to 0 \to 0 $ and $0\to 0 \to E_1 \xrightarrow{\id} E_1 \to 0 $.
These maps are $\OO_S(S)$-linear, showing that $X_{d,d'}$ is indeed a direct  summand of the affine $\Aff^1_Z$-module $\Ta X_{d+d'}|_Z$. As mentioned earlier, the kernel of an idempotent linear map on the affine $\Aff^1_Z$-module $\Ta X_{d+d'}|_Z$ must also be affine. Thus $X_{d,d'}\to Z$ is an affine $\Aff^1_Z$-module of the form $\Specbf_S\Sym E$ with $E$ denoting the coherent sheaf of fiberwise linear functions. By projecting to the other components and Corollary \ref{tangent_bundle}, we obtain the full splitting of $\Ta X_{d+d'}$ restricted to $X_d\times X_{d'}$.

In the final step we want to determine the fiber $F$ of $\hat{\pi}_1\times \hat{\pi}_3$ over a $\KK$-point of $Z$ corresponding to a pair $\big((E_1,\psi_1),(E_3,\psi_3)\big)$ with $E_1,E_3\in \AA_\KK$ for any field extension $\KK\supset \kk$. Every element $(0\to E_1 \to E_2\to E_3 \to 0,\psi_1,\psi_2,\psi_3)$ of the fiber maps to $\Ext^1_{\AA_\KK}(E_3,E_1)$ by considering the isomorphism class of the underlying sequence $0\to E_1\to E_2\to E_3\to 0$. For  any exact sequence $0\to E_1\to E_2\to E_3\to 0$ in $\AA_\KK$, the resulting sequence $0\to \omega_\KK(E_1)\to \omega_\KK(E_2) \to \omega_\KK(E_3)\to 0 $ of $\KK$-vector spaces splits, and by choosing a basis of $\omega_\KK(E_1)$ and $\omega(E_3)$ we see that the $\KK$-linear map from the fiber to $\Ext^1_{\AA_\KK}(E_3,E_1)$ is onto. If $(0\to E_1\to E_2\to E_3 \to 0,\psi_1,\psi_2,\psi_3)$ is in the kernel of the $\KK$-linear map 
\[ F=(\hat{\pi}_1\times\hat{\pi}_3)^{-1}\big((E_1,\psi_1),(E_3,\psi_3)\big)\longrightarrow  \Ext^1_{\AA_\KK}(E_3,E_1),\]
there must be an isomorphism
\[ \xymatrix { 0\ar[r] & E_1 \ar@{=}[d] \ar[r] & E_1\oplus E_3 \ar[d]^\wr_\phi \ar[r] & E_3 \ar@{=}[d] \ar[r] & 0 \\
0 \ar[r] & E_1 \ar[r] & E_2 \ar[r]  & E_3\ar[r] & 0 }
\]
of short exact sequences in $\AA_\KK$. Then, $(\id,\psi_2\omega_\KK(\phi)(\psi_1\oplus\psi_3)^{-1},\id)$ is an automorphism of the trivial sequence (\ref{trivial_sequence}) fixing the outer terms. The group of automorphisms of the trivial sequence fixing the outer terms can be identified with $\Hom_\KK(\KK^{d'},\KK^{d})$.
However, $\phi$ and therefore also $\psi_2\omega_\KK(\phi)(\psi_1\oplus\psi_3)^{-1}$ is only unique up to some morphism in 
\[ \xymatrix{ \Hom_{\AA_\KK}(E_3,E_1)\ar@{^{(}->}[r]^(0.5){\omega_\KK} & \Hom_\KK( \KK^{d'}, \KK^{d}) } \]
On the other hand, given an automorphism $(\id,\psi,\id)$ of the trivial sequence fixing the outer terms, $(0\to E_1 \to E_1\oplus E_3 \to E_3\to 0,\psi_1,\psi(\psi_1\oplus\psi_3),\psi_3)$  is another point in the the kernel,  and this point is  isomorphic to the zero element $(0\to E_1\to E_1\oplus E_3,\to E_3,\psi_1,\psi_1\oplus\psi_3,\psi_3)$ of the fiber if and only if $\psi$ is in the subgroup corresponding to $\Hom_{\AA_\KK}(E_3,E_1)$. Thus, the kernel of the surjective $\KK$-linear map from the fiber to $\Ext^1_{\AA_\KK}(E_3,E_1)$ is  
\[ \label{quotient2} \Hom_\KK(\KK^{d'}, \KK^{d})/\Hom_{\AA_\KK}(E_3,E_1).\]
\end{proof}

\begin{remark}\rm
By slightly extension of the previous proof one can show  that the universal Grassmannian must be representable. Thus, in the presence of assumption (6), assumption (5) reduces to requirement that the universal Grassmannian is proper. In fact, we will make no use of the properness in this paper, but it will be important in \cite{DavisonMeinhardt3} and \cite{DavisonMeinhardt4}.  
\end{remark}

The following construction of triples $(\AA,\omega,p)$ satisfying assumptions (1)--(6) will be used in the next subsection, at least in a special case which is very important for the remaining part of the paper.

\begin{proposition} \label{fiber_product}
Consider the diagram 
\[ \xymatrix {  \AA'\times_{\AA'''} \AA'' \ar[dd]_\pr \ar[rr]^\pr \ar[dr]_{\omega} && \AA' \ar[dl]^{\omega'} \ar[dd]^{\phi'} \\ & \Vect^I & \\ \AA'' \ar[ur]^{\omega''} \ar[rr]_{\phi''} & & \AA''' \ar[ul]_{\omega'''} } \]
with exact transformations $\phi', \phi''$. Then $\phi'$ and $\phi''$ are faithful and restricted to the moduli stacks, they are also representable and induce morphisms $\phi'_d:X'_d \to X'''_d$ and $\phi''_d:X''_d\to X'''_d$ for every $d\in \NN^{\oplus I}$. Assume that $\phi'_d$ and $\phi''_d$ extend to the  partial compactifications  compatible with the projective morphisms. That is, we require a commutative diagram 
\[ \xymatrix { 
X'_d \ar[r]^{\phi'_d} \ar@{^{(}->}[d] & X'''_d \ar@{^{(}->}[d] & X''_d \ar[l]_{\phi''_d} \ar@{^{(}->}[d] \\ 
\overline{X}'_d \ar[d] \ar[r]^{\overline{\phi}'_d} & \overline{X}'''_d \ar[d] & \ar[l]_{\overline{\phi}''_d} \ar[d] \overline{X}''_d \\ \Spec A'_d \ar[r] & \Spec A'''_d & \ar[l] \Spec A''_d  }\]
for every $d\in \NN^{\oplus I}$.  Moreover, the diagram
\[ \xymatrix { \AA'_{S\otimes R} \ar[r]^{\phi'_{S\otimes R}} \ar[d]_{p'^R_S}^{\sim} & \AA'''_{S\otimes R} \ar[d]_{\wr}^{p'''^R_S} \\ \AA'^R_S \ar[r]^{\phi^R_S} & \AA'''^R_S }\]
and a similar one for $\phi''$ should commute for all $\kk$-schemes $S$ and all local Artin $\kk$-algebras $R$, where $\phi'^R_S$ is the natural lift of $\phi'_S$ to $R$-objects mapping $(E,\nu)$ to $(\phi'_S(E),\phi'_S\nu)$.
If $(\AA',\omega',p'), (\AA'',\omega'',p''),(\AA''',\omega''',p''')$ satisfy all assumptions (1)--(6), then $\AA:=\AA'\times_\AA \AA''$  together with the (up to isomorphism) naturally defined $\omega:\AA\to \Vect^I$ and $p'\times_{p'''}p''$ will also satisfy (1)--(6) except maybe for the existence of a moduli space $\Msp_d$ for $\AA$. If the inclusion  \[ \overline{X}'^{ss}_d\times_{\overline{X}'''^{ss}_d}\overline{X}''^{ss}_d\subseteq (\overline{X}'_d\times_{\overline{X}'''_d}\overline{X}''_d)^{ss},\]                                                                                                                                                                                                                                                                                                                                                              
where the open subscheme on the right hand side is taken with respect to the $G_d$-linearization $\Lin'^a_d\boxtimes \Lin''^b_d$ for natural numbers $a,b>0$, is the identity, then $\Msp_d$ exists.                                                                                                                                                                                                                                                                                                                                                 

\end{proposition}
\begin{proof} Faithfulness of $\phi'$ and $\phi''$ follow from the faithfulness of $\omega'$ and $\omega''$. Using $\omega'''\phi'=\omega'$, we conclude 
\[ \Mst'_d\times_{\Mst'''_d} X'''_d \cong  \Mst'_d\times_{\Mst'''_d}(\Mst'''_d \times_{\Spec \kk/G_d}\Spec \kk)  \cong \Mst'_d \times_{\Spec \kk/G_d}\Spec \kk \cong  X'_d \] proving that $\phi'$ and lifts to $\phi'_d:X'_d\to X'''_d$. The same must hold for $\phi''_d$. 
Using the notation introduced so far, we get $\Mst_d=\Mst'_d\times_{\Mst'''_d} \Mst''_d=X_d/G_d$ with $X_d:=X'_d\times_{X'''_d} X''_d$. The latter space has an embedding into $\overline{X}'_d\times_{\overline{X}'''_d} \overline{X}''_d$ which maps projectively to $\Spec A'\times_{\Spec A'''}\Spec A''$ with relative ample $G_d$-linearization $\Lin'^a_d\boxtimes \Lin''^b_d$ for some natural numbers $a,b>0$. We leave it to the reader to check all requirements of assumption (4) except for the existence of a moduli space $\Msp_d$. Under the additional assumption 
\[ \overline{X}'d^{ss}_d\times_{\overline{X}'''^{ss}_d}\overline{X}''^{ss}_d=(\overline{X}'_d\times_{\overline{X}'''_d}\overline{X}''_d)^{ss}\]
we obtain a map 
\[m: (\overline{X}'_d\times_{\overline{X}'''_d}\overline{X}''_d)^{ss}/\!\!/G_d \longrightarrow \overline{X}'^{ss}_d/\!\!/G_d \times_{\overline{X}'''^{ss}_d/\!\!/G_d}\overline{X}''^{ss}_d/\!\!/G_d, \]
of GIT-quotients, and $\Msp_d$ is the preimage of $\Msp'_d\times_{\Msp'''_d}\Msp''_d$ which is a subscheme of the target of $m$.
To check assumption (5), we use Lemma \ref{lemma3} and get $Y_{d,d'}=\pr'^\ast Y'_{d,d} \cap \pr''^\ast Y''_{d,d'} \longrightarrow X_{d+d'}$, where the intersection can be  taken in $X_{d+d'}\times \Gr_d^{d+d'}$. Hence, $Y_{d,d'}\to X_{d+d'}$ is still representable and proper. Finally, the conditions on $p',p'',p'''$ ensure that $p'\times_{p'''}p''$ satisfies all requirements of assumption (6) because of $\AA'^R_S\times_{\AA'''^R_S}\AA''^R_S\cong (\AA'_S\times_{\AA''_S}\AA'''_S)^R$ for every $\kk$-scheme $S$ and every local Artin $\kk$-algebra $R$.   We encourage the reader to fill in the details.     
\end{proof}

\subsection{Smoothness and symmetry}

Let us come to our final two assumptions which turn out to be crucial for the rest of the paper.\\
\\
\textbf{(7) Smoothness assumption:} The schemes $X_d$ are locally integral, and the quantity $\dim_\KK\Hom_{\AA_\KK}(E,F)- \dim_\KK \Ext^1(E,F)$ is locally constant on $\Mst\times \Mst$. \\

\begin{corollary}
 The $\kk$-schemes $X_d$ are smooth and the affine $\Aff^1$-module $\hat{\pi}_1\times \hat{\pi}_3: X_{d,d'} \longrightarrow X_d \times X_{d'}$ is a vector bundle.  
\end{corollary}
\begin{proof}
 Because of Corollary \ref{tangent_bundle}, $\Ta X_d$ and its sheaf $\Omega_{X_d}$ of fiberwise linear functions has constant fiber direction. By \cite{Hartshorne}, chap.\ 2, Lemma 8.9, $\Omega_{X_d}$ must be locally free which implies smoothness of $X_d$. The same argument shows that the affine $\Aff^1$-module $X_{d,d'}$ is a vector bundle on $X_d\times X_{d'}$
\end{proof}

\begin{corollary}
The number 
\[(E_1,E_2):=\dim_\KK \Hom_{\AA_\KK}(E_1,E_2)-\dim_\KK\Ext^1_{\AA_\KK}(E_1,E_2)\] 
depends only on the simple factors of $E_1$ and $E_2$. Thus, it defines a symmetric pairing $(-,-)$ on $\Ka_0(\AA_\KK)$. 
\end{corollary}
\begin{proof}
Using Luna's \'etale slice theorem (cf.\ Theorem \ref{Luna}) and Proposition \ref{normal_ext}, we can find a deformation of any extension $E_1$ of two objects $E'_1,E''_1$ into the trivial deformation $E'_1\oplus E''_1$. By the previous corollary $(E_1,E_2)=(E'_1\oplus E'_1,E_2)=(E'_1,E)+(E''_1,E)$ for every object $E_2\in \AA_\KK$, and similarly for the second entry. Thus, $(-,-)$ descends to $\Ka_0(\AA_\KK)$.
\end{proof}
This result is another justification of the philosophy that $\AA_\KK$ looks like a category of homological dimension for which $(-,-)$ is the Euler pairing of $\AA_\KK$.  
Let $\Mst=\sqcup_{\lambda\in \Lambda} \Mst_\lambda$ be the decomposition into connected components.  Thus, we obtain a well-defined map $\gamma:\Lambda \to \Gamma$ to the lattice $\Gamma=\Ka_0(\AA_{\bar{\kk}})/\rad(-,-)$ by $\gamma(\lambda)=\cl(E)\mod \rad(-,-)$ for some $E\in \Mst_\lambda(\bar{\kk})$. Hence, we also get a decomposition $\Mst=\sqcup_{\gamma\in \Gamma}\Mst_\gamma$ and similarly $\Msp^{s}=\sqcup_{\gamma\in \Gamma}\Msp^{s}_\gamma$ with equidimensional (not necessarily connected or non-empty) smooth stacks $\Mst_\gamma=\sqcup_{\gamma(\lambda)=\gamma} \Mst_\lambda$ and  $\Msp^{s}_\gamma=\sqcup_{\gamma(\lambda)=\gamma} \Msp^{s}_\lambda$ of dimension $-(\gamma,\gamma)$ and $1-(\gamma,\gamma)$ respectively. Combining this with the decomposition into dimension vectors, we finally get $\Msp=\sqcup_{\gamma\in \Gamma,d\in \NN^{\oplus I}} \Msp_{\gamma,d}$ with $\Msp_{\gamma,d}$ parametrizing all semisimple objects of class $\cl(E)=\gamma$ and dimension vector $\dim(E)=d$, and similarly for $\Msp^{s}$.
\begin{example} \rm
Consider the one-loop quiver. The pairing $(-,-)$ is zero and $\Gamma=0$ follows. Nevertheless, $\Msp$ has a decomposition using dimension vectors with $\Msp_d=\Sym^d \Aff^1_\kk=\Aff^d_\kk$. Thus, even though $E\to \dim E$ is an additive function, it might not be a function of $\gamma\in \Gamma$. On the other hand, if we forget the $Q_0$-grading of a quiver representation such that $\omega_\KK(E)$ is just the total underlying vector space, $\Msp^{s}_d$ might contain components of different dimension as $|d|=\sum_{i\in Q_0}d_i$ does not determine $d=(d_i)_{i\in I}$ or the Euler pairing $(d,d)$. Therefore, it is important to have both indices at hand.  
\end{example} 

Our final assumption cannot be overestimated in Donaldson--Thomas theory and boils down to a genericity assumption on the (hidden) stability condition. \\
\\
\textbf{(8) Symmetry assumption:} For all field extensions $\KK\supset \kk$  the pairing $(-,-)$ on $\Ka_0(\AA_\KK)$ is symmetric, i.e.\ 
\[\dim_\KK \Hom_{\AA_\KK}(E,F)-\dim_\KK\Ext_{\AA_\KK}^1(E,F)=\dim_\KK\Hom_{\AA_\KK}(F,E)-\dim_\KK \Ext_{\AA_\KK}^1(F,E).\]
for all objects $E,F\in \AA_\KK$.

\begin{example} \rm If $\AA_\KK$ is the heart of bounded t-structure on a 3-Calabi--Yau category $\mathcal{T}$, then this condition is just saying that the Euler pairing $\langle E,F \rangle$ of $\mathcal{T}$ vanishes. 
\end{example}
 
\begin{example}\rm
For a quiver $Q$ and a generic Bridgeland stability condition $\zeta$, the symmetry condition is fulfilled for the category of semistable representations of a fixed slope. To construct a $G_d$-linearization on $X^{\zeta-ss}_d$, one can proceed as follows. One can find a neighbourhood $U$ of $\zeta$ such that $X^{\zeta-ss}_{d'}=X^{\zeta'-ss}_{d'}$ for all $d'\le d$ and all $\zeta'\in U$. In particular, we can assume that $\zeta$ has rational real and imaginary parts. Rescaling $\zeta$ will also not change semistability and we may assume $\zeta\in \ZZ[\sqrt{-1}]^{Q_0}$. Another modification not changing $X^{\zeta-ss}_{d'}$ for all $d'\le d$ can be done to ensure $\Im m \,\zeta_i=1$ for all $i\in Q_0$, in other words, $\zeta$ is a King stability condition providing a $G_d$-linearization.  
\end{example} 
\begin{example} \rm For a smooth projective curve $C$ and a number $\mu\in (-\infty,\infty]$ the symmetry condition is fulfilled as semistable sheaves $E$ of the same slope have collinear numerical data $(\deg(E),\rk(E))$. The pairing $(-,-)$ is just the Euler pairing depending only on the ranks and the degrees of the sheaves $E_1,E_2$. 
\end{example}
\begin{example} \rm
Let $X$ be a smooth projective surface with a real ample divisor $D$. We start with two dimensional sheaves and fix a normalized polynomial $p\in \QQ[t]$ of degree 2. Let us  consider the category $\AA_\kk$ of Gieseker semistable coherent sheaves on $X$ with normalized Hilbert polynomial $p$. If $K_X\cdot D<0$, the sheaf $E(K_X)$ has  a  normalized Hilbert polynomial smaller than $p$. Then, $\Ext^2(E_1,E_2)\cong \Hom(E_2,E_1(K_X))^\vee=0$ by general arguments about morphisms between semistable sheaves of different normalized Hilbert polynomial. Thus, $\dim\Hom(E_1,E_2)-\dim\Ext^1(E_1,E_2)$ is the Euler paring of $E_1$ with $E_2$ which depends only on the Chern characters of $E_1$ and $E_2$. Thus, condition (7) is fulfilled. If we choose $D$  generic (depending on $p$), the Chern characters of $E_1$ and $E_2$ are collinear proving the symmetry assumption. In the case of 1-dimensional sheaves $E$ we need to assume $c_1(E)\cdot K_X<0$ to ensure that the same arguments apply. This condition is for example  fulfilled on  del Pezzo surfaces. Note that the symmetry condition is not fulfilled for zero-dimensional 
sheaves. By Serre duality $\dim\Ext^1(E_1,E_2)$ is symmetric but $\dim \Hom(E_1,E_2)$ is not. Consider for example the sheaves $E_1=\OO_{X,x}/\mathfrak{m}_x\cong \kk$ and $E_2=\OO_{X,x}/\mathfrak{m}_x^2$ for some $\kk$-point $x\in X$. Then $\dim\Hom(E_1,E_2)=2$ but $\dim\Hom(E_2,E_1)=1$.
In order to construct $G_d$-linearizations, we proceed as in the quiver case by modifying $D$ to some integral ample divisor not changing the stack of semistable sheaves of a finite set of Chern characters. As the GIT-quotient is universal, the quasiprojective scheme $\Msp_d$ is independent of any choices made.  
\end{example}
\begin{example}\rm
Let $(\AA,\omega,p)$ satisfy assumptions (1)--(8) and let $J$ be any set, then $\AA\times \Vect^J$ with the  functor $\omega\times \id:\AA\times \Vect^J\longrightarrow \Vect^I\times \Vect^J\cong \Vect^{I\sqcup J}$ will also satisfy  assumptions (1)--(8). 
\end{example}

\begin{corollary}
Assuming the symmetry condition (8), the smoothness assumption (7) is equivalent to the weaker requirement that all spaces $X_d$ are smooth.
\end{corollary}
\begin{proof} As $\dim_\KK\Ext^1_{\AA_\KK}(E,E)-\dim_\KK\Hom_{\AA_\KK}(E,E)$ computes the dimension of the smooth stack $\Mst$ in $E$, it must be constant on connected components. Applying this to $E=E,F$ and $E\oplus F$, the first statement follows from the symmetry assumption and the formula
\begin{eqnarray*} \lefteqn{\hom(F,E)-\ext^1(F,E)=}\\ &&\frac{1}{2}\Big(\hom(F\oplus E,F\oplus E)-\ext^1(F\oplus E,F\oplus E) \\ && -\hom(F,F)+\ext^1(F,F) -\hom(E,E)+\ext^1(E,E)\Big). 
\end{eqnarray*}
using the shorthands $\hom=\dim_\KK\Hom_{\AA_\KK}$ and $\ext^1=\dim_\KK\Ext^1_{\AA_\KK}$. 
\end{proof}

\begin{lemma} Let $P$ be a property of objects in $\AA_\KK$ which is closed under subquotients and extension, that means for every short exact sequence
\[ 0\longrightarrow E_1 \longrightarrow E_2 \longrightarrow E_3\longrightarrow 0 \]
in $\AA_\KK$, $E_2$ has property $P$ if and only if $E_1$ and $E_3$ have property $P$. In particular, the full subcategory $\AA^P_\KK\subset \AA_\KK$ of objects having property $P$ is a Serre subcategory. We also assume that the points in $\Msp$ having property $P$ are the points of a subscheme $\Msp^P$. We define $\AA^P_S$ to be the full subcategory of objects $E$ in $\AA_S$ such that $E|_s$ is in $\AA^P_\KK$ for every point $s\in S$ with residue field $\KK$. Let $\omega^P$ be the restriction of $\omega$ to $\AA^P$. Then, $p_S^R:\AA_{S\otimes R}\to \AA_S^R$ induces also an equivalence $p^P$ of the corresponding subfunctors, and the triple $(\AA^P,\omega^P,p^P)$ satisfies the same assumptions as $(\AA,\omega,p)$.  
\end{lemma}
The proof is straight forward and left to the reader. Examples of such a  situation are given in Example \ref{property2}.

\subsection{Framed objects}
Framed objects occur  as a special case of Proposition \ref{fiber_product}. Let $(\AA,\omega,p)$ satisfy our assumptions (1)--(6). For $\hat{I}=I \sqcup\{\infty\} $ 
and $f\in \NN^I$ we consider the quiver $Q^{Fr}_f$ with vertex set $\hat{I}$ and $f_i$ arrows from $\infty$ to $i\in I$. Let $\omega^{Fr}_f:Q^{Fr}_f-\rep \to \Vect^{\hat{I}}$ be the usual functor forgetting the $\kk Q^{Fr}_f$-action. Given $(A,\omega,p)$ as above, we can form the cartesian diagram 
\[ \xymatrix {  \AA_f \ar[dd]_\pr \ar[rr]^\pr \ar[dr]_{\omega_f} && Q^{Fr}_f-\rep \ar[dl]^{\omega^{Fr}_f} \ar[dd]^{\omega^{Fr}_f } \\ & \Vect^{\hat{I}} & \\ \AA\times \Vect \ar[ur]^{\omega\times \id} \ar[rr]_{\omega\times \id} & & \Vect^{\hat{I}}. \ar[ul]_{\id} } \]
Thus, $(\AA_f,\omega_f,p_f)$ also satisfies our assumptions (1)--(6). Let us spell out some details relevant for the paper by providing another equivalent model for $(\AA_f,\omega_f,p_f)$ which will be used throughout this paper.\\
For a $\kk$-scheme $S$, objects of $\AA_{f,S}$ are triples $(E,W,h)$ with $E\in \AA_S$, $W\in \Vect_S$ and a morphism $h:W^f:=\bigoplus_{i\in I}W^{\oplus f_i}\longrightarrow \omega_S(E)$ of $I$-graded vector bundles on $S$. For $S=\Spec\KK$ and $W=\KK^r$, $h$ is just a collection of $rf_i$ vectors in the $i$-the component $\omega_\KK(E)_i$ of $\omega(E)$ for $i\in I$. A morphism $\hat{\phi}:(E,W,h)\longrightarrow (E',W',h')$ between such triples is a pair $(\phi,\phi_\infty)$ consisting of a morphism $\phi:E\to E'$ in $\AA_S$ and a morphism $\phi_\infty:W\to W'$ of vector bundles such that 
\[ \xymatrix @C=3cm{ W^f \ar[d]_h \ar[r]^{\phi_\infty^f} & W'^f \ar[d]^{h'}\\ \omega_S(E) \ar[r]^{\omega_S(\phi)} & \omega_S(E')}\]
commutes. It is obvious that $\AA_f$ is an exact category which is abelian for $S=\Spec\KK$ by our assumptions on $\AA$. We define $\omega_f$ by means of $\omega_{f,S}(E,W,h):=\omega(E)\oplus W$ and leave the definition of $p_f$ to the reader.
\begin{example}\rm If $\AA_\kk=\kk Q-\rep^{\zeta-ss}_\mu$ is the category of $\zeta$-semistable quiver representations of slope $\mu$, then $\AA_{f,\kk}$ is the category  of representations of the quiver $Q_f$ defined in  section 2 such that the underlying $Q$-representation is $\zeta$-semistable of slope $\mu$. If $\AA_\kk$ is the category of Gieseker semistable sheaves $E$ of fixed normalized Hilbert polynomial on a smooth polarized projective variety  $X$, then $\AA_{f,\kk}$ is essentially the category of such sheaves equipped with a bunch of  sections of the twisted sheaf $E(m)$.   
\end{example}
Apparently, $\AA_S$ and $\Vect_S$ can be identified with full subcategories of $\AA_{f,S}$ for every $\kk$-scheme $S$. 
The moduli stack is given by $\Mst_{f,\hat{d}}=X_{f,\hat{d}}/G_{\hat{d}}$ with $\hat{d}=(d,r)$ and
\[ X_{f,\hat{d}}=X_d \times_\kk \prod_{i\in I} \Hom_\kk(\kk^r,\kk^{d_i})^{ f_i}=X_d\times_\kk \Aff_\kk^{r(f\cdot d)}\]
with $f\cdot d=\sum_{i\in I}f_id_i$. Using the trivial $G_{\hat{d}}$-linearization on $\Aff^{r(fd)}_\kk=(\Aff^{r(fd)}_\kk)^{ss}$, we get $\overline{X}_{f,\hat{d}}^{ss}=\overline{X}_d^{ss}\times \Aff^{r(fd)}$ and according to Proposition \ref{fiber_product}, $X_{f,\hat{d}}$ has a moduli space which is actually isomorphic to $\Msp_d$ for every polarization $\Lin_d^a\boxtimes \OO_{\Aff^{r(fd)}_\kk}$ with $a>0$. Indeed, using the projection formula we get 
\begin{eqnarray*} \lefteqn{ \bigoplus_{k\ge 0} \Ho^0(\overline{X}^{ss}_d\times \Aff^{r(fd)}_d, \Lin_d^{\otimes ak}\boxtimes \OO_{\Aff^{r(fd)}_\kk} )^{G_{\hat{d}}} } \\
 &=& \bigoplus_{k\ge 0} \Big( \Ho^0(\overline{X}^{ss}_d, \Lin_d^{\otimes ak})\otimes_\kk\kk[\Aff^{r(fd)}_\kk] \Big)^{G_{d}\times \Gl(r)} \\
  &=& \bigoplus_{k\ge 0} \Big(\Ho^0(\overline{X}^{ss}_d, \Lin_d^{\otimes ak})\otimes_\kk\kk[\Aff^{r(fd)}_\kk]^{\Gl(r)} \Big)^{G_{d}}  \\
   &=& \bigoplus_{k\ge 0} \Ho^0(\overline{X}^{ss}_d, \Lin_d^{\otimes ak})^{G_{d}}. 
\end{eqnarray*}
Thus $\overline{X}_{f,\hat{d}}^{ss}/\!\!/G_{\hat{d}}\cong \overline{X}^{ss}/\!\!/G_d$, and the claim follows.

\subsection{The Hilbert--Chow morphism}
Let us modify the previous example by replacing $\kk Q^{Fr}_f-\rep$ with the subcategory $\kk Q^{Fr}_f-\rep^{\hat{\theta}}_0$ of semistable $Q^{Fr}_f$-representations of slope $0$ with respect to the King stability condition $\hat{\theta}=(\theta,-\theta d)\in \ZZ^{\hat{I}}$ with $\theta_i>0$ for all $i\in I$. Denote the fiber product $(\AA\times\Vect)_{\Vect^{\hat{I}}} \kk Q^{Fr}_f-\rep^{\theta}_0$ with $\AA_{f}^{\hat{\theta}}$.  
Again we use the partial compactification $(\Aff^{r(fd)}_\kk)^{ss}\subset \Aff^{r(fd)}_\kk$ to compute moduli spaces. From now on, we will only consider the case $\dim W\le 1$, i.e.\ $\hat{d}=(d,0)$ or $(d,1)$. Let us simplify our notation slightly by using the subscript ``$f,d$'' for ``$f,(d,1)$''  and ``$d$'' for ``$f,(d,0)$''. \\
To construct  quotients, we have to study the $G_{\hat{d}}$-action on $\overline{X}_{f,d}=\overline{X}_d\times \Aff^{fd}_\kk$ for $G_{\hat{d}}=G_{(d,1)}=G_d\times \Gl(1)$.  Let $\OO^\theta$ be the trivial line bundle on $\Aff^{fd}_k$ equipped with the $G_{(d,1)}$-linearization given by $(g_i)_{i\in \hat{I}} \longmapsto \prod_{i\in \hat{I}} \det(g_i)^{\theta_i}\in \Gl(1)$ with $\theta_\infty=-\theta d$. Then, $(\Aff^{fd}_\kk)^{ss}=(\Aff^{fd}_\kk)^{st}$ parametrizes trivialized $Q_f^{Fr}$-representations $(V\cong \kk^d,W\cong \kk_\infty,h)$ which are generated by $1\in \kk_\infty$.  
\begin{proposition} \label{framed_reps_2}
Let  $\theta_i>0$ for all $i\in I$ as above and $a>\theta d$. Consider the $G_d$-linearization $\Lin^{a,\theta}:=\Lin_d^{\otimes a}\boxtimes \OO^{\theta}$ on $\overline{X}_{f,d}=\overline{X}_d\times \Aff^{fd}_\kk$ and form the corresponding subschemes of (semi)stable points. Then $\overline{X}^{ss}_{f,d}$ maps into $\overline{X}_d^{ss}$ under the projection $\overline{X}_{f,d}\to \overline{X}_d$. Denote the preimage inside $\overline{X}_{f,d}^{ss}$ of the  subscheme $X_d\subset  \overline{X}^{ss}_d$ with $X^{ss}_{f,d}\subset\overline{X}^{ss}_{f,d}$ and similarly for $X^{st}_{f,d}\subset \overline{X}_{f,d}^{st}$. Then $X^{st}_{f,d}= X^{ss}_{f,d}$, and this scheme  parametrizes quadruples  $\hat{E}=(E,\psi,\KK,h)$ (up to isomorphism) with $E\in\AA_\KK$, $\psi:\omega_\KK(E)\xrightarrow{\sim}\KK^d$ and $h:\KK\to \KK^d$ such that $\im(h)\subset \psi(\omega_\KK(E'))$ implies $E'=E$ for each subobject $E'\subset E$, in other words $\hat{E}$ contains no proper subobject with dimension vector $(d',1)$. In particular, it is independent of the choice of $\theta$ as above 
and $a>\theta d$.  
\end{proposition}
\begin{proof} We want to apply Mumford's criterion to $\overline{X}_d\times \Aff^{fd}_\kk$ which comes with a the projective morphism $\overline{X}_d\times \Aff^{fd}_\kk \longrightarrow \Spec A_d \times \Aff^{fd}_\kk$ with (relative) ample line bundle $\Lin^{a,\theta}=\Lin_d^a\boxtimes \OO^{\theta}_{\Aff^{fd}_\kk}$. This can be done due to Theorem 3.3 of \cite{GHH}. Note that  $\mu^{\Lin}(x,\lambda)$ is additive in $\Lin$ and has integer values. We first show that $\hat{x}=(x,h)\in \overline{X}_{f,d}^ss$ implies $x\in \overline{X}_d^{ss}$ by taking a 1-PS $\lambda$ and assume the existence of the limit point $x_0^\lambda$ since otherwise $\mu^{\Lin_d}(x,\lambda)=+\infty>0$ by definition. Using equation (\ref{eq7}), it remains to show $\mu^{\Lin_d}(x,\lambda^{W^n})\ge 0$ for every $n\in \ZZ$ in order to conclude $\mu^{\Lin_d}(x,\lambda)\ge 0$. Here $W^\bullet$ is the filtration on $\KK^d$, where $\KK$ denotes a field extension over which $\lambda$ and $\hat{x}$ are defined. But
\[ 0\le \mu^{\Lin^{a,\theta}}(\hat{x},\lambda^{W^n})= \mu^{\Lin^a_d}(x,\lambda^{W^n}) + \mu^{\OO^\theta}(h,\lambda^{W_n})\le  \mu^{\Lin^a_d}(x,\lambda^{W^n}) + \theta d^n  \]
for any $n\in \ZZ$. If $\mu^{\Lin_d}(x,\lambda^{W^n})<0$, $\mu^{\Lin^a_d}(x,\lambda^{W^n})\le -a<  -\theta d\le -\theta d^n$ producing a contradiction. Thus, $x\in \overline{X}^{ss}_d$. 
Denote the set of framed objects $\hat{E}$ having no subobject of dimension vector $(d,1)$ with $U$. To show $U\subset X^{st}_{f,d}$ we pick a point $\hat{x}=(x,h)\in U$ defined over $\KK$ and a 1-PS $\lambda$ in $G_{\hat{d},\KK}$ not mapping into the diagonal $\Gl(1)$ and  inducing a filtration $\hat{W}^n=W^n\oplus W^n_\infty$ on $\KK^d\oplus \KK_\infty$ such that the limit $(\hat{x}^\lambda_0)=(x^\lambda_0,h^\lambda_0)\in \overline{X}_d\times \Aff^{fd}_\kk$ exists since otherwise $\mu^{\Lin^{a,\theta}}(\hat{x},\lambda)=+\infty>0$ by definition. Using equation (\ref{eq7}) and $\KK_\infty\subset \hat{W}_0$  we conclude  
\begin{eqnarray*} \mu^{\Lin^{a,\theta}}(\hat{x},\lambda)&=&\mu^{\Lin_d^a}(x,\lambda)+\mu^{\OO^{\theta}}(h,\lambda) \\
&=& \sum_{n>0} \Big(\mu^{\Lin^a_d}(x,\lambda^{W^n}) + \theta d^n\Big) + \sum_{n\le  0} \Big( \mu^{\Lin^a_d}(x,\lambda^{W_n}) + \theta d^n- \theta d\Big). 
\end{eqnarray*} 
The first summands corresponding to $W^n_\infty=0$ are non-negative as $x\in \overline{X}_d^{ss}$ implies $\mu^{\Lin^a_d}(x,\lambda^{W^n})\ge 0$, and strictly positive if $W^n\not= 0$. Also the summands of the second sum corresponding to $W^n_\infty=\KK_\infty$ are non-negative and strictly positive if $W^n\not= \KK^{d}$. Indeed, if $\theta d^n <\theta d$, then $\mu^{\Lin^a_d}(x,\lambda^{W^n})\ge (\theta d)\mu^{\Lin_d}(x,\lambda^{W^n})\ge \theta d$ since otherwise $\mu^{\Lin_d}(x,\lambda^{W^n})=0$ as we always have $\mu^{\Lin_d}(x,\lambda^{W^n})\ge 0$. But $\mu^{\Lin_d}(x,\lambda^{W^n})=0$ implies $\psi(W^n)=\omega_\KK(E^n)$ for some proper subobject $E^n\in \AA_\KK$ by our equivalent conditions. In particular, $\omega_\KK(E^n)$ contains the image of $h$ which is impossible by assumption $\hat{x}\in U$. Thus, every summand in the second sum is $\ge \theta d^n$. Since at least one of the $W^n$ is a proper subspace, we obtain $\mu^{\Lin^{a,\theta}}(\hat{x},\lambda)>0$ proving $\hat{x}\in X^{st}_{f,d}$. \\
It remains to show $X^{ss}_{f,d}\subset U$. Pick $\hat{x}=(x,h)$ in $X_{f,d}^{ss}$. We can represent $x$ by some pair $(E,\psi)$ with $E\in \AA_\KK$.  
If $E'\subset E$ is a proper subobject defined over $\KK$ such that $\omega_\KK(E')$ contains the image of $h$, then 
\[ \mu^{\Lin^{a,\theta}}(\hat{x},\lambda^{W'})= \mu^{\Lin_d^a}(x,\lambda^{W'}) + \theta\dim(E')-\theta d \ge 0\] 
by assumption on $\hat{x}$, where $W'=\psi(\omega_\KK(E'))\oplus \KK_\infty$ is made into a 2-step filtration $0=W^1 \subset W^0=W' \subset W^{-1}=\KK^d\oplus \KK_\infty$ with associated 1-PS $\lambda^{W'}$. By our equivalent conditions in assumption (4), $\mu^{\Lin_d^a}(x,\lambda^{W'})=0$ giving rise to a contradiction as $\theta\dim(E')<\theta d$. Hence, $\hat{x}\in U$ finishing the proof.  
\end{proof}
Let us denote the scheme $X_{f,d}^{ss}=X_{f,d}^{ss}$ with $X_{f,d}$. 
Passing to GIT-quotients, we obtain the following commutative diagram
\[ \xymatrix @C=2cm  { X_{f,d} \ar[r]^{\pr_{\overline{X}_d}|_{X_{f,d}}} \ar[d]_{\tilde{p}_{f,d}} & X_d \ar[d]^{\tilde{p}_d} \\  \overline{X}_{f,d}^{ss}/\!\!/G_{\hat{d}} \ar[r]^{\overline{\pi}_{f,d}} & \overline{X}_d^{ss}/\!\!/G_d }\]
Because of $\tilde{p}^{-1}_d(\Msp_d)=X_d$ (cf.\ assumption (4)), we could also describe $X_{f,d}=X^{st}_{f,d}$ as the preimage of the (open if $\Char \kk>0$) subscheme $\Msp_{f,d}:=\pi_{f,d}^{-1}(\Msp_d) \subset \overline{X}^{ss}_{f,d}/\!\!/G_{\hat{d}}$ with respect to $\tilde{p}_{f,d}$. In particular, $\Msp_{f,d}$ is the uniform geometric quotient of $X^{st}_{f,d}=X_{f,d}$, and $\overline{\pi}_{f,d}$ restricts to a morphism $\pi_{f,d}:\Msp_{f,d}\to \Msp_d$.
\begin{remark} \rm
The set $X_d\times (\Aff^{fd}_\KK)^{ss}$ is in general strictly contained in $X_{f,d}$, as for $(E,\psi,\KK,h)\in X_{f,d}$ the vector $1\in \KK_\infty$ can generate proper subrepresentations $(V,\KK,h)$ of $(\KK^d,\KK,h)$, but $\psi^{-1}(V)$ is not of the form $\omega_\KK(E')$ for some proper subobject $E'\subset E$. By the same reason $\overline{X}_d^{ss}\times (\Aff^{fd}_\KK)^{ss} \subset\overline{X}_{f,d}^{ss}$ is a strict inclusion and the sufficient condition of Proposition \ref{fiber_product} does not apply. Nevertheless, a moduli space for $\AA_f^{\hat{\theta}}$ exists because $X_{f,d}\to \Msp_{f,d}$ is a geometric quotient and $X_d\times (\Aff^{fd}_\KK)^{ss}\subset X_{f,d}$ an open $G_{\hat{d}}$-invariant subspace. Its open image in $\Msp_{f,d}$ is the moduli space for $\AA_f^{\hat{\theta}}$.  
\end{remark}

\begin{corollary} If $(\AA,\omega,p)$ also satisfies our smoothness assumption (7), $\Msp_{f,d}$ is smooth, and $X_{f,d}=X^{st}_{f,d} \longrightarrow  \Msp_{f,d}$ is a principal $G_d$-bundle. 
\end{corollary}
\begin{proof}
As $X_d$ is smooth, the same must hold for $X^{st}_{f,d}\subset X_d\times \Aff^{fd}_\kk$. The automorphism group of a point $(E,\psi,\KK,h)$ in $X^{st}_{f,d}$ is $\Gl(1)$ corresponding to the diagonal embedding $\Gl(1)\to G_{\hat{d}}$, the group $PG_{\hat{d}}\cong G_d$ acts freely, and the quotient is smooth. 
\end{proof}

Let us come to the main result of this section. 
\begin{theorem} \label{virtsmall} Assume that $(\AA,\omega,p)$ satisfies all of our assumptions (1)--(8). Then  $\pi_{f,d}:\Msp^{ss}_{f,d} \longrightarrow \Msp_d$ is  a projective and virtually small morphism, that is, there is a finite stratification $\Msp_d=\sqcup_\lambda S_\lambda$ with empty or dense stratum $S_0=\Msp^{s}_d=X^{st}_d/PG_d$ such that $\pi_{f,d}^{-1}(S_\lambda) \longrightarrow S_\lambda$ is \'etale locally trivial and 
 \[ \dim \pi_{f,d}^{-1}(x_\lambda) - \dim \PP^{f\cdot d-1} \le \frac{1}{2} \codim S_\lambda\] 
for every $x_\lambda\in S_\lambda$ with equality only for $S_\lambda=S_0\not=\emptyset$ with fiber $\pi_{f,d}^{-1}(x_0)\cong \PP^{f\cdot d-1}$. 
\end{theorem}

\begin{proof}
To proof projectivity, we consider the commutative diagram
\[ \xymatrix @C=2cm {\overline{X}^{ss}_{f,d}/\!\!/G_{\hat{d}} \ar[r]^{\overline{\pi}_{f,d}} \ar[d] & \overline{X}^{ss}_d/\!\!/G_d \ar[d] \\ \Spec(\kk[\overline{X}^{ss}_{f,d}]^{G_{\hat{d}}}) \ar[r] &\Spec(\kk[\overline{X}^{ss}_d]^{G_d}) }\]
which exists by the previous Proposition and the functorial properties of GIT-quotients and the $\Spec$-functor. Moreover, the vertical maps are induced by the  inclusions  
\[ \xymatrix @C=2cm { 
\kk[\overline{X}^{ss}_d]^{G_d}=\Ho^0(\overline{X}^{ss}_d,\Lin_d^{\otimes 0})^{G_d} \ar[r]^(0.5){\pr^\ast} \ar@{^{(}->}[d]&  \kk[\overline{X}^{ss}_{f,d}]^{G_{\hat{d}}}=\Ho^0(\overline{X}^{ss}_{f,d},(\Lin^{a,\theta})^{\otimes 0})^{G_{\hat{d}}} \ar@{^{(}->}[d] \\ 
\bigoplus_{k\ge 0} {\Ho}^0(\overline{X}^{ss}_d,{\Lin}_d^{\otimes ak})^{G_d}  & \bigoplus_{k\ge 0} \Ho^0(\overline{X}^{ss}_{f,d},(\Lin^{a,\theta})^{\otimes k})^{G_{\hat{d}}},  }\]
of the degree $0$ parts of the graded rings. Therefore, they are projective by general properties of the $\Proj$-construction. The important point is that the  horizontal arrow $\pr^\ast$ is an isomorphism. Indeed, as a $G_{\hat{d}}=G_d\times \Gl(1)$-representation, 
\[\kk[\overline{X}^{ss}_{f,d}]=\kk[\overline{X}^{ss}_d]\otimes_\kk \Sym\Big(\bigoplus_{i\in I} \big((\kk^{d_i})^\vee\big)^{ \oplus f_i}\Big) \] 
with $\Gl(1)$ acting trivially on the first factor. Thus, $\kk[\overline{X}^{ss}_{f,d}]^{G_{\hat{d}}}=(\kk[\overline{X}^{ss}_{f,d}]^{\Gl(1)})^{G_d}=(\kk[\overline{X}^{ss}_d]\otimes \kk)^{G_d} = \kk[\overline{X}^{ss}_d]^{G_d}$ follows. As $\Spec(\kk[\overline{X}^{ss}_{f,d}]^{G_{\hat{d}}}) \cong\Spec(\kk[\overline{X}^{ss}_d]^{G_d})$, $\overline{\pi}_{f,d}:\overline{X}^{ss}_{f,d}/\!\!/G_{\hat{d}} \longrightarrow  \overline{X}^{ss}_d/\!\!/G_d$ is also projective and its restriction to $\Msp_{f,d}\subset \overline{X}^{ss}_d/\!\!/G_d$ is just $\pi_{f,d}:\Msp_{f,d}\to \Msp_d$ which must be projective, too.\\
For $E\in \Msp^{s}_d$ defined over $\KK$, every choice of $h\not=0$ and $\psi:\omega_\KK(E)\cong \KK^d$ provides a point $\hat{E}=(E,\psi,\KK,h)\in X_d\times \AA^{fd}_\kk$ satisfying the conditions of Proposition \ref{framed_reps_2}. Thus, $\pr^{-1}(X^{st}_d)\cap X_{f,d}=X^{st}_d\times_\kk (\Aff^{f\cdot d} \setminus \{0\})$ and 
\[ \pi_{f,d}:X^{st}_d\times_\kk (\Aff^{f\cdot d} \setminus \{0\})/G_{\hat{d}}=X^{st}_d\times_\kk \PP^{f\cdot d - 1} / PG_d \longrightarrow \Msp^{s}_d \]
is an \'etale locally trivial $\PP^{f\cdot d-1}$ bundle.\\
The points of $\Msp_d$ correspond to the closed orbits in $X_d$, i.e.\ to those orbits whose points have reductive stabilizer in $G_d$. The latter is just $\Aut_{\AA_\KK}(E)$ which is reductive if and only if $E=\bigoplus_{k=1}^s E_k^{\oplus m_k}$ is semisimple with pairwise non-isomorphic simple objects $E_k$ of multiplicity $m_k$. \\
Let $\xi=\big((\gamma^\bullet,d^\bullet,m_\bullet)=(\gamma^1,\ldots,\gamma^s),(d^1,\ldots,d^s),(m_1,\dots,m_s)\big)$ be an element in $\Gamma^s\times (\NN^{\oplus I})^s\times \NN^s$ for some $s\in \NN$. We define $S_\xi\subset\Msp_{\gamma,d}$ with $\gamma=\sum_{k=1}^sm_k\gamma_k$ and $d=\sum_{k=1}^sm_kd^k$ to be the locally closed subvariety corresponding to semisimple objects $E$ having a decomposition $E=\bigoplus_{k=1}^s E_k^{\oplus m_k}$ with pairwise non-isomorphic $E_k\in \Msp^{s}_{\gamma_k,d_k}$. Apparently, $S_\xi$ depends only on $(\gamma^\bullet,d^\bullet)$, but the embedding into $\Msp_{\gamma,d}$ depends also on $m_\bullet$. If $S_\xi$ is non-empty, $\Msp^{s}_{\gamma^k,d^k}$ must also be non-empty, and $(\gamma^k,\gamma^k)=1-\dim \Msp^{s}_{\gamma^k,d^k} \le 1$ follows for all $1\le k\le s$. Moreover
\[ \dim S_\xi =\sum_{k=1}^s \dim \Msp^{s}_{\gamma,d}=s-\sum_{k=1}^s (\gamma^k,\gamma^k).\]
We define the local $\Ext^1$-quiver $Q_\xi$ which in fact only depends on $\gamma^\bullet$ as the set of vertices $\{1,\ldots,s\}$ with $\delta_{kl}-(\gamma^k,\gamma^l)$ arrows form $i_k$ to $i_l$. Note that $Q_\xi$ is symmetric by our symmetry assumption on $\AA$. Define a local dimension vector $d_\xi$ depending only on $m_\bullet$  by $(d_\xi)_{i_k}=m_k$, and a local framing datum $f_\xi$ by $(f_\xi)_{i_k}=f\cdot d^k$ which depends only on $d^\bullet$. We consider the trivial stability on $Q_\xi$. Then we have a local Hilbert--Chow map $$\pi_\xi:\Msp_{f_\xi,d_\xi}^{0-ss}(Q_\xi)\rightarrow \Msp_{d_\xi}^{0-ss}(Q_\xi)=R_{d_\xi}/\!\!/G_{d_\xi}.$$
We denote the fiber over the class of the zero representation by $\Msp_{f_\xi,d_\xi}^{nilp}(Q_\xi).$ By construction, we have
$$\dim M_{d_\xi,f_\xi}^{nilp}(Q_\xi)=\dim N_{d_\xi}-\dim G_{d_\xi}+f_\xi\cdot d_\xi,$$
where $N_{d_\xi}$ denotes the nullcone of the representation of $G_{d_\xi}$ on $R_{d_\xi}$, i.e.\ the preimage of $\rho(0)\in \Spec \kk[R_{d_\xi}]^{G_{d_\xi}}$ under the quotient map $\rho:R_{d_\xi}\longrightarrow \Spec \kk[R_{d_\xi}]^{G_{d_\xi}}$. We will now use the following two results which will be (partially) proven in section 7.
\begin{theorem} \label{nilpotent stack dimension} If $Q$ is a symmetric quiver and $N_d\subset R_d$ the preimage of $\rho(0)\in \Spec \kk[R_d]^{G_d}$ under the quotient map $\rho:R_d\longrightarrow \Spec \kk[R_d]^{G_d}$, then 
\[ \dim N_d - \dim G_d \le -\frac{1}{2}(d,d)+\frac{1}{2}\sum_{k\in Q_0}(k,k)d_k-|d|, \]
where $k\in Q_0$ is also identified with the dimension vector $(\delta_{ik})_{i\in Q_0}$.
\end{theorem}
\begin{theorem} \label{locally_trivial_fibration} The Hilbert--Chow morphism $\pi_{f,d}:\Msp^{ss}_{f,d} \to \Msp_d$ restricted to $S_\xi$ is \'etale locally trivial with fiber isomorphic to $\Msp^{nilp}_{f_\xi,d_\xi}(Q_\xi)$.
\end{theorem}
The proof of the theorem is a straight forward generalization of the quiver case which is Theorem 4.1 in \cite{Reineke1}, but we give a  proof of $\pi_{f,d}^{-1}(E)\subset \Msp^{nilp}_{f_\xi,d_\xi}(Q_\xi)$ in section 7 for the sake of completeness. Only this inclusion will be used in the sequel. 
Having  these two results at hand, we obtain for $E\in S_\xi$ the estimation
\begin{eqnarray*}\dim\pi_{f,d}^{-1}(E)& \le &\dim\Msp_{f_\xi,d_\xi}^{nilp}(Q_\xi)\;=\;\dim N_{d_\xi}-\dim G_{d_\xi}+f_\xi\cdot d_\xi \\
&\le&  -\frac{1}{2}( d_\xi,d_\xi)_{Q_\xi}+\frac{1}{2}\sum_{k=1}^s( k,k)_{Q_\xi}(d_\xi)_{k}- |d_\xi|+f_\xi\cdot d_\xi.
\end{eqnarray*}
Using the definition of $Q_\xi$, $d_\xi$ and $f_\xi$, this simplifies to
\[ \dim\pi_{f,d}^{-1}(E)\leq -\frac{1}{2}( \gamma,\gamma)+\frac{1}{2}\sum_{k=1}^s( \gamma^k,\gamma^k) m_k-\sum_{k=1}^sm_k+f\cdot d.\]
On the other hand, we can rewrite the dimension formula for $S_\xi$ as
\[ \codim S_\xi=-( \gamma,\gamma)+\sum_{k=1}^s( \gamma^k,\gamma^k)+1-s.\]
The inequality
\[ \dim \pi_{f,d}^{-1}(E)-(f\cdot d-1)\leq\frac{1}{2} \codim S_\xi\]
(with equality only if $\xi=((d,1))$) claimed in the Theorem can thus be rewritten as
\[ -\frac{1}{2}( \gamma,\gamma)+\frac{1}{2}\sum_{k=1}^s( \gamma^k,\gamma^k) m_k-\sum_{k=1}^s m_k+1\leq -\frac{1}{2}( \gamma,\gamma)+\frac{1}{2}\sum_{k=1}^s( \gamma^k,\gamma^k)+\frac{1}{2}(1-s).\]
This is easily simplified to
\[ \frac{1}{2}\sum_{k=1}^s\big(( \gamma^k,\gamma^k)-2\big)\big(m_k-1\big)\leq\frac{1}{2}(s-1).\]
Since $( \gamma^k,\gamma^k)\leq 1$, the left hand side is non-positive, whereas the right hand side is non-negative. Equality holds if both sides are zero, thus $s=1$, proving virtual smallness.
\end{proof}

\section{Some background on mixed Hodge modules}

\subsection{Mixed Hodge modules} 

The ground field in the next two sections will be $\kk=\CC$. In this section we recall some standard facts about mixed Hodge modules, intersection complexes and Schur functors. The interested reader will find more details in \cite{CaMi} and \cite{Saito1}. Let $X$ be a variety with quasiprojective connected components. We denote with $\MHM(X)$ the abelian categories of mixed Hodge modules on $X$. For a morphism $f:X\longrightarrow Y$ of finite type we get two pairs $(f^\ast,f_\ast), (f_!,f^!)$ of adjoint triangulated functors $f_\ast,f_!:D^b(\MHM(X)) \longrightarrow D^b(\MHM(Y))$ and $f^\ast,f^!:D^b(\MHM(Y)) \longrightarrow D^b(\MHM(X))$,  satisfying Grothendieck's axioms of the four functor formalism. Moreover,  there 
are duality functors relating $f_\ast$ with $f_!$ and $f^\ast$ 
with $f^!$. We also mention that for each connected component $X_\alpha$ of $X$, the categories $\MHM(X_\alpha)$ are of finite length. Furthermore, there is an exact equivalence $\TT$ of $\MHM(X)$, called the Tate twist, commuting with all four functors and 
satisfying $\rat\circ \TT = \rat$. In practice, $\TT$ usually acts by means of the composition $\LL:=[-2]\circ \TT$ on $D^b(\MHM(X))$. In our applications, $\MHM(X)$ is equipped with an exact symmetric monoidal tensor product with unit object $\unit$ and $\TT(-)\cong \TT(\unit)\otimes (-)$. By abusing language, we will also denote $\TT(\unit)$, resp.\ $\LL(\unit)$, with $\TT$, resp.\ $\LL$.      
The actions of $\TT$ and $\LL$ on $\Ka_0(\MHM(X))$ coincide, making it into a $\ZZ[\LL^{\pm 1}]$-module. We denote by $\Ka_0(\MHM(X))[\LL^{-1/2}]$ the $\ZZ[\LL^{\pm 1/2}]$-module obtained by adjoining a square root of $\LL$. One can also categorify this, giving rise to an exact equivalence $\TT^{1/2}$ on an enlarged abelian category of mixed Hodge modules. Then, $\LL^{-1/2}$ is given by $[1]\circ \TT^{-1/2}$, and should be seen as a refinement of the shift functor $[1]$ on the underlying perverse sheaf.

\subsection{Intersection complex}
Given a closed equidimensional subvariety $Z\subset X$ and a polarizable variation $V$ of mixed Hodge structures on a dense open subset $Z^o$ of the regular part $Z_{reg}$ of $Z$, there is canonical mixed Hodge module $\ICS_Z(L)$ on $X$, called the $L$-twisted intersection complex of $Z$, such that $\ICS_Z(L)|_{Z^o}=L[\dim Z]$. If $Z$ and $L$ are irreducible, $\ICS_Z(L)$ is an irreducible object of $\MHM(X)$, and all irreducible objects are obtained in this way. We will, however, use the slightly non-standard convention $\ICS_Z(V)|_{Z^o}=\LL^{-\dim Z/2}(V)=\TT^{-\dim Z/2}(V)[\dim Z]$. Note that an irreducible variation of mixed Hodge structures is pure, and application of $\TT^{-1/2}$ reduces the weight by one. If $Z$ has several connected components of different dimension, the construction of $\ICS_Z(L)$ resp.\ $\ICS_Z(V)$ generalizes accordingly. Applying this to the trivial variation $\QQ$ of pure Hodge structures of type $(
0,0)$ on 
$Z_
{reg}$, we obtain a distinguished intersection complex $\ICS_Z(\QQ)$.       

\subsection{Schur functors}
Let us now specialize to $X=\Msp$, although everything in this section remains true for arbitrary commutative monoids $(X,\oplus,0)$ in the category of varieties with quasiprojective connected components such that $\oplus:X\times X \longrightarrow X$ is of finite type. We obtain a symmetric monoidal tensor product
\[ \otimes: D^b(\MHM(\Msp))\times D^b(\MHM(\Msp)) \longrightarrow D^b(\MHM(\Msp)), \quad \EE\otimes \FF:=\oplus_\ast(\EE\boxtimes \FF), \]
and if $\oplus$ is even a finite morphism, the tensor product preserves the abelian subcategory $\MHM(X)$ of mixed Hodge modules. More details can be found in \cite{Schuermann}. The unit $\unit$ is given by $\ICS_{\Msp_0}(\QQ)$, which is a skyscraper sheaf of rank one supported on the moduli space $\Msp_0\cong\Spec\kk$ containing the zero object $0$ of dimension vector $d=0$. We drop the $\otimes$-sign when dealing with the associated Grothendieck group $\Ka_0(\MHM(\Msp))$. \\
Given $\EE\in D^b(\MHM(\Msp))$ and $n\in \NN$, the object $\EE^{\otimes n}$ carries a natural action of the symmetric group $S_n$. By general arguments (see \cite{Deligne1}), we obtain a decomposition 
\[ \EE^{\otimes n} =  \bigoplus\limits_{\lambda \dashv n} W_\lambda \otimes S^\lambda(\EE) \]
for certain objects $S^\lambda(\EE)\in D^b(\MHM(\Msp))$ which are actually in $\MHM(\Msp)$ if $\oplus$ is finite, and where $W_\lambda$ denotes the irreducible representation of $S_n$ associated to the partition $\lambda$ of $n$. The tensor product used on the right hand side can be defined for every additive category, and should not be confused with the tensor product explained above. However, after identifying vector spaces $W$ with trivial variations of pure Hodge structures of type $(0,0)$ over $\Msp_0$, both tensor products agree. The decomposition is functorial,  giving rise to Schur functors $S^\lambda:D^b(\MHM(\Msp)) \longrightarrow D^b(\MHM(\Msp))$ for every partition $\lambda$. 
\begin{example} \rm \quad 
\begin{enumerate}
 \item For $\lambda=(n)$, the representation $W_\lambda$ is the trivial representation of $S_n$ and $S^\lambda(\EE)=:\Sym^n(\EE)$. If $\EE|_{\Msp_0}=0$, we get  $\Sym^n(\EE)|_{\Msp_d}=0$ for every $d\in \NN^{\oplus I}$ provided $n\gg 0$. In particular, $\Sym(\EE)=\oplus_n \Sym^n(\EE)$ is well-defined. 
  \item For $\lambda=(1,\ldots,1)$, the representation $W_\lambda$ is the sign representation of $S_n$ and $S^\lambda(\EE)=:\Alt^n(\EE)$. As before $\Alt(\EE)=\oplus_n \Alt^n(\EE)$ is well-defined provided $\EE|_{\Msp_0}=0$. 
\end{enumerate}
\end{example}
The following proposition is a standard result. 
\begin{proposition}
 Let $\EE,\FF$ be in $D^b(\MHM(\Msp))$  such that $\EE|_{\Msp_0}=\FF|_{\Msp_0}=0$. Denote by $\Part$ be the set of all partitions of arbitrary size. Then
 \begin{eqnarray}
  \Sym(\EE \oplus \FF) & = & \Sym(E) \otimes \Sym(F) , \;\mbox{ in particular} \nonumber\\
  \Sym^n(\EE \oplus \FF) & = & \bigoplus\limits_{i+j=n} \Sym^i(\EE)\otimes \Sym^j(\FF),\;\mbox{and} \label{eqn1} \\
  \Sym (\EE \otimes \FF) & = & \bigoplus\limits_{\lambda \in \Part } S^\lambda(\EE)\otimes S^\lambda(\FF), \;\mbox{ in particular} \nonumber  \\
  \Sym^n(\EE\otimes \FF) & = & \bigoplus\limits_{\lambda \dashv n} S^\lambda(\EE)\otimes S^\lambda(\FF). \label{eqn2}
 \end{eqnarray}
\end{proposition}
Equations (\ref{eqn1}) and (\ref{eqn2}) are of course also true without the additional assumptions on $\EE$ and $\FF$. The next result is also well-known.
\begin{proposition} \label{lambda-ring}
 The Schur functors $S^\lambda$ induce well defined operations on the Grothendieck group $\Ka_0(\MHM(\Msp))$ satisfying the analogs of equation (\ref{eqn1}) and (\ref{eqn2}). In particular, the Grothendieck group carries the structure of a (special) $\lambda$-ring. 
\end{proposition}
It is worth to mention the following technical detail. Although $\Sym(\EE)=\oplus_n \Sym^n(\EE)$ by definition, this equation cannot hold on the level of Grothendieck groups as we do not have infinite sums. To define these, we need to complete the Grothendieck groups as follows. Let $F^p\subset \Ka_0(\MHM(\Msp))$ be the subgroup generated by all mixed Hodge modules  $\EE$ such that $\EE|_{\Msp_d}=0$ if $d$ cannot be written as a sum of $p$ nonzero dimension vectors, i.e.\ $|d|:=\sum_{i\in I}d_i <p$. It is easy to these that $F^pF^q\subset F^{p+q}$  and $S^\lambda(F^p)\subset F^{np}$ for all $\lambda\dashv n$ and all $n,p,q\in \NN$. Hence, the $F^p$ provide a $\lambda$-ring filtration, and the corresponding completion $\hat{\Ka}_0(\MHM(\Msp))=\prod_{d\in \NN^{\oplus I}}\Ka_0(\MHM(\Msp_d))$ has a well defined ring structure and action of $S^\lambda$. Moreover, $\sum_n \Sym^n(\EE)$ is well-defined and agrees with the class of $\Sym(\EE)$ for $\EE\in F^1$. 
\\
As $\TT=\LL$ in $\Ka_0(\MHM(\Msp))$ and $\Sym^n(\TT^{\pm 1})=\TT^{\pm n}$, the $\lambda$-ring structure of Proposition \ref{lambda-ring} also extends to $\Ka_0(\MHM(\Msp))[\LL^{-1/2}]$, and even to 
\[ \Ka_0(\MHM(\Msp))[\LL^{-1/2}, (\LL^N-1)^{-1}\, :\, N\in \NN]=\]
\[= \Ka_0(\MHM(\Msp))\otimes_{\ZZ[\LL^{\pm 1}]} \ZZ[\LL^{-1/2}, (\LL^N-1)^{-1}\, :\, N\in \NN] \] by putting 
\[ S^\lambda(\TT^{\pm 1/2})= \begin{cases} \TT^{\pm n/2} & \mbox{ for }\lambda=(n), \\ 0 & \mbox{ otherwise. } \end{cases} \]
Note that $S^\lambda(\LL^{\pm 1/2})=S^\lambda(-\TT^{\pm 1/2})$ satisfies more complicated equations.\\ Again, we consider the filtration $F^p[\LL^{-1/2}]$, resp.\ $F^p[L^{-1/2},(\LL^n-1)^{-1}\,:\, n\in \NN]$, defined accordingly and perform a completion as before. By abusing notation let us denote the resulting $\lambda$-ring by 
\[ \hat{\Ka}_0(\MHM(\Msp))[\LL^{-1/2}, (\LL^N-1)^{-1}\, :\, N\in \NN] := \]
\[= \prod_{d\in \NN^{\oplus I}} \Bigl( \Ka_0(\MHM(\Msp_d))\otimes_{\ZZ[\LL^\pm]}[\LL^{-1/2},(\LL^N-1)^{-1}\,:\,N\in \NN] \Bigr) \]
which should not be confused with 
\[ \Bigl( \prod_{d\in \NN^{\oplus I}} \Ka_0(\MHM(\Msp))\Bigr) \otimes_{\ZZ[\LL^{\pm 1}]} \ZZ[\LL^{-1/2}, (\LL^N-1)^{-1}\, :\, N\in \NN].\] 
\begin{remark} \rm The reason for adjoining $\LL^{\pm 1/2}$ to the Grothendieck group is given by our non-standard convention for the intersection complex involving powers of $\LL^{1/2}$. The various completions are needed in the next section when we pass to stacks and define Donaldson--Thomas invariants.
\end{remark}
The following result illustrates the nice behavior of intersection complexes with respect to Schur functors.
\begin{proposition}
Given a dimension vector $d$ and a natural number $n$, let us denote by $\Delta$ resp.\ $\tilde{\Delta}$ the big diagonal in $\Sym^n \Msp^{s}_d \subset \Msp_{nd}$ resp.\ $(\Msp^{s}_d)^n$. For an irreducible representation $W_\lambda$ of $S_n$ denote by $\underline{W}_\lambda$ the variation of Hodge structure of type 
$(0,0)$ on $\Sym^n\Msp^{s}_d\setminus \Delta$ given by $\bigl((\Msp^{s}_d)^n\setminus \tilde{\Delta}\bigr) \times_{S_n} W_\lambda$. Then
\begin{equation} \label{eqn3} S^\lambda \bigl( \ICS_{\overline{\Msp_d^{s}}}(\QQ) \bigr) = \ICS_{\overline{\Sym^n\Msp_d^{s}}}(\underline{W}_\lambda).
\end{equation}
\end{proposition}
\begin{proof}
We may assume $\Msp^{s}_d\neq \emptyset$ so that $\overline{\Msp^{s}_d}=\Msp_d$, and we may replace $\Msp_{nd}$ with $\overline{\Sym^n\Msp^{s}_d}$. Since $\oplus:(\Msp_d)^n\longrightarrow \overline{\Sym^n\Msp^{s}_d}$ is finite, the Decomposition Theorem of Beilinson, Bernstein, Deligne, Gabber and Saito tells us that $\ICS_{\Msp_d}(\QQ)^{\otimes n}=\oplus_\ast \bigl(\ICS_{\Msp_d}(\QQ)^{\boxtimes n}\bigr)=\ICS_{\overline{\Sym^n\Msp^{s}_d}}(V)$ for a suitable variation of Hodge structures $V$ on the smooth locus $\Sym^n\Msp^{s}_d\setminus \Delta$. As the restriction of $\oplus$ to $(\Msp^{s}_d)^n\setminus \tilde{\Delta}$ is a left $S_n$-principal bundle over $\Sym^n\Msp^{s}_d\setminus \Delta$, we can trivialize it \'{e}tale locally as $U\times S_n$, showing that the fiber of $V$ is just $H^0(S_n,\QQ)$. The natural $S_n$-action on $\ICS_{\Msp_d}(\QQ)^{\otimes n}$ is induced by the left multiplication with $S_n$ on $U\times S_n$, while the right multiplication induces the 
transitions 
between different trivializations. The $S_n$-bimodule $H^0(S_n,\QQ)$ decomposes as $\oplus_{\lambda \dashv n} W_\lambda\otimes W_\lambda$ with the left factor, resp.\ the right factor, corresponding to the left action, resp.\ to the right action. Hence, $V=\oplus_{\lambda \dashv n} W_\lambda \otimes \underline{W}_\lambda$, completing the proof.      
\end{proof}

As mentioned at the beginning of this section, we can also replace $\Msp$ by $\NN^{\oplus I}\times\Spec \CC$ considered as a zero-dimensional monoid in the category of varieties with quasiprojective connected components. All of the constructions above go through with $\Msp_0$ replaced by $\{0\} \subset \NN^{\oplus I}\times\Spec\CC$. It is not difficult to see that 
\[ \hat{\Ka}_0(\MHM(\NN^{\oplus I}\times \Spec\CC))[\LL^{-1/2}, (\LL^N-1)^{-1}\, :\, N\in \NN] =\]
\[ =\Ka_0(\MHM(\CC))[\LL^{-1/2}, (\LL^N-1)^{-1}\, :\, N\in \NN][[t_i\,:\, i\in I]].\]
Since $\dim:\Msp \longrightarrow \NN^I$ is a homomorphism of monoids with quasiprojective connected components, $\dim_\ast$ and $\dim_!$ define triangulated tensor functors 
\[D^b(\MHM(\Msp)) \longrightarrow D^b(\MHM(\NN^{\oplus I}\times \Spec\CC))\] 
commuting with Schur functors of the same type. In particular, we get $\lambda$-ring homomorphisms $\dim_\ast$ and $\dim_!$ from
\[ \hat{\Ka}_0(\MHM(\Msp))[\LL^{-1/2}, (\LL^N-1)^{-1}\, :\, N\in \NN] \]
to
\[  \Ka_0(\MHM(\CC))[\LL^{-1/2}, (\LL^N-1)^{-1}\, :\, N\in \NN][[t_i\,:\, i\in I]] \]
commuting with the Schur operators.

\section{DT invariants and intersection complexes}

\subsection{Donaldson--Thomas invariants}

From now on we assume that $(\AA,\omega,p)$ satisfies assumptions (1)--(8). We will  introduce a generalization of Donaldson--Thomas invariants using the notation of the previous sections. Consider the morphism $p:\Mst \longrightarrow \Msp$. Our first object is\footnote{Note that $p_!$ is the derived direct image with compact support, while $p_\ast$ is the usual derived direct image.} $p_! \ICS_{\Mst}(\QQ)$ in $\hat{\Ka}_0(\MHM(\Msp))[\LL^{-1/2}, (\LL^N-1)^{-1}\, :\, N\in \NN]$. To define it properly, we should develop a theory of mixed Hodge modules on Artin stacks along with a four functor formalism. However, in our situation of smooth quotient stacks we will use a more direct approach avoiding complicated machinery. First of all, $\Mst_d$ is smooth,  motivating $\ICS_{\Mst_d}(\QQ)=\LL^{-\dim \Mst_d/2}(\QQ)=\LL^{(d,d)/2}(\QQ)$. Recall that $\rho_d:X_d \longrightarrow \Mst_d$ is a $G_d$-principal bundle for every dimension vector $d$. By means of the projection formula we 
would expect a formula like 
\[ H_c^\ast(G_d,\QQ) \,\ICS_{\Mst_d}(\QQ) = \rho_{d\,!} \rho_d^\ast \ICS_{\Mst_d}(\QQ) = \LL^{\dim G_d/2} \rho_{d\,!} \ICS_{X_d}(\QQ)= \LL^{(d,d)/2}\rho_{d\,!} \QQ \]
in $\hat{\Ka}_0(\MHM(\Mst))[\LL^{-1/2}, (\LL^n-1)^{-1}\, :\, n\in \NN]$. Hence, we will define $p_{d\, !} \ICS_{\Mst_d}(\QQ)$ as the product in $\hat{\Ka}_0(\MHM(\Msp))[\LL^{-1/2}, (\LL^n-1)^{-1}\, :\, n\in \NN]$ of $\LL^{(d,d)/2}p_{d\,!} q_{d\,!} \QQ$ with the inverse of the class $\prod_{i\in I}\LL^{d_i \choose 2}\prod_{n=1}^{d_i}(\LL^n-1)\in \ZZ[\LL]$ of $H_c^\ast(G_d,\QQ)$. Summing over $d\in \NN^{\oplus I}$ gives $p_! \ICS_{\Mst}(\QQ)$. The following lemma is a standard fact in the theory of (filtered) $\lambda$-rings.
\begin{lemma}
 There is an element $\DTS\in \hat{\Ka}_0(\MHM(\Msp))[\LL^{-1/2}, (\LL^n-1)^{-1}\, :\, n\in \NN]$ with $\DTS|_{\Msp_0}=0$ such that
 \[ p_! \ICS_{\Mst}(\QQ) = \Sym \Bigl( \frac{1}{\LL^{1/2}-\LL^{-1/2}} \DTS \Bigr).\]
\end{lemma}
\begin{definition}
 We call $\DTS_d:=\DTS|_{\Msp_d}\in \hat{\Ka}_0(\MHM(\Msp_d))[\LL^{-1/2}, (\LL^n-1)^{-1}\, :\, n\in \NN]$ the Donaldson--Thomas ``sheaf'' and $\DT_d:=\dim_! \DTS_d = H^\ast_c(\Msp_d,\DTS_d)\in \Ka_0(\MHM(\CC))[\LL^{-1/2}, (\LL^n-1)^{-1}\, :\, n\in \NN]$ the Donaldson--Thomas invariant of dimension vector $d$.   
\end{definition}
As $\dim_!$ is a $\lambda$-ring homomorphism, our definition of Donaldson--Thomas invariants agrees with the usual one \cite{KS2}. 
The following result is Corollary \ref{alternative_form}. 
\begin{proposition}
Given a  framing vector $f\in \NN^I$ such that $2|f_i$ for all $i\in I$, we obtain the following formula with $\NN^{\oplus I}_{\not=0}:=\NN^{\oplus I}\setminus \{0\}$
\begin{equation} \label{eqn4} \pi_{f\,\ast} \ICS_{\Msp_{f}}(\QQ)=\pi_{f\,!} \ICS_{\Msp_{f}}(\QQ) = \Sym \Bigl( \sum_{d\in \NN^{\oplus I}_{\not= 0}} [\PP^{f\cdot d-1}]_{vir} \DTS_d \Bigr) \end{equation}
in $\hat{\Ka}_0(\MHM(\Msp))[\LL^{-1/2}, (\LL^n-1)^{-1}\, :\, n\in \NN]$, using the shorthand $[\PP^{f\cdot d-1}]_{vir}:= \frac{\LL^{f\cdot d/2} - \LL^{-f\cdot d/2}}{\LL^{1/2}-\LL^{-1/2}}$ for $d\in \NN^{\oplus I}$. Here $\pi_f=\sqcup_d\pi_{f,d}:\Msp_{f} \longrightarrow\Msp$ is the morphism forgetting the framing. 
\end{proposition}
The parity assumption on the framing vector is made to avoid typical ``sign problems''. 

\subsection{The main result}

We also need to assume the following result proven in section 6.
\begin{theorem} \label{intconj}
The Donaldson--Thomas sheaf $\DTS$ is in the image of the natural map \[\hat{\Ka}_0(\MHM(\Msp))[\LL^{-1/2}] \longrightarrow \hat{\Ka}_0(\MHM(\Msp))[\LL^{-1/2}, (\LL^N-1)^{-1}\, :\, N\in \NN].\]  
\end{theorem}
\begin{remark} \rm Note that $\Ka_0(\MHM(\Msp_d))$ is free over $\ZZ[\LL^{\pm 1}]$. Indeed, consider the set of all pairs $(Z,V)$ with $Z$ being an irreducible closed subvariety and $V$ being an irreducible variation of Hodge structures of weight $0$ or $1$ on an open dense non-singular subset $Z_o\subset Z$. Call two pairs $(Z,V)$ and $(Z',V')$ equivalent if $Z=Z'$ and $V|_{Z_o\cap Z'_o}=V'|_{Z_o\cap Z'_o}$, and denote with $[Z,V]$ the equivalence class of $(Z,V)$. Notice that $\ICS_Z(V)=\ICS_{Z'}(V')$ for equivalent pairs. Then, set of all intersection complexes $\ICS_Z(V)$, with $[Z,V]$ running through the set of equivalence classes provides a basis of the $\ZZ[\LL^{\pm 1}]$-module $\Ka_0(\MHM(\Msp_d))$. As  $\ZZ[\LL^{\pm 1/2}] \hookrightarrow \ZZ[\LL^{-1/2}, (\LL^n-1)^{-1}\, :\, n\in \NN]$ is injective, we can, therefore, identify $\hat{\Ka}_0(\MHM(\Msp))[\LL^{-1/2}] $ with a $\lambda$-subring of $\hat{\Ka}_0(\MHM(\Msp))[\LL^{-1/2}, (\LL^n-1)^{-1}\, :\, n\in \NN]$.
\end{remark}
\begin{theorem} We have  $ \DTS=\ICS_{\overline{\Msp^{s}}}(\QQ)$ in $\hat{\Ka}_0(\MHM(\Msp))[\LL^{-1/2}]$. In particular, 
 \[ \DT_d=\begin{cases} \IC_c(\Msp_d,\QQ)=\IC(\Msp_d,\QQ)^\vee &\mbox{ if } \Msp^{s}_d\neq \emptyset, \\
          0 & \mbox{ otherwise}
         \end{cases} \]
holds in $\Ka_0(\MHM(\CC))[\LL^{-1/2}]$ for every dimension vectors $d\in \NN^{\oplus I}_{\not= 0}$.
\end{theorem}
\begin{proof} To prove the theorem we make use of the weight filtration of a mixed Hodge module. Given a mixed Hodge module $\EE$ on a variety $X$ with irreducible support $Z$, the stalk (shifted by $-\dim Z$) at a generic closed point $\eta\in Z$ is a mixed Hodge structure with weight filtration 
\[ 0=W_{m-1} \subsetneqq W_m \subset \ldots \subset W_{n-1} \subsetneqq W_n=\EE_\eta.\]
Define $\deg\EE:=\max\{|m|,|n|\}+\dim Z$. As the action of $\TT$ shifts the weight filtration by two, we require that $\TT^{1/2}$ increases the indices by one. In particular, $\deg \ICS_Z(\QQ) = 0$ by our convention. More generally, if $\EE$ is a class in $\Ka_0(\MHM(X))[\LL^{-1/2}]$, we write it uniquely as a linear combination of irreducible intersection complexes, and denote by $\deg(E)$ the maximal degree of its summands. If $P$ is a Laurent polynomial in $\LL^{1/2}$ invariant under $\LL^{1/2} \leftrightarrow \LL^{-1/2}$, we get $\deg(P(\LL^{1/2})\EE)=\deg(\EE)+\deg(P)$. In particular, if $\lambda$ is a partition of $N$ and $r\in \NN$, then $\deg( S^\lambda[\PP^{r}]_{vir}\cdot \EE) \le \deg(\EE)+Nr$, with equality only for $\lambda=(N)$.\\
As before, $\Part$ denotes the set of all partitions of arbitrary size and $\NN^{\oplus I}_{\not=0}=\NN^{\oplus I}\setminus \{0\}$. We fix a framing vector $f\in \NN^I$ such that $2|f_i$ for all $i\in I$ and rewrite equation (\ref{eqn4}) using equations (\ref{eqn1}) and (\ref{eqn2}):
\[ \pi_{f,d\,\ast} \ICS_{\Msp_{d,f}}=\sum_{\begin{array}{c} \scriptstyle\lambda: \NN^{\oplus I}_{\not=0} \rightarrow \Part \\ \scriptstyle \sum |\lambda_e|e=d\end{array}} \prod_{e\in \NN^{\oplus I}_{\not=0}} S^{\lambda_e}[\PP^{f\cdot e -1}]_{vir} \cdot S^{\lambda_e}\DTS_e.\]
By induction over $|d|=\sum_{i\in I}d_i$, we conclude using equation (\ref{eqn3}) that
\begin{eqnarray}
 \pi_{f,d\,\ast} \ICS_{\Msp_{f,d}}& =& \underbrace{[\PP^{f\cdot d-1}]_{vir}\DTS_d}_{\mbox{\scriptsize for }\lambda=\delta_d} + \sum_{\begin{array}{c} \scriptstyle \lambda: \NN^{\oplus I}_{\not=0} \rightarrow \Part \\ \scriptstyle \sum |\lambda_e|e=d \\ \scriptstyle \lambda\neq\delta_d \end{array}} \prod_{e\in \NN^{\oplus I}_{\not=0}} S^{\lambda_e}[\PP^{f\cdot e -1}]_{vir} \cdot \ICS_{\overline{\Sym^{|\lambda_e|}\Msp^{s}_e}}(\underline{W}_{\lambda_e}) \nonumber \\
 \label{eqn5} & = & [\PP^{f\cdot d-1}]_{vir}\DTS_d + \sum_{\begin{array}{c} \scriptstyle \lambda: \NN^{\oplus I}_{\not=0} \rightarrow \Part \\ \scriptstyle \sum |\lambda_e|e=d \\ \scriptstyle \lambda\neq\delta_d \end{array}} \Bigl(\prod_{e\in \NN^{\oplus I}_{\not=0}} S^{\lambda_e}[\PP^{f\cdot e -1}]_{vir} \Bigr)\cdot \ICS_{\overline{\Sym^{\lambda}\Msp^{s}}}(\underline{W}_{\lambda}),
\end{eqnarray}
where we used the shorthands $\Sym^\lambda\Msp^{s}:=\prod_e \Sym^{|\lambda_e|}\Msp^{s}_e \subset \Msp_d$ and $ \underline{W}_\lambda:=\boxtimes_e\underline{W}_{\lambda_e}$. Moreover, $\delta_d:\NN^{\oplus I} \rightarrow \Part$ maps $d$ to the partition $(1)$ and any other dimension vector to the zero partition.\\
Note that if $\DTS_d\not= 0$, the first summand in equation (\arabic{equation}) has degree \footnote{Note that $\deg(\DTS_d)$ is well-defined by the integrality condition \ref{intconj}.} at least $f\cdot d-1$. In contrast to this, the summands for $\lambda\neq\delta_d$ have degree at most $f\cdot d-\sum_e |\lambda_e|$ which is less than $f\cdot d-1$ by  assumption on $\lambda$.\\
On the other hand, we can use the Decomposition Theorem of Beilinson, Bernstein, Deligne, Gabber and Saito and the fact that $\pi_{f,d}$ is virtually small (see Theorem \ref{virtsmall}) to obtain
\begin{equation} \label{eqn6}
\pi_{f,d\,\ast} \ICS_{\Msp_{f,d}}=[\PP^{f\cdot d-1}]_{vir}\ICS_{\overline{\Msp^{s}_d}}(\QQ) + \sum_{\begin{array}{c} \mbox{\scriptsize lower strata }\\ \scriptstyle S\subset \Msp_d,\: S=\bar{S} \end{array}} \sum_{\begin{array}{c} \mbox{\scriptsize irred.\ VHS} \\ \scriptstyle L\mbox{ \scriptsize on } S_{reg} \end{array}} m_{S,L} \:\ICS_S(L) 
\end{equation}
for certain $m_{S,L}\in \ZZ$. Moreover, $\deg(\ICS_S(L))<f\cdot d-1$, and if $\Msp_d^{s}\not= \emptyset$, the first summand has degree at least $f\cdot d-1$. Note that the expression for the first summand is a direct consequence of the relative Hard Lefschetz Theorem applied to the projective morphism $\pi_{f,d}$ with generic fiber $\PP^{f\cdot d-1}$. 
Using the integrality condition \ref{intconj}, we make the Ansatz 
\[ \DTS_d=\sum_{\begin{array}{c} \mbox{\scriptsize strata }\\ \scriptstyle \emptyset\not= S\subset \Msp_d, \: S=\bar{S} \end{array}} \sum_{\begin{array}{c} \mbox{\scriptsize irred.\ VHS} \\ \scriptstyle L\mbox{ \scriptsize on } S_{reg} \end{array}} u_{S,L}\: \ICS_S(L) \]
for certain unique $u_{S,L}\in \ZZ$ and put this into equation (\ref{eqn5}). By comparing the degrees of the resulting expression with equation (\ref{eqn6}), we finally conclude
\[ u_{S,L}=\begin{cases} 1 & \mbox{ if } S=\overline{\Msp^{s}_d}\not=\emptyset, L=\QQ, \\ 0 & \mbox{ otherwise} \end{cases} \]
proving the theorem.
\end{proof}

\section{Motivic DT-theory and the integrality conjecture}

We prove a stronger version of Theorem \ref{intconj} involving motivic functions instead of mixed Hodge modules. The reader not familiar with motivic functions might have a look at \cite{JoyceMF}, but we will also recall the main definitions below. In fact, there is no difference between the two versions of Donaldson--Thomas invariants, as long as we work with representations of quivers without potential, since all computations will take place in the ring of Tate motives. We will make extensive use of restriction to generic points. Since it is not so clear how to deal with mixed Hodge modules on varieties defined over arbitrary (function) fields, we will work in the motivic world. The machinery used to define Donaldson--Thomas sheaves will also work in this more general context, and one ends up with the motivic Donaldson--Thomas invariants respectively Donaldson--Thomas 
sheaves. In fact, 
we will not construct any motivic sheaf. Instead, we define an element in some ring $\hat{\Ka}(\Var/\Msp)[\LL^{-1/2}, (\LL^N-1)^{-1}\, :\, N\in \NN]$, which is a simplified version of the Grothendieck group of the category of motivic sheaves. There is a $\lambda$-ring homomorphism from
\[\hat{\Ka}(\Var/\Msp)[\LL^{-1/2}, (\LL^N-1)^{-1}\, :\, N\in \NN] \]
to
\[ \hat{\Ka}_0(\MHM(\Msp))[\LL^{-1/2}, (\LL^N-1)^{-1}\, :\, N\in \NN],\] induced by $[X\xrightarrow{q} \Msp]\mapsto q_! \QQ_X$, giving rise to corresponding results for mixed Hodge modules.

\subsection{Motivic functions}

Given an arbitrary Artin stack or scheme $\BB$ of finite type over\footnote{In practice, $\KK$ will be our ground field $\kk$ or some extension of $\kk$.} $\KK$, we define the Grothendieck group $\Ka(\Var/\BB)$ to be the free abelian group generated by isomorphism classes $[\mathcal{X} \rightarrow \BB]$ of representable morphisms of finite type such that  $\mathcal{X}$ has a locally finite stratification by quotient stacks $\mathcal{X}_i=X_i/\Gl_{n_i}$, subject to the cut and paste relation 
\[ [\mathcal{X} \rightarrow \BB] = [\mathcal{Z} \rightarrow \BB]+ [\mathcal{X}\setminus \mathcal{Z} \rightarrow \BB], \]
where $\mathcal{Z}\subset \mathcal{X}$ is a closed substack. 
\begin{remark} \rm \label{reduction_to_affine_case}
Using the cut and paste relation we arrive at the conclusion that if $\BB=\Spec B$ as an affine scheme of a finitely generated $\KK$-algebra $B$, the group $\Ka(\Var/\Spec B)$ can also be described as the  abelian group generated by symbols $[A]$ for each finitely generated $B$-algebras $A$ subject to the following two conditions.
\begin{enumerate}
\item If $A\cong A'$ as $B$-algebras, then $[A]=[A']$.
\item If $a_1,\ldots,a_r\in A$ is a finite set of elements, then 
\[ [A]=[A/(a_1,\ldots a_r)] \quad + \sum_{\emptyset\not= J\subset \{1,\ldots,r\}} (-1)^{|J|-1} [A_{\prod_{j\in J} a_j}].\]
\end{enumerate}
\end{remark}
The fiber product defines a ring structure on $\Ka(\Var/\KK)$ and a $\Ka(\Var/\KK)$-module structure on $\Ka(\Var/\BB)$. Taking the product over $\KK$ defines an exterior product $\boxtimes:\Ka(\Var/\BB)\times \Ka(\Var/\BB')\longrightarrow \Ka(\Var/\BB\times \BB')$. Let us also introduce the module 
\[ \Ka(\Var/\BB)[\LL^{-1/2}, (\LL^N-1)^{-1}\,:\, N\in \NN]:=\Ka(\Var/\BB)\otimes_{\ZZ[\LL]} \ZZ[\LL^{-1/2}, (\LL^N-1)^{-1} \,:\, N\in \NN]\] 
with $\LL$ denoting the Lefschetz motive $\LL:=[\mathbb{A}^1_\KK]\in \Ka(\Var/\KK)$.\footnote{For $\BB=\Spec \KK$, we simplify the notation by suppressing the structure morphism to $\Spec \KK$.}  We will also add the relations 
\begin{equation} \label{principal_bundle_relation} [X/\Gl_n \rightarrow \BB]= [X\rightarrow \BB]/[\Gl_n] 
\end{equation}
for every $\Gl_n$-action on a scheme $X$. Here, $[\Gl_n]=\LL^{n \choose 2}\prod_{p=1}^n (\LL^p-1)$. In particular, due to our assumption on $\mathcal{X}$ for a generator $[\mathcal{X}\to \BB]$, the group $\Ka(\Var/\BB)[\LL^{-1/2}, (\LL^N-1)^{-1}\, :\, N\in \NN]$ is generated as a $\ZZ[\LL^{-1/2}, (\LL^N-1)^{-1}\, :\, N\in \NN]$-module by morphisms $[X\rightarrow \BB]$, with $X$ being a scheme. Because of this, the proper push forward $\phi_!$ along morphisms $\phi:\BB \rightarrow \BB'$ such that $\pi_0(\phi):\pi_0(\BB)\to \pi_0(\BB')$ has finite fibers is well defined by composition $\phi_!([X\rightarrow \BB])=[X\rightarrow \BB']$. \\
We can also define \[\phi^\ast:\Ka(\Var/\BB')[\LL^{-1/2},(\LL^N-1)^{-1}\,:\,N\in \NN] \longrightarrow  \Ka(\Var/\BB')[\LL^{-1/2},(\LL^N-1)^{-1}\,:\,N\in \NN] \]
for all $\phi:\BB\to \BB'$ via $\phi^\ast([\mathcal{X}\to \BB])=[\mathcal{X}\times_{\BB'}\BB \to \BB]$ on generators. In particular, the restriction of a motivic function to a substack $\mathcal{U}\hookrightarrow \BB$ is well-defined. We will identify two motivic functions on $\BB$ if their restrictions to any substack of finite type\footnote{An Artin stack with an atlas of finite type will be called of finite type.} agree, and denote the resulting group by $\hat{\Ka}(\Var/\BB)[\LL^{-1/2}, (\LL^N-1)^{-1}\, :\, N\in \NN]$. \\
The pull-back and the push-forward satisfy some base change formula for every cartesian square. Moreover, for every quotient stack $\rho:X\to X/G$ with $G$ being a special linear algebraic group, the formula 
\begin{equation} \label{push_pull} \rho_!\rho^\ast(f)=[G]\cdot f\qquad \mbox{for all }f\in \hat{\Ka}(\Var/\BB)[\LL^{-1/2},(\LL^N-1)^{-1}\,:\,N\in \NN]. 
\end{equation}
holds, and $[G]$ is invertible. Indeed, if $[Y\xrightarrow{u} X/G]$ is a generator, then $\rho_!\rho^\ast[Y\to X/G]=[Y\times_{X/G} X \longrightarrow Y \longrightarrow X/G]$ with $P=Y\times_{X/G}X$  being  a principal $G$-bundle on $Y$. As $G$ is special, it is Zariski locally trivial and $[P\to Y]=[G][Y\to Y]$ follows in $\hat{\Ka}(\Var/Y)[\LL^{-1/2},(\LL^N-1)^{-1}\,:\,N\in \NN]$. Hence, 
\[\rho_!\rho^\ast([Y\to X/G])=[P\to Y\xrightarrow{u} X/G]= u_!([P\to Y])=[G][Y\to X/G]. \] 
The principal $G$-bundle $\Gl(n)\to \Gl(n)/G$ is Zariski locally trivial and $[\Gl(n)]=[G][\Gl(n)/G]$ is invertible proving the invertibility of $[G]$.

\subsection{$\lambda$-ring structures}

If the base $\BB$ is a scheme and has an additional structure of a commutative monoid with zero $\Spec \KK \xrightarrow{\;0\;}\BB$ and sum $\oplus:\BB\times_\KK \BB \rightarrow \BB$, then $\Ka(\Var/\BB)$ can be equipped with the structure of a $\lambda$-ring by putting
\begin{eqnarray*} [\mathcal{X}\rightarrow \BB]\cdot [\mathcal{Y}\rightarrow \BB] &:=& [\mathcal{X}\times_\KK \mathcal{Y} \rightarrow \BB\times_\KK \BB \xrightarrow{\;\oplus\;} \BB] \;\mbox{ and} \\
\sigma^n([\mathcal{X} \rightarrow \BB]) &:=& [ \mathcal{X}^n/S_n \rightarrow \BB^n/S_n \xrightarrow{\;\oplus\;} \BB], \end{eqnarray*}
where $S_n$ denotes the symmetric group. Due to a result of Kresch \cite{Kresch}, the quotient by $S_n$ can either be taken   in the stacky or in the naive sense. On can extend the $\lambda$-ring structure to $\Ka(\Var/\BB)[\LL^{-1/2}, (\LL^N-1)^{-1}\, :\, N\in \NN]$ by defining $-\LL^{1/2}$ to be a line element, that is, $\sigma^n(-\LL^{1/2}):=(-\LL^{1/2})^n$. Moreover, the $\lambda$-ring structure extends to $\hat{\Ka}(\Var/\BB)[\LL^{-1/2}, (\LL^N-1)^{-1}\, :\, N\in \NN]$.\\
Given a motivic function $f\in \hat{\Ka}(\Var/\BB)[\LL^{-1/2}, (\LL^N-1)^{-1}\, :\, N\in \NN]$ such that $\sigma^n(f)|_{\mathcal{U}}$ vanishes for all but finitely many $n\in \NN$ depending on $\mathcal{U}$ and all substacks $\mathcal{U}\subset \BB$ of finite type, the sum 
\[ \Sym(f):=\sum_{n\ge 0} \sigma^n(f) \]
is well defined in $\hat{\Ka}(\Var/\BB)[\LL^{-1/2}, (\LL^N-1)^{-1}\, :\, N\in \NN]$ and satisfies $\Sym(0)=1=[\Spec \KK \xrightarrow{\;0\;}\BB]$ as well as $\Sym(f+g)=\Sym(f)\cdot\Sym(g)$.\\

Formation of direct sums of semisimple objects in $\AA$, respectively of sums of dimension vectors, provides $\Msp$, respectively $\NN^{\oplus I}\times\Spec\KK$, with the structure of a commutative monoid, inducing a $\lambda$-ring structure on $\hat{\Ka}(\Var/\Msp)[\LL^{-1/2}, (\LL^N-1)^{-1}\, :\, N\in \NN]$, respectively on $\hat{\Ka}(\Var/\NN^{\oplus I}\times \Spec\KK)[\LL^{-1/2}, (\LL^N-1)^{-1}\, :\, N\in \NN]$. If a motivic function $f$ on $\Msp$, respectively on $\NN^{\oplus I}\times\Spec\KK$, is supported away from the zero element, the infinite sum $\Sym(f)$ is 
well defined.

\begin{lemma} \label{lambda_pull_back} 
Assume that $\iota:N\to M$ is a homomorphism of commutative monoids in the category of schemes locally of finite type over $\kk$ such that the commuting diagram
\[
\xymatrix { N\times N \ar[d]^\oplus \ar[r]^{\iota\times \iota} & M\times M \ar[d]^\oplus \\ N \ar[r]^\iota & M }
\]
is cartesian. Then $\iota^\ast(fg)=\iota^\ast(f)\iota^\ast(g)$ and $\iota^\ast(\sigma^n(f))=\sigma^n(\iota^\ast(f))$ for all $n\in \NN$ and all $f,g\in \Ka(\Var/M)[\LL^{-1/2},(\LL^N-1)^{-1}\,:\, N\in \NN]$.
\end{lemma}
\begin{proof} We will show $\iota^\ast(\sigma^n(f))=\sigma^n(\iota^\ast(f))$ for a generator $[X\to M]$ and leave the rest to the reader. Using the assumption, one concludes easily that the outer square of the diagram 
\[
\xymatrix { Y^n \ar[d] \ar[r] & X^n \ar[d] \\ Y^n/S_n \ar[d]\ar[r] & X^n/S_n \ar[d] \\ N \ar[r]^\iota & M }
\]
is also cartesian for $Y=N\times_M X$. The terms of the middle horizontal line are quotient stacks. Using the description of morphism to $S_n$-quotient stacks by means of principal $S_n$-bundles and the fact that the outer square is cartesian, one shows that the lower small square is also cartesian. This proves that $\iota^\ast([X^n/S_n\to M])=[Y^n/S_n\to N]$ which is what we want due a previous remark related to work of Kresch.
\end{proof}

\subsection{Convolution product and integration map}

We define a ``convolution'' product, the so-called Ringel--Hall product, on $\hat{\Ka}(\Var/\Mst)[\LL^{-1/2},(\LL^N-1)^{-1}\,:\,N\in \NN]$ by means of the following diagram
\[ \xymatrix { & {\EE xact} \ar[dr]^{\pi_2} \ar[dl]_{\pi_1\times\pi_3} & \\ \Mst \times \Mst & & \Mst } \]
via $f\ast g :=\pi_{2\, !}(\pi_1\times \pi_3)^\ast(f\boxtimes g)$, where $\EE xact=\sqcup_{d,d'\in \NN^{\oplus I}}\EE xact_{d,d'}$ denotes the stack of short exact sequences $0\to E_1 \to E_2 \to E_3 \to 0$ in $\AA$, and $\pi_i$ maps such a sequence to its $i$-th entry. The open and closed substack $\EE xact_{d,d'}$ is determined by $\dim(V_1)=d$ and $\dim E_2=d'$. It is well-known that the convolution product provides $\hat{\Ka}(\Var/\Mst)[\LL^{-1/2},(\LL^N-1)^{-1}\,:\,N\in \NN]$ with a $\Ka(\Var/\kk)[\LL^{-1/2},(\LL^N-1)^{-1}\,:\,N\in \NN]$-algebra structure with unit given by the motivic function $[\Spec\kk \xrightarrow{\;0\;} \Mst]$.
As we have seen, we can describe $\Mst_d$ and  $\EE xact_{d,d'}$ as quotient stacks
\[ \Mst_d=X_d/G_d\quad\mbox{ with }\quad X_d=\Mst_d\times_{\Spec\kk/G_d} \Spec\kk \] and 
\[ \EE xact_{d,d'}=X_{d,d'}/G_{d,d'}\quad\mbox{ with }\quad X_{d,d'}=\EE xact_{d,d'}\times_{\Spec\kk/G_{d,d'}}\Spec\kk,\]
where $G_{d,d'}$ is the parabolic subgroup of $G_{d+d'}$ fixing $\kk^{d}$ inside $\kk^{d}\oplus\kk^{d'}\cong\kk^{d+d'}$. As before, $X_d$ is the moduli stack parametrizing families $E\in \AA_S$ together with an isomorphism $\psi:\omega_S(E)\cong \OO_S^d$ of $I$-graded vector bundles. Note that automorphism group of $(E,\psi)$ is trivial and $X_d(S)$ can be interpreted as the set of isomorphism classes of pairs $(E,\psi)$. Similarly, $X_{d,d'}$ is the moduli space parametrizing (isomorphism classes of) families of short exact sequences $0\to E_1\to E_2\to E_3\to 0$ in $\AA_S$ together with an isomorphism
\[ \xymatrix { 0 \ar[r] & \omega_S(E_1) \ar[r]\ar[d]^\wr_{\psi_1} &\omega_S(E_2)\ar[r]\ar[d]^\wr_{\psi_2} &\omega_S(E_3)\ar[d]^\wr_{\psi_3} \ar[r] & 0\\ 0 \ar[r] & \OO_S^d \ar[r] & \OO^d\oplus \OO_S^{d'} \ar[r] &  \OO_S^{d'} \ar[r] & 0} \]
of short exact sequences of $I$-graded vector bundles on $S$. Giving such an element, we can map it to $(E_i,\psi_i)$ inducing maps $\hat{\pi}_i$ such that the following diagram commutes.
\[
 \xymatrix @C=2cm @R=1.5cm{
 X_d\times X_{d'} \ar[d]_{\rho_d\times \rho_{d'}} &X_{d,d'} \ar[l]_{\hat{\pi}_1\times\hat{\pi}_3} \ar[r]^{\hat{\pi}_2} \ar[d]^{\rho_{d,d'}} & X_{d+d'} \ar[d]^{\rho_{d+d'}}  \\
 X_d/G_d \times X_{d'}/G_{d'}  &  X_{d,d'}/G_{d,d'} \ar[l]^{\pi_1\times \pi_3} \ar[r]_{\pi_2} & X_{d+d'}/G_{d+d'}  }
\]

Recall that $\Mst$ and $\Msp$ have a decomposition $\Mst=\sqcup_{\gamma\in \Gamma} \Mst_\gamma$ and $\Msp=\sqcup_{\gamma\in \Gamma} \Msp_{\gamma}$ for $\Gamma=\Ka(\AA_{\bar{\kk}})/\rad (-,-)$. We will also define a $\ast$-product on $\hat{\Ka}(\Var/\Msp)[\LL^{-1/2},(\LL^N-1)^{-1}\,:\,N\in \NN]$ by means of $f_\gamma\ast g_{\gamma'}:=\LL^{\langle \gamma,\gamma'\rangle /2}f_\gamma\cdot g_{\gamma'}$ for $f_\gamma$ and $g_{\gamma'}$ having support on $\Msp_\gamma$ and $\Msp_{\gamma'}$ respectively. Here $\langle \gamma,\gamma'\rangle =(\gamma,\gamma')-(\gamma',\gamma)$ is the skew-symmetrization of $(-,-)$ which vanishes if and only if assumption (8) is satisfied. Notice that the $\ast$-product is non-commutative unless (8) holds. Let us denote the restriction of $p:\Mst\to \Msp$ to $\Mst_\gamma$ with $p_\gamma$.\\

\begin{proposition} \label{integration_map}
If $(\AA,\omega,p)$ satisfies all of our assumptions (1)--(7), then the ``integration'' map 
\[I:\hat{\Ka}(\Var/\Mst)[\LL^{-1/2},(\LL^N-1)^{-1}\,:\,N\in \NN] \longrightarrow \hat{\Ka}(\Var/\Msp)[\LL^{-1/2},(\LL^N-1)^{-1}\,:\,N\in \NN] \]
given by $I(f):=\sum_{\gamma\in \Gamma} \LL^{(\gamma,\gamma)/2} p_{\gamma\, !}(f|_{\Mst_\gamma})$ is a $\Ka(\Var/\kk)[\LL^{-1/2},(\LL^N-1)^{-1}\,:\,N\in \NN]$-algebra homomorphism with respect to the $\ast$-products. 
\end{proposition}
\begin{proof}
As $\Mst=\sqcup_{\gamma\in \Gamma,d\in \NN^{\oplus I}}\Mst_{\gamma,d}$ with $\Mst_{\gamma,d}=X_{\gamma,d}/G_d$ and similarly for $\Msp$ and $\mathfrak{E} xact$, it suffices to prove $I(f\ast g)=I(f)\ast I(g)$ for $f\in  \hat{\Ka}(\Var/\Mst_{\gamma,d})[\LL^{-1/2},(\LL^N-1)^{-1}\,:\,N\in \NN]$ and $g\in  \hat{\Ka}(\Var/\Mst_{\gamma',d'})[\LL^{-1/2},(\LL^N-1)^{-1}\,:\,N\in \NN].$
We use the notation of the following commutative diagram. 
\[
 \xymatrix { & X_{(\gamma,d),(\gamma',d')} \ar[dl]_{\hat{\pi}_1\times\hat{\pi}_3} \ar[dr]^{\hat{\pi}_2} \ar[dd]^(0.3){\rho_{d,d'}}& \\
 X_{\gamma,d}\times X_{\gamma',d'} \ar[dd]_{\rho_d\times \rho_{d'}} & & X_{\gamma+\gamma',d+d'} \ar[dd]^{\rho_{d+d'}}  \\
 & X_{(\gamma,d),(\gamma',d')}/G_{d,d'} \ar[dl]^{\pi_1\times \pi_3} \ar[dr]_{\pi_2} & \\
 X_{\gamma,d}/G_d \times X_{\gamma',d'}/G_{d'} \ar[dd]_{p_\gamma\times p_{\gamma'}} & & X_{\gamma+\gamma',d+d'}/G_{d+d'} \ar[dd]^{p_{\gamma+\gamma'}}\\ & & \\
 \Msp_{\gamma,d}\times \Msp_{\gamma',d'} \ar[rr]^\oplus & & \Msp_{\gamma+\gamma',d+d'} }
\]

The first computation generalizes formula (\ref{push_pull}) to the map $\pi_1\times \pi_3$ by applying (\ref{push_pull}) to the principal bundles $\rho_{d,d'}$ and $\rho_d\times \rho_{d'}$ with structure groups $G_{d,d'}$ and $G_d\times G_{d'}$ respectively, and to the vector bundle $\hat{\pi}_1\times \hat{\pi}_3$ of  fiber dimension $dd'-(\gamma,\gamma')$ (see Proposition \ref{vector_bundle}).
For $h\in \hat{\Ka}(\Var/\Mst_d\times\Mst_{d'})[\LL^{-1/2},(\LL^N-1)^{-1}\,:\,N\in \NN]$ we get
\begin{eqnarray*}
(\pi_1\times \pi_3)_!(\pi_1\times \pi_3)^\ast(h)&=& \frac{1}{[G_{d,d'}]}(\pi_1\times \pi_3)_!\rho_{d,d' \,!}\rho_{d,d'}^\ast (\pi_1\times \pi_3)^\ast(h) \\
&=& \frac{1}{[G_{d,d'}]}(\rho_d\times \rho_{d'})_!(\hat{\pi}_1\times\hat{\pi}_3)_!((\hat{\pi}_1\times\hat{\pi}_3)^\ast(\rho_d\times \rho_{d'})^\ast(h) \\
&=& \frac{\LL^{dd'-(\gamma,\gamma')}}{[G_{d,d'}]}(\rho_d\times \rho_{d'})_!(\rho_d\times \rho_{d'})^\ast(h)\\
&=& \LL^{-(\gamma,\gamma')} h.
\end{eqnarray*}
Thus, for $f\in  \hat{\Ka}(\Var/\Mst_{\gamma,d})[\LL^{-1/2},(\LL^N-1)^{-1}\,:\,N\in \NN]$ and $g\in  \hat{\Ka}(\Var/\Mst_{\gamma',d'})[\LL^{-1/2},(\LL^N-1)^{-1}\,:\,N\in \NN]$
\begin{eqnarray*}
I(f\ast g) &=& \LL^{(\gamma+\gamma',\gamma+\gamma')/2} p_{(\gamma+\gamma')!}(f\ast g ) \\ 
&=& \LL^{(\gamma+\gamma',\gamma+\gamma')/2} (p_{\gamma+\gamma'}\pi_2)_!(\pi_1\times \pi_3)^\ast(f\boxtimes g)\\
&=& \LL^{(\gamma+\gamma',\gamma+\gamma')/2} \big(\oplus(p^\zeta_\gamma\times p^\zeta_{\gamma'})(\pi_1\times \pi_3)\big)_!(\pi_1\times \pi_3)^\ast(f\boxtimes g) \\
&=& \LL^{(\gamma+\gamma',\gamma+\gamma')/2-(\gamma,\gamma')}  \oplus_!(p^\zeta_\gamma\times p^\zeta_{\gamma'})_!(f\boxtimes g)\\
&=& \LL^{(\gamma,\gamma)/2}\LL^{(\gamma',\gamma')/2} \LL^{\langle \gamma,\gamma'\rangle /2} \oplus_!(p^\zeta_\gamma\times p^\zeta_{\gamma'})_!(f\boxtimes g)\\
&=& I(f)\ast I(g).
\end{eqnarray*}
\end{proof}

\subsection{A useful identity}

Fix a framing vector $f\in \NN^{I}$ and use the notation of Section 3. Consider the motivic function $H:=[\Mst_{f}\xrightarrow{\tilde{\pi}_f} \Mst]$ with $\Mst_f=\sqcup_{d\in \NN^{\oplus I}} X_{f,d}/G_{\hat{d}}$ and $\unit_{\Xst}:=[\Xst\xrightarrow{\,\id\,} \Xst]$  for any Artin stack $\Xst$. We claim 
\begin{equation} \label{Hilbert_scheme_identity} 
\Big(H\ast \unit_{\Mst}\Big)|_{\Mst_d}=\frac{\LL^{fd}}{\LL-1}\unit_{\Mst_d}.
\end{equation} 
Indeed, consider the following commutative diagram
\[ \xymatrix @C=1.5cm{ {\Xst}:=\mathfrak{E}xact(\AA_f)|_{\Mst_{f}\times \Mst} \ar[d]^{\pi_1^f\times \pi_3^f} \ar[r]^(0.6){\hat{\pi}_f} & 
\mathfrak{E}xact(\AA) \ar[r]^(0.6){\pi_2} \ar[d]^{\pi_1\times\pi_3} & \Mst \\ \Mst_{f}\times\Mst \ar[r]^{\tilde{\pi}_f\times\id_{\Mst}} & \Mst\times\Mst, } \]
where the terms on the left hand side correspond to objects in $\AA_{f,\KK}$ for various $\KK\supset \kk$ with $\Mst$ interpreted as the space of all  objects in $\AA_{f,\KK}$ with dimension vector in $\NN^{\oplus I}\times\{0\}$. The reader should convince himself that the square is cartesian and that $\Xst$ is the moduli stack of all objects in $\AA_{f,\KK}$ of dimension vector in $\NN^{\oplus I}\times\{1\}$. Indeed, any such object $\hat{E}=(E,W,h)$ has a unique subobject $\hat{E}_c=(E_c,W,h)$  ``generated'' by $W\cong \KK$, i.e.\ a subobject in $\Mst_{f}$, and the quotient $\hat{E}/\hat{E}_c=E/E_c$ will be in $\Mst$. By construction, $E_c$ is the intersection of all subobjects $E'\subseteq E$ with $\im(h)\subset \omega_\KK(E')$. The map $\hat{\pi}$ forgets the framing in the short exact sequence $0\to \hat{E}_c\to \hat{E}\to \hat{E}/\hat{E}_c\to 0$. We finally get
\begin{eqnarray*} 
H\ast \unit_{\Mst} &=& \pi_{2\, !} (\pi_1\times \pi_3)^\ast \big(\tilde{\pi}_{f\,!}(\unit_{\Mst_{f}}) \boxtimes \unit_{\Mst} \big) \\
&=& \pi_{2\, !} (\pi_1\times \pi_3)^\ast (\tilde{\pi}_f\times \id_{\Mst})_!\big(\unit_{\Mst_{f}} \boxtimes \unit_{\Mst}\big)  \\
&=& \pi_{2\,!}\hat{\pi}_{f\,!} (\pi^f_1\times\pi^f_3)^\ast(\unit_{\Mst_{f}\times\Mst}) \\
&=& (\pi_{2}\hat{\pi}_f)_! ( \unit_\Xst). 
\end{eqnarray*}
Looking at  components associated to dimension vectors, the map $\pi_2\hat{\pi}_f$ is a stratification of
\[ X_d\times \Aff^{fd}/G_d\times \Gl(1) \xrightarrow{\tilde{\pi}_{f,d}} X_d/G_d \]
with $\Aff^{fd}$ parametrizing the matrix coefficients of the maps from $W\cong\KK$ to $\KK^{d_i}\cong\omega_\KK(E)_i$ for $i\in I$, i.e.\ the coordinates of the framing vectors, and $\Gl(1)$ corresponds to basis change in $W\cong\KK$. 
Applying equation  (\ref{push_pull}) to the principal $G_d$ respectively $G_d\times \Gl(1)$-bundles 
\begin{eqnarray*} 
X_d&\xrightarrow{\;\rho_d\;}& X_d/G_d, \\
 X_d\times \Aff^{fd} &\xrightarrow{\;\rho_{f,d}\;}& X_d\times \Aff^{fd}/G_d\times \Gl(1), \end{eqnarray*}
yields
\begin{eqnarray*}
\lefteqn{ \tilde{\pi}_{f,d\,!}\Big(\unit_{ X_d\times \Aff^{fd}/G_d\times \Gl(1)}\Big) }\\
&=& (\tilde{\pi}_{f,d}\circ \rho_{f,d})_!\Big(\unit_{X_d\times\Aff^{fd}}\Big)/[G_d\times\Gl(1)],\\
&=& (\rho_d\circ \pr_{X_d})_!\Big(\unit_{X_d\times \Aff^{fd}}\Big)/[G_d\times\Gl(1)], \\
&=& \frac{\LL^{fd}}{\LL-1}\rho_{d\,!} \Big(\unit_{X_d}\Big)/[G_d], \\
&=& \frac{\LL^{fd}}{\LL-1} \unit_{\Mst_d}, 
\end{eqnarray*}
and the equation for the restriction of $H\ast \unit_{\Mst}$ to $\Mst_d$ follows.

\subsection{Donaldson--Thomas invariants}

The following definition of Donaldson--Thomas invariants applies only to our situation of abelian categories $\AA_\kk$ satisfying conditions (1)--(8). There is a more general and much more complicated version which can be applied to triangulated 3-Calabi--Yau $A_\infty$-categories. If $\AA_\kk$ is the category of quiver representations, we can embed $\AA_kk$ into the 3-Calabi--Yau $A_\infty$-category $D^b(\Gamma_\kk Q - \rep)$ introduced in section 2.1, and the general version simplifies to the one given here. \\
\\
The motivic version of the intersection complex $\ICS_{\Mst}$ is defined by the motivic function 
\[ \ICS_{\Mst}^{mot}:= \sum_{\gamma\in \Gamma} \LL^{(\gamma,\gamma)/2}[\Mst_\gamma \hookrightarrow \Mst]. \]
Here $\LL^{(\gamma,\gamma)/2}$ is the analog of the ``perverse shift'' since $\dim \Mst_\gamma=-(\gamma,\gamma)$.
Taking the proper push forward along the morphisms $p:\Mst \rightarrow \Msp$ and $\dim\times \cl:\Msp \rightarrow (\NN^I\times \Gamma)\times \Spec\kk$, respectively, we can define the motivic Donaldson--Thomas sheaf $\DTS^{mot}\in \hat{\Ka}(\Var/\Msp)[\LL^{-1/2}, (\LL^N-1)^{-1}\, :\, N\in \NN]$ and the motivic Donaldson--Thomas invariants $\DT^{mot}\in \hat{\Ka}(\Var/\NN^{\oplus I}\times \Gamma\times \Spec\kk)[\LL^{-1/2}, (\LL^N-1)^{-1}\, :\, N\in \NN]$, both vanishing at the zero element, by means of 
\begin{eqnarray*}
p_! \ICS_{\Mst}^{mot} &=& \Sym\Bigl( \frac{1}{\LL^{1/2}-\LL^{-1/2}}\, \DTS^{mot} \Bigr), \\
\dim_! p_! \ICS^{mot}_{\Mst} &=& \Sym\Bigl( \frac{1}{\LL^{1/2}-\LL^{-1/2}}\, \DT^{mot} \Bigr) \\ &=& \Sym\Bigl( \frac{1}{\LL^{1/2}-\LL^{-1/2}} \,\dim_!\DTS^{mot} \Bigr).
\end{eqnarray*}
For the last equation we use the fact that the monoid homomorphism $\dim\times \cl$ induces a $\lambda$-ring homomorphism $\dim_!:\hat{\Ka}(\Var/\Msp)[\LL^{-1/2}, (\LL^N-1)^{-1}\, :\, N\in \NN] \longrightarrow \hat{\Ka}(\Var/\NN^{\oplus I}\times \Gamma\times \Spec\kk)[\LL^{-1/2}, (\LL^N-1)^{-1}\, :\, N\in \NN]$. \\ 
\\
We can give an alternative definition of the Donaldson--Thomas sheaf by using framed moduli spaces.  By applying the ``integration map'' $I=\prod_{d\in \NN^{\oplus I}} I_d$ to the identity (\ref{Hilbert_scheme_identity}) and  by using $\Sym(\LL^i a)=\sum_{n\ge 0} \LL^{ni}\Sym^n(a)$, we obtain
\begin{eqnarray*}
\lefteqn{\frac{1}{\LL-1}\Sym\Big(\sum_{0\not= d\in \NN^{\oplus I} } \frac{\LL^{fd}}{\LL^{1/2}-\LL^{-1/2}} \DTS^{mot}_d\Big) }\\
&=&\sum_{d\in  \NN^I} \frac{\LL^{fd}}{\LL-1}p _{d\,!}(\ICS_{\Mst_d}) \\
&=&  I\Big( \sum_{d\in  \NN^{\oplus I}} \frac{\LL^{fd}}{\LL-1} \unit_{\Mst_d}\Big) \\
&=& I(H)\cdot I(\unit_{\Mst})\\
&=& \Big(p_!\sum_{d\in \NN^{\oplus I}} \LL^{(d,d)/2}\tilde{\pi} _{f,d\, !} (\unit_{\Mst_{f,d}})\Big)\Sym\Big(\frac{\DTS^{mot} }{\LL^{1/2}-\LL^{-1/2}}\Big) \\
&=& \Big(\pi_{f\,!}\sum_{d\in \NN^{\oplus I}} \LL^{(d,d)/2}p_{f,d\,!} (\unit_{\Mst_{f,d}})\Big)\Sym\Big(\frac{\DTS^{mot} }{\LL^{1/2}-\LL^{-1/2}}\Big) \\
&=& \frac{1}{\LL^{1/2}-\LL^{-1/2}}\Big(\pi_{f\,!}\sum_{d\in \NN^{\oplus I}} \LL^{(fd-1)/2}\ICS_{\Msp_{f,d}}\Big)\Sym\Big(\frac{\DTS^{mot} }{\LL^{1/2}-\LL^{-1/2}}\Big), 
\end{eqnarray*}
where we applied equation (\ref{push_pull}) to the principal $G_d$-bundle $X_{f,d}\to \Mst_{f,d}$ and to the principal $PG_d$-bundle $X_{f,d}\to \Msp_{f,d}$ once more to compute $p_{f,d\,!} (\unit_{\Mst_{f,d}})=\unit_{\Msp_{f,d}}/(\LL-1)$.
Using the properties of $\Sym$ and $\frac{\LL^{fd}-1}{\LL^{1/2}-\LL^{-1/2}}=\LL^{1/2}[\PP^{fd-1}]$, we get the so-called PT--DT correspondence.
\begin{proposition}[PT--DT correspondence] \label{PT-DT_stack} The following formula 
\[ \pi_{f\,!}\sum_{d\in \NN^{\oplus I}} \LL^{fd/2}\cdot\ICS_{\Msp_{f,d}}= \Sym\Big(\sum_{0\not= d\in \NN^{\oplus I} } \LL^{1/2}[\PP^{fd-1}] \DTS^{mot}_d\Big)
\]
holds for all framing vectors $f\in \NN^I$. 
\end{proposition}
If $f\in (2\NN)^I$, we have $fd/2\in \NN$, and the map 
\[  (a_d)_{d\in \NN^{\oplus I}} \longmapsto (\LL^{-fd/2}a_d)_{d\in \NN^{\oplus I}} \]
is an isomorphism of the $\lambda$-ring  $\hat{\Ka}_0(\Var/\Msp)[\LL^{-1/2},(\LL^N-1)^{-1} \,:\, N\in \NN]$ as $\Sym^n(\LL^{-fd/2}a_d)=\LL^{-nfd/2}\Sym^n(a_d)$ in this case. Applying this isomorphism to the PT--DT correspondence yields the alternative form.
\begin{corollary}[PT--DT correspondence, alternative form] The following formula  
\[ \pi_{f\,!}(\ICS_{\Msp_{f}})=\Sym\Big(\sum_{0\not= d\in \NN^{\oplus I} } [\PP^{fd-1}]_{vir} \DTS^{mot}_d\Big) \]
holds for all framing vectors $f\in (2\NN)^I$ with $[\PP^{fd-1}]_{vir}=\int_{\PP^{fd-1}}\ICS_{\PP^{fd-1}}=\frac{\LL^{fd/2}-\LL^{-fd/2}}{\LL^{1/2}-\LL^{-1/2}}$. 
\end{corollary}
Notice that $\PP^{fd-1}$ is the fiber of $\pi_{f,d}$ over $\Msp^{s}_d$.  By applying the $\lambda$-ring homomorphism from $\hat{\Ka}(\Var/\Msp)[\LL^{-1/2},(\LL^N-1)^{-1}\,:\, N\in \NN]$ to $\hat{\Ka}_0(\MHM(\Msp))[\LL^{-/2},(\LL^N-1)^{-1}\,:\,N\in \NN]$ mentioned at the beginning of this section to the previous result, we obtain the corresponding formula in $\hat{\Ka}_0(\MHM(\Msp))[\LL^{-1/2},(\LL^N-1)^{-1}\,:\,N\in \NN]$.
\begin{corollary} \label{alternative_form} The following formula
\[ \pi_{f\,\ast}(\ICS_{\Msp_{f}})=\pi_{f\, !}(\ICS_{\Msp_{f}})=\Sym\Big(\sum_{0\not= d\in \NN^{\oplus I} } [\PP^{fd-1}]_{vir} \DTS_d\Big) \]
holds for all framing vectors $f\in (2\NN)^I$. 
\end{corollary}

\subsection{The integrality conjecture}

The following rather technical conjecture plays a fundamental role in Donaldson--Thomas theory. A proof has been sketched in \cite{KS2}. A relative version, saying that whenever the conjecture holds for one stability condition, it also holds for any other, has been given in \cite{JoyceDT} (see also \cite{Reineke3}). Our proof is different from the very complicated one given by Kontsevich and Soibelman. In fact, we reduce the general situation of abelian categories satisfying assumptions (1)--(8) to a special situation for which the integrality conjecture has been proven by Efimov \cite{Efimov}.\\
\\
In simple terms the integrality conjecture says that the Donaldson--Thomas invariants look like (a linear combination of) motives of varieties rather than Artin stacks. Actually, we prove a more general version of this, from which the original integrality conjecture can be deduced by applying the proper push forward $\dim_!$.

\begin{theorem}[integrality conjecture, sheaf version]  \label{intconjsv}
The Donaldson--Thomas sheaf $\DTS^{mot}$ is in the image of the natural map 
\[ \hat{\Ka}(\Var/\Msp)[\LL^{-1/2}] \longrightarrow \hat{\Ka}(\Var/\Msp)[\LL^{-1/2}, (\LL^N-1)^{-1}\, :\, N\in \NN].\] 
\end{theorem}

\begin{corollary}[integrality conjecture] 
The Donaldson--Thomas invariant $\DT^{mot}$ is in the image of the natural map 
\[ \Ka(\Var/\kk)[\LL^{-1/2}][[t_i\,:\, i\in I]] \longrightarrow \Ka(\Var/\kk)[\LL^{-1/2}, (\LL^N-1)^{-1}\, :\, N\in \NN][[t_i\,:\,i\in I]].\]
\end{corollary}

In order to prove Theorem \ref{intconjsv}, it suffices to show that the stalk of $\DTS^{mot}$ at a not necessarily closed point $\eta \in\Msp_{\gamma,d}$ with residue field $\KK\supset \kk$ does not involve denominators apart from powers of $\LL^{1/2}$.  Indeed, we first restrict ourselves to an open affine neighbourhood $\Spec B$ of the generic point $\eta$ of an irreducible component of $\Msp_{\gamma,d}$ and show the absence of denominators. If that has been done for all generic points of $\Msp_{\gamma,d}$, we can restrict ourselves to the closed complement of the union of these neighborhoods by the cut and paste relation and proceed by induction on dimension. \\
To show the absence of denominators on irreducible affine subschemes $\Spec B$, we can use the alternative definition of $\Ka(\Var/\Spec B)$ given in Remark \ref{reduction_to_affine_case}. We have to show the following for arbitrary $N$ and arbitrary $f\in \Ka(\Var/\Spec B)$: If there is an element $g \in \Ka(\Var/\Quot(B))$ given by linear combinations of finitely generated $\Quot(B)$-algebras such that  $f\otimes_B \Quot(B)=g\otimes_{\Quot(B)} \Quot(B)[x_1,\ldots,x_N]- g$, where $\Quot(B)$ denotes the quotient field of $B$, then one can find elements $b\in B$ and $\tilde{g}\in\Ka(\Var/\Spec B_b)$ given by linear combinations of finitely generated $B_b$-algebras  such that $f\otimes_B B_b=\tilde{g}\otimes_{B_{b'}} B_{b'}[x_1,\ldots,x_N]- \tilde{g}$ and $\tilde{g}\otimes_{B_b} \Quot(B)=g$. In such a situation, we may replace the open neighbourhood of $\eta$ with $\Spec B_b$ and cancel a denominator of the form $\LL^N-1$. \\
As any finite set of finitely generated $\Quot(B)$ algebras is already defined over $B_{b'}$ for some $b'$, we can certainly ``lift'' $g$ to some $g'$. It remains to show that $f\otimes_B B_b=g'\otimes_{B_{b'}} B_{b'}[x_1,\ldots,x_N]- g'$. Over $\Quot(B)$ this is true due to the existence of a finite chain of relations presented in Remark \ref{reduction_to_affine_case}. But each of these relations does also lift to a relation over $B_b$ for some sufficiently ``large'' $b\in B\subset B_{b'}$. Then $\tilde{g}:=g'\otimes_{B_{b'}} B_{b}$ does what we want. \\
Thus, we have proven that restriction to all (not necessarily) closed points suffices to prove the absence of denominators.  In fact, we will show the following: \\
If $\eta$ with residue field $\KK$ is in the Luna stratum $S_\xi\subset \Msp_{\gamma,d}$ associated to $\xi=(\gamma^\bullet,d^\bullet,a_\bullet)$ with $\gamma=\sum_{k=1}^sa_k\gamma^k$ and  $d=\sum_{k=1}^sa_kd^k$, it represents a semisimple object $\bigoplus_{k=1}^sE_k^{a_k}$ with pairwise non-isomorphic simple objects $E_k$ in $\AA_\KK$ of class $\cl E_k=\gamma^k$, dimension vector $d^k=\dim E_k$ and multiplicity $a_k\in \NN\setminus \{0\}$. We write $E =(E_k)_{k=1}^s$ for the $s$-tuple of simple objects.\\
Consider the embedding $\iota_E:\NN^s\times\Spec\KK \hookrightarrow \Msp$ of locally finite schemes mapping $\Spec \KK$ indexed by $(n_k)_{k=1}^s$ to $\bigoplus_{k=1}^s E_k^{n_k}$. \\
Note that $\hat{\Ka}(\Var/\NN^s\times\Spec\KK)[\LL^{-1/2}, (\LL^N-1)^{-1}\, :\, N\in \NN]$ can be identified with the ring \[\Ka_0(\Var/\KK)[\LL^{-1/2}, (\LL^N-1)^{-1}\, :\, N\in \NN][[t_1,\ldots,t_s]]\] of power series in $s$ variables. We will prove that $\iota_E^\ast \DTS^{mot}$ lies in the image of \[\hat{\Ka}_0(\Var/\KK)[\LL^{-1/2}][[t_1,\ldots,t_s]] \longrightarrow \hat{\Ka}_0(\Var/\KK)[\LL^{-1/2}, (\LL^N-1)^{-1}\, :\, N\in \NN][[t_1,\ldots,t_s]].\] 
Let us form the following fiber product: 
\[ \xymatrix @C=2cm { \Mst_E \ar@{^{(}->}[r]^{\tilde{\iota}_E} \ar[d]_{\tilde{p}} & \Mst \ar[d]^p \\ \NN^s\times\Spec\KK \ar@{^{(}->}[r]^{\iota_E} & \Msp } \] 
The stack $\Mst_E =\sqcup_{n\in \NN^s} \Mst_{E,n}$ can be seen as the substack of objects having a decomposition series with factors in the collection $E =(E_k)_{k=1}^s$. Since $p_!$ commutes with base change and Lemma \ref{lambda_pull_back} applied to $\iota_E:\NN^s \hookrightarrow \Msp$, we get 
\[ \tilde{p}_! \bigl( \tilde{\iota}_E^\ast \ICS^{mot}_{\Mst} \bigr) = \Sym \Bigl( \frac{1}{\LL^{1/2}-\LL^{-1/2}} \iota_E^\ast\DTS^{mot} \Bigr). \]
Note that $\tilde{\iota}_E^\ast \ICS^{mot}_{\Mst}$  restricted to $\Mst_{E,n}$ is just $\LL^{(\gamma(n),\gamma(n))/2}[\Mst_{E,n} \xrightarrow{id} \Mst_{E,n}]$, where $\gamma(n):=\sum_{k=1}^sn_k\gamma^k$ is the class in $\Ka_0(\AA_{\bar{\kk}})/\rad(-,-)$ of any $\bar{\kk}$-point in the connected component of $\Msp$ containing $\bigoplus_{k=1}^s E^{n_k}_k$. \\
Let us recall the Ext-quiver $Q_\xi $ of $\xi=(\gamma^\bullet,d^\bullet,a_\bullet)$. Its vertex set is $\{1,\ldots,s\}$, and the number of arrows from $k$ to $l$ is given by $\delta_{kl}-(\gamma^k,\gamma^l)$. In fact, $Q_\xi$ depends only on $\gamma^\bullet=(\gamma^k)_{k=1}^s$.  For a dimension vector $n\in \NN^s$ of $Q_\xi$, we denote by $R_{n}(Q_\xi)\cong \mathbb{A}_\KK^{\sum_{\alpha:k \to l}n_kn_l}$ the affine space parametrizing all representations of $Q_\xi$ on a fixed $\KK$-vector space of dimension $n$. Recall that $R_n(Q_\xi)/G_{n}$ is the stack of $n$-dimensional $\KK Q_\xi$-representations on any vector space of dimension vector $n$. \\   
As $(-,-)$ is symmetric by assumption, the quiver $Q_\xi$ is symmetric, and we can apply the following result of Efimov to the quiver $Q_\xi$. 

\begin{theorem}[\cite{Efimov}, Theorem 1.1] 
Given any quiver $Q$ with vertex set $\{1,\dots,s\}$, we define for every $n\in \NN^s\setminus \{0\}$ the ``motivic'' Donaldson--Thomas invariants $\DT^{mot}(Q)_{n}\in \ZZ[\LL^{-1/2}, (\LL^N-1)^{-1}\, :\, N\in \NN]$ of $Q$ with respect to the trivial stability condition $\theta=0$ by means of 
\[ \sum_{n\in \NN^s} \LL^{(n,n)/2}\frac{[R_{n}(Q)]}{[G_{n}]}\, t^n =: \Sym \Bigl( \frac{1}{\LL^{1/2}-\LL^{-1/2}} \sum_{n\in \NN^s\setminus\{0\}} \DT^{mot}(Q)_{n}t^n \Bigr), \]
where we might think of $\LL^{1/2}$ as a formal variable. 
If the quiver $Q$ is symmetric, the invariant $DT^{mot}(Q)_n$ is contained in the Laurent subring $\ZZ[\LL^{\pm 1/2}]$ of $\ZZ[\LL^{-1/2}, (\LL^N-1)^{-1}\, :\, N\in \NN]$.
\end{theorem}

When we apply Efimov's Theorem to $Q_\xi$ and specialize $\LL$ to $[\AA_\KK^1]$, we use the notation $(-,-)_{Q_\xi}, R_{n}(Q_\xi)$ and $\DT^{mot}(Q_\xi):=\sum_{n\in \NN^s\setminus\{0\}} \DT^{mot}(Q_\xi)_nt^n$ to distinguish the objects from their counterparts for $\AA_\KK$ which might be $\KK Q -\rep^{ss}_\mu$. Theorem \ref{intconjsv} is then a direct consequence of the following result and the remarks made at the beginning of the proof.  

\begin{proposition} \label{localDT}
Denote by $\DT^{mot}(Q_\xi)|_{\LL^{1/2}\mapsto \LL^{-1/2}}$ the series in $\ZZ[\LL^{\pm 1/2}][[t_1,\ldots,t_s]]$ obtained by the indicated substitution. Then $\DT^{mot}(Q_\xi)|_{\LL^{1/2}\mapsto \LL^{-1/2}}=\iota_E^\ast \DTS^{mot}$. In particular, $\iota_E^\ast\DTS^{mot}$ is an element of the subring $\ZZ[\LL^{\pm 1/2}][[t_1,\ldots,t_s]]$, respectively \\ $\Ka_0(\Var/\KK)[\LL^{-1/2}][[t_1,\ldots,t_s]]$, of $\Ka_0(\Var/\KK)[\LL^{-1/2}, (\LL^N-1)^{-1}\, :\, N\in \NN][[t_1,\ldots,t_s]]$.
\end{proposition}

\begin{remark} \rm
The substitution $\LL^{1/2}\mapsto \LL^{-1/2}$ has an intrinsic meaning. For any base $\BB$ there is a duality operation on $\hat{\Ka}(\Var/\BB)[\LL^{-1/2}, (\LL^N-1)^{-1}\, :\, N\in \NN]$ which can be seen as a motivic version of (relative) Poincar\'{e} duality. See \cite{Bittner04} for more details on this.
\end{remark}

\begin{proof} As the substitution $\LL^{1/2}\mapsto \LL^{-1/2}$ is compatible with the $\lambda$-ring structure of $\ZZ[\LL^{-1/2}, (\LL^N-1)^{-1}\, :\, N\in \NN][[t_1,\ldots,t_d]]$, which contains $\ZZ[\LL^{\pm 1/2}][[t_1,\ldots,t_s]]$ as a $\lambda$-subring, it suffices to show the identity 
\begin{equation} \label{nilpotent} \Bigl(\sum_{n\in \NN^s} \LL^{(n,n)_{Q_\xi}/2} \frac{[R_{n}(Q_\xi)]}{[G_n]}t^n \Bigr)\Big|_{\LL^{1/2} \to \LL^{-1/2}} \cdot \Bigl(\sum_{m\in \NN^s} \LL^{(\gamma(m),\gamma(m))/2} [\Mst_{U,m}]t^m\Bigr) = 1 
\end{equation}
in $\Ka_0(\Var/\KK)[\LL^{-1/2}, (\LL^N-1)^{-1}\, :\, N\in \NN][[t_1,\ldots,t_s]]$. Indeed, the factor on the left hand side is by definition 
\[\Sym\Bigl(\frac{\DT^{mot}(Q_\xi)}{\LL^{1/2}-\LL^{-1/2}}\Bigr)\Bigr|_{\LL^{1/2}\mapsto \LL^{-1/2}} = \Sym\Bigl(- \frac{\DT^{mot}(Q_\xi)|_{\LL^{1/2}\mapsto \LL^{-1/2}}}{\LL^{1/2}-\LL^{-1/2}}\Bigr).\]  
On the other hand, the factor on the right hand side is nothing else than 
\[ \tilde{p}_! (\tilde{\iota}_E^\ast \ICS_{\Mst}) = \Sym\Bigl( \frac{ \iota_E^\ast \DTS^{mot}}{\LL^{1/2}-\LL^{-1/2}}\Bigr). \]
Consider the following two motivic functions on $\Mst_{\KK}=\Mst\times_\kk \Spec\KK$. 
\[ f:=\sum_{n\in \NN^s}(-1)^{|n|}\LL^{\sum_{k=1}^s{n_k \choose 2}}[ \Spec \KK /G_n \rightarrow \Mst_{\KK}] \quad\mbox{and}\quad g:=[\Mst_E \rightarrow \Mst_{\KK}],\]
where for $n\in \NN^s$ the quotient stack $\Spec \KK/G_n$ maps to the object $\bigoplus_{k=1}^s E_k^{n_k}$ of class $\gamma(n)$ and its automorphism group. In particular, the morphisms used to define $f$ and $g$ correspond to substacks of $\Mst_{\KK}$. We compute the convolution product $f\ast g$ by means 
of the following diagram
\[ \xymatrix { \mathcal{Z}_{\gamma(n),\gamma(m)} \ar@{^{(}->}[r] \ar[d] & \mathfrak{Exact}_{\gamma(n),\gamma(m),\KK} \ar[d]_{\pi_1\times \pi_3} \ar[r]^{\pi_2} & \Mst_{\gamma(n)+\gamma(m),\KK} \\ \Spec \KK/G_n \times_\KK \Mst_{U,\gamma(m)} \ar@{^{(}->}[r] & \Mst_{\gamma(n),\KK} \times_\KK \Mst_{\gamma(m),\KK},  } \] 
where the left square is cartesian, and $\mathfrak{Exact}_{\gamma(n),\gamma(m),\KK}$ denotes the stack of short exact sequences  in $\AA_\KK$ with prescribed classes for the first and third object in the sequence. The morphisms $\pi_1,\pi_2$ and $\pi_3$ map a sequence to the the corresponding entries. Since $\pi_2$ is representable, $\mathcal{Z}_{\gamma(n),\gamma(m)} \longrightarrow \Mst_{\gamma(n)+\gamma(m),\KK}$ is representable, too. In fact, $\mathcal{Z}_{\gamma(n),\gamma(m)}$ is the substack of objects $F$ that are extensions of an object  with dimension vector $\gamma(m)$ and Jordan--H\"older factors among the $(E_k)_{k=1}^s$ by the semisimple object $\bigoplus_{k=1}^s E_k^{n_k}$. In particular, the  Jordan--H\"older factors of $F$ are also among the $(E_k)_{k=1}^s$, and $\bigoplus_{k=1}^s E_k^{n_k}$ must embed into the socle $\bigoplus_{k=1}^s E_k^{N_k}$ of $F$ for certain integers $N_k$ depending on $F$. The space of such embeddings, that is, the fiber of the map $\mathcal{Z}_{\gamma(n),\gamma(m)} \longrightarrow \Mst_{\gamma(n)
+\gamma(m),\KK}$ over $F$, is given by the product of finite Grassmannians $\prod_{k=1}^s\Gr_{n_k}^{
N_k}$ over $\KK$. Hence, the convolution product $f\ast g$ restricted to $F \in \Mst_{\gamma(n)+\gamma(m),\KK}$ is
\[\sum_{0\leq n_k\leq N_k}\prod_{k=1}^s(-1)^{n_k}\LL^{{n_k}\choose 2}\left[{N_k\atop n_k}\right],\] 
the $\LL$-binomial coefficient $\left[{N_k\atop n_k}\right]$ being the Lefschetz motive of the Grassmannian $\Gr^{N_k}_{n_k}$ over $\KK$. A standard identity for $\LL$-binomial coefficients then shows that this sum vanishes as soon as $N_k\not=0$ for some $k\in K$, that is, for every non-zero $F$. Since $F$ was arbitrary, this can only happen if the motivic function $f\ast g$ is concentrated on the zero representation, and a direct computation shows $f\ast g= [\Spec \KK \xrightarrow{\;0\;} \Mst_{\KK}]=1$. 
Using Lemma \ref{integration_map}, we get the identity $1=I(f\ast g)=I(f)\cdot I(g)$ of motivic functions on $\Msp_\KK$ which are actually supported on $\NN^s\times \Spec\kk \hookrightarrow \Msp_\KK$ via the embedding $\iota_E:n\mapsto \bigoplus_{k=1}^sE_k^{n_k}$. Using $[R_{n}(Q_\xi)]=\LL^{-(n,n)_{Q_\xi}+\sum_{k=1}^sn_k^2}=\LL^{-(\gamma(n),\gamma(n))+\sum_{k=1}^s n_k^2}$, a simple computation shows that $I(f)$ is the first factor in equation (\ref{nilpotent}) while the second is obviously $I(g)$. 

\end{proof}
\begin{corollary} \label{localDT2}
Let $E=\oplus_{k=1}^s E_k^{m_k}$ be a semisimple object in $\AA_\KK$ corresponding to a point $[E]\in \Msp$. As before, $Q_\xi$ is the $\Ext^1$-quiver of the collection $(E_k)_{k=1}^s$ of simple objects. Let $\DT^{mot}(Q_\xi)^{nilp}_m$ be the ``fiber'' of $\DT^{mot}(Q_\xi)$ (with respect to the trivial stability condition) of the ``origin'' in $\Msp(Q_\xi)_m$ corresponding to zero-representation of dimension $m=(m_k)_{k=1}^s$. Then, $\DT^{mot}_E=\DT^{mot}(Q_\xi)_m^{nilp}$ for the fiber of $\DT^{mot}$ over $E\in \Msp$.
\end{corollary}
\begin{proof}
The zero-representation of dimension m is the semisimple $Q_\xi$-representation $\bigoplus_{k=1}^s S_k^{m_k}$, where $S_k$ denotes the 1-dimensional $\kk Q_\xi$-module at vertex $k$. We simply apply Proposition \ref{localDT} to the category  $\kk Q_\xi-\rep$ and the collection $(S_k)_{k=1}^s$ and take into account that the local $\Ext^1$-quiver of this collection is $Q_\xi$ again. Thus, $\iota^\ast_S\DT^{mot}(Q_\xi) =\DT^{mot}(Q_\xi)|_{\LL^{1/2}\mapsto \LL^{-1/2}}=\iota^\ast_E \DT^{mot}$. 
\end{proof}

\section{Proof of Theorem \ref{nilpotent stack dimension} and Theorem \ref{locally_trivial_fibration}}

\subsection{The stack of nilpotent quiver representations}

This section taken from \cite{MeinhardtReineke}. We include it here for the sake of completeness. Let $Q$ be a finite quiver and $d\in\NN^{Q_0}$ a dimension vector for $Q$, and consider the action of the linear algebraic group $G_d$ on the vector space $R_d$. Let $p:R_d\rightarrow R_d/\!\!/G_d$ be the invariant-theoretic quotient; in other words, $R_d/\!\!/G_d$ is the spectrum of the ring of $G_d$-invariants in $R_d$, which, by \cite{LeBruyn-Procesi}, is generated by traces along oriented cycles in $Q$. We consider the nullcone of the representation of $G_d$ on $R_d$, that is,
$$N_d:=p^{-1}(p(0)).$$

By a standard application of the Hilbert criterion (see \cite[Chapter 6]{LeBruyn} for a much finer analysis of the geometry of $N_d$ using the Hesselink stratification) we can characterize points in $N_d$ either as those representations such that every cycle is represented by a nilpotent operator, or as those representation admitting a composition series by the one-dimensional irreducible representations $S_i$ concentrated at a single vertex $i\in Q_0$ (and with all loops at $i$ represented by $0$).\\
The main observation of this section is that, under the assumption of $Q$ being symmetric, there is an effective estimate for the dimension of $N_d$.

\begin{theorem} If $Q$ is symmetric, we have
$$\dim N_d-\dim G_d\leq -\frac{1}{2}( d,d)+\frac{1}{2}\sum_{i\in Q_0}( i,i) d_i-\dim d.$$
\end{theorem}

\begin{proof} For a decomposition $d=d^1+\ldots+d^s$, denoted by $d^*$, we consider the closed subvariety $R_{d^*}$ of $R_d$ consisting of representations $V$ admitting a filtration $0=V_0\subset V_1\subset \ldots\subset V_s=V$ by subrepresentations, such that $V_k/V_{k-1}$ equals the zero representation of dimension vector $d^k$ for all $k=1,\ldots,s$. This subvariety being the collapsing of a homogeneous bundle over a variety of partial flags in $\bigoplus_{i\in Q_0}\KK^{d_i}$, its dimension is easily estimated as
$$\dim R_{d^*}\leq \dim G_d-\sum_{k<l}( d^l,d^k)-\sum_{i\in Q_0}\sum_k(d^k_i)^2.$$
The above characterization of $N_d$ allows us to write $N_d$ as the union of all $R_{d^*}$ for decompositions $d^*$ which are thin, that is, all of whose parts are one-dimensional (one-dimensionality is obscured by the notation to avoid multiple indexing and to make the argument more transparent). Thus $\dim N_d-\dim G_d$ is bounded above by the maximum of the values
$$-\sum_{k<l}( d^l,d^k)-\sum_{i\in Q_0}\sum_k(d^k_i)^2$$
over all thin decompositions. Since $Q$ is symmetric, we can rewrite
$$\sum_{k<l}( d^l,d^k)=\frac{1}{2}( d,d)-\frac{1}{2}\sum_k( d^k,d^k).$$
All $d^k$ being one-dimensional, we can easily rewrite
$$\sum_{i\in Q_0}\sum_k(d^k_i)^2=\dim d,\;\;\;
\sum_k( d^k,d^k)=\sum_{i\in Q_0}( i,i) d_i.$$
All terms now being independent of the chosen thin decomposition, we arrive at the required estimate.
\end{proof}

\subsection{The fiber of the Hilbert--Chow morphisms over a Luna stratum}

The following Lemma can be checked easily. In fact, we only need it for $S=\Spec\KK$. 
\begin{lemma} \label{ext_and_hom_groups} Thinking of $\AA_S$ and $\Vect_{S}$ as full subcategories of $\AA_{f,S}$, the following is true.  
\begin{enumerate}
\item Every object $\hat{E}=(E,W,h)\in \AA_{f,S}$ fits into a canonical short exact sequence in $\AA_{f,S}$
\[ 0\longrightarrow E \longrightarrow \hat{E} \longrightarrow W \longrightarrow 0.\]
\item For $E_1,E_2\in \AA_S$ and $W_1,W_2\in \Vect_{S}$ we have
\begin{eqnarray*}
\Hom_{\AA_{f,S}}(E_1,E_2)=\Hom_{\AA_S}(E_1,E_2) & \mbox{and} & \Hom_{\AA_{f,S}}(W_1,W_2)=\Hom_{\OO_S}(W_1,W_2) \\
\Ext^1_{\AA_{f,S}}(E_1,E_2)=\Ext^1_{\AA_S}(E_1,E_2) & \mbox{and} & \Ext^1_{\AA_{f,S}}(W_1,W_2)=\Ext^1_{\OO_S}(W_1,W_2).
\end{eqnarray*}
In particular, $\AA_\KK$ and $\Vect_{\KK}$ are Serre subcategories of $\AA_{f,\KK}$.
\item For $E\in \AA_S$ and $W\in \Vect_{S}$ we have 
\begin{eqnarray*} \Hom_{\AA_{f,S}}(E,W)=0   & \mbox{and} &   \Hom_{\AA_{f,S}}(W,E)=0\\ \Ext^1_{\AA_{f,S}}(E,W)=0 & \mbox{and} &\Ext^1_{\AA_{f,S}}(W,E)\cong\bigoplus_{i\in I} \Hom_{\OO_S}(W,\omega(E)_i)^{\oplus f_i} .
\end{eqnarray*}
In particular, $\AA_{f}$ does not satisfy our symmetry assumption (8) even if $\AA$ does. 
\end{enumerate}
\end{lemma}

Let us recall Luna's slice theorem, adopted to our setting.

\begin{theorem}[Luna's \'Etale Slice Theorem] \label{Luna}
 Let $G$ be a reductive algebraic group acting on a smooth  $\kk$-scheme $Y$ with good quotient $Y\to Y/\!\!/G$.  Pick a point $y\in Y$ with closed orbit. In particular, $\Stab_G(y)=:G_y$ is also reductive. Then, there is a locally closed $G_y$-invariant affine smooth subscheme $S$ in $Y$ containing $y$ such that the following commutative diagram
 \[ \xymatrix { S\times_{G_y} G \ar[r]^\psi \ar[d] &  Y \ar[d] \\ \Spec \kk[S]^{G_y} \ar[r] & {Y/\!\!/G} }\]
 with $\psi$ being given by the $G$-action on $Y$ is cartesian with \'etale horizontal arrows. Moreover, if $N=T_yY /T_y Gy$ is the normal space of the $G$-orbit at $y$, there is another \'etale morphism $S\to N$ mapping $y$ to $0$ with tangent map at $y$ being  $T_yS \hookrightarrow T_yY \twoheadrightarrow N$ such that 
 \[ \xymatrix { S \ar[r]\ar[d] & N \ar[d] \\ \Spec\kk[S]^{G_y} \ar[r] & \Spec \kk[N]^{G_y} } \]
 commutes and is cartesian.
\end{theorem}

Using the notation of section 4, we want to show the following result.
\begin{theorem} If $(\AA,\omega,p)$ satisfies (1)--(8), then the Hilbert--Chow morphism $\pi_{f,d}:\Msp_{f,d} \to \Msp_d$ restricted to $S_\xi$ is \'etale locally trivial with fiber contained in $\Msp^{nilp}_{f_\xi,d_\xi}(Q_\xi)$.
\end{theorem}

\begin{proof}
Consider the polarization $\Lin_d\boxtimes \OO^0=\pr_{X_d}^\ast\Lin_d$ on the smooth $\kk$-scheme $Y:=X_d\times \Aff^{fd}_\kk$ with respect to the $G_{\hat{d}}$-action with good quotient $Y/\!\!/PG_{\hat{d}}=Y/\!\!/G_d\cong X_d/\!\!/G_d=\Msp_d$ as the $G_{\hat{d}}$-invariant sections of $\pr_{X_d}^\ast \Lin_d^{\otimes k}$ are pull-backs of those on $X_d$. The $\KK$-points in the quotient $Y/\!\!/G_d$ correspond to semisimple objects in $\AA_{f,\KK}$ of dimension vector $\hat{d}=(d,1)$ which are direct sums of a semisimple object in $\AA_\KK$ and the vector space $\KK$ at the vertex $\infty$ both embedded into $\AA_{f,\KK}$. These objects are in one to one correspondence to semisimple objects in $\AA_\KK$ providing another justification of the isomorphism $Y/\!\!/G_d\cong \Msp_d$. We want to apply Luna's slice theorem to $\hat{q}:Y\to \Msp_d$.  Given a $\KK$-point $y\in Y$ with closed orbit and reductive stabilizer $G_y$, it corresponds to a semisimple object $\hat{E}=E\oplus\KK_\infty$  of dimension vector $\hat{d}=(d,1)$ with $E=\bigoplus_{k\in K} E_k^{m_k}$ for pairwise non-
isomorphic simple objects $E_k$ giving rise to a triple $\xi=(\gamma^\bullet,d^\bullet,m_\bullet)$ with $\gamma^k=\cl E_k$ and $d^k=\dim E_k$. Note that the $\Ext^1$-quiver of $\hat{E}$ is the framed quiver $Q_{\xi,f_\xi}$ associated to the $\Ext^1$-quiver $Q_\xi$ for $E$ and the framing vector $f_\xi=(f\cdot d^k)_{k\in K}$ by  Lemma \ref{ext_and_hom_groups}. Moreover, $\hat{d}_\xi=(d_\xi,1)$ with $d_\xi=(m_k)_{k\in K}$. We also want to study the quotient map $q_\xi:R_{\hat{d}_\xi}(\hat{Q}_\xi) \longrightarrow \Msp_{d_\xi}^{ssimp}(Q_\xi)$ which we obtain as above with the trivial linearization on $R_{\hat{d}_\xi}(\hat{Q}_\xi)$. By Proposition \ref{normal_ext} applied to $\AA_f$, $R_{\hat{d}_\xi}(\hat{Q}_\xi)$ can be identified with the normal space to the $G_d$-orbit\footnote{Note that the $G_d$-orbit and the $G_d\times \Gl(1)$-orbit coincide.} through $y$ corresponding to $\hat{E}$. By Luna's \'etale slice theorem, we get a smooth locally closed affine subscheme $S\subset Y$ with $y\in S$, stable under 
 \[ \Stab_{G_{d}}(y)\cong  G_{d_\xi}, \]
and a commutative diagram
 \[ \xymatrix {
 Y \ar[d]^{\hat{q}} & S \times_{G_{d_\xi}} G_{d}  \ar[l] \ar[d] \ar[r] & R_{d_\xi}(\hat{Q}_\xi) \times_{G_{d_\xi}}  G_{d} \ar[d]^{q_\xi \times_{G_{d_\xi}} G_{d}} \\ 
  \Msp_d & {S/\!\!/G_{d_\xi}} \ar[l] \ar[r] & \Msp_{d_\xi}^{ssimp}(Q_\xi) }
 \]
with cartesian squares and \'etale horizontal maps. In particular, the fibers 
\[ \hat{q}^{-1}(E)\qquad  \mbox{and} \quad  (q_\xi\times_{G_{d_\xi}}G_d )^{-1}(0)\cong q_\xi^{-1}(0)\times_{G_{d_\xi}}G_d \]
are isomorphic to each other. The fiber on the left hand side contains the open subvariety $\hat{q}^{-1}(E)\cap X_{f,d}$ of framed objects $(E',\KK,h')$ (along with a choice of a basis in $\omega(E)$) containing no proper subobjects $(E'',\KK,h')$ such that the associated graded of $E'$ is $E$. Under the isomorphism of fibers this corresponds to tuples $(V,\KK,h_\xi,g)$ with $g\in G_d$, $V$ a representation of $Q_\xi$ (with a choice of a basis) having Jordan--Hölder factors among the 1-dimensional simple representations $S_k(Q_\xi)$ located at vertex $k$, and $f\cdot d^k$ framing vectors in $V_k$ given by $h_\xi$, such that $(V,\KK,h_\xi)$ has no proper subrepresentation $(V',\KK,h'_\xi)$.  Indeed, any such subrepresentation $(V',\KK,h'_\xi)$ determines a one-parameter subgroup of $G_{d_\xi}$ giving rise to a one-parameter subgroup in $G_d$. The latter corresponds to a proper subobject $(E'',W,h'')$ of $(E',\KK,h)$ with $W$ being  $0$ or $\KK$. We claim $W=\KK$. Indeed, the multiplicities of the 
Jordan--Hölder factors of $(E'',W,h'')$ are determined by the map $\Gl(1) \to T$, where $T\subset G_{d_\xi}$ is a maximal torus, and the same is true for the multiplicities of the $\big(S_k(Q_{\xi,f_\xi})\big)_{k\in K\cup\{\infty\}}$ with respect to $(V',\KK,h'_\xi)$ under the correspondence $E_k\leftrightarrow S_k(Q_{\xi,f_\xi})=S_k(Q_\xi)$ for $k\in K$ and $\KK=S_\infty(Q_f)\leftrightarrow S_\infty(Q_{\xi,f_\xi})=\KK$. Thus,  
\[ \hat{q}^{-1}(E)\cap X_{f,d} \subset \big(q_\xi^{-1}(0)\cap R^{0-ss}_{f_\xi,d_\xi}(Q_\xi)\big)\times_{G_{d_\xi}} G_d \cong R^{0,nilp}_{f_\xi,d_\xi}(Q_\xi)\times_{G_{d_\xi}} G_{d}. \]
Taking the quotient by $G_d$, we finally get 
\[ \pi_{f,d}^{-1}(E) \cong \big(\hat{q}^{-1}(E)\cap X_{f,d}\big)/\!\!/G_d \subset R^{0,nilp}_{f_\xi,d_\xi}(Q_\xi)/\!\!/G_{d_\xi} = \Msp^{nilp}_{f_\xi,d_\xi}(Q_\xi).\] 
\end{proof}

\bibliographystyle{plain}
\bibliography{Literatur_neu}

\vfill
\textsc{\small S. Meinhardt: Fachbereich C, Bergische Universit\"at Wuppertal, Gau{\ss}stra{\ss}e 20, 42119 Wuppertal, Germany}\\
\textit{\small E-mail address:} \texttt{\small meinhardt@uni-wuppertal.de}\\
\\


\end{document}